\documentclass[a4paper, 11pt]{preprint}
\usepackage{latexsym,amssymb, graphicx, amsthm, amsmath, mathtools}
\usepackage[cal=boondoxo]{mathalfa}
\usepackage[full]{textcomp}
\usepackage[osf]{newtxtext}
\usepackage{mhequ}
\usepackage{mathrsfs}
\usepackage{microtype}
\usepackage{tikz-cd}
\usepackage{orcidlink}
\usepackage{cprotect}
\usepackage{hyperref}

\DeclareSymbolFont{timesoperators}{T1}{ptm}{m}{n}
\SetSymbolFont{timesoperators}{bold}{T1}{ptm}{b}{n}

\newcommand{\eqdef}{\stackrel{\mbox{\rm\tiny def}}{=}}

\makeatletter
\renewcommand{\operator@font}{\mathgroup\symtimesoperators}
\makeatother

\newtheorem{theorem}{\bf  {Theorem}}
\newtheorem{proposition}[theorem]{\bf {Proposition}}
\newtheorem{lemma}[theorem]{\bf Lemma}
\newtheorem{definition}[theorem]{\bf Definition}
\newtheorem{corollary}[theorem]{\bf Corollary}
\newtheorem{remark}[theorem]{\bf Remark}

\newtheorem{assumption}[theorem]{\bf Assumption}
\numberwithin{theorem}{section}

\def\bN{\mathbb{N}}
\def\bR{\mathbb{R}}
\def\bZ{\mathbb{Z}}
\def\bE{\mathbb{E}}

\def\fe{\mathfrak{e}}
\def\fn{\mathfrak{n}}
\def\fo{\mathfrak{o}}

\def\ft{\mathfrak{t}}
\def\fl{\mathfrak{l}}
\def\fs{\mathfrak{s}}
\def\Lab{\mathfrak{L}}

\def\fH{\mathfrak{H}}
\def\fM{\mathfrak{M}}
\def\fR{\mathfrak{R}}

\def\Tpoly{\CT_{\mathrm{poly}}}
\def\bBPHZ{\overline{\text{BPHZ}}}

\def\ck{\mathcal{k}}
\def\CA{\mathcal{A}}
\def\CB{\mathcal{B}}
\def\CC{\mathcal{C}}
\def\CT{\mathcal{T}}
\def\CG{\mathcal{G}}
\def\CD{\mathcal{D}}
\def\CF{\mathcal{F}}

\def\CH{\mathcal{H}}
\def\CI{\mathcal{I}}
\def\CK{\mathcal{K}}
\def\cA{\mathcal{A}}
\def\cG{\mathcal{G}}
\def\cD{\mathcal{D}}
\def\cT{\mathcal{T}}
\def\cM{\mathcal{M}}
\def\CR{\mathcal{R}}
\def\cR{\mathcal{R}}
\def\cQ{\mathcal{Q}}
\def\cB{\mathcal{B}}

\def\CO{\mathcal{O}}
\def\CQ{\mathcal{Q}}
\def\CN{\mathcal{N}}
\def\CJ{\mathcal{J}}
\def\CM{\mathcal{M}}

\def\CY{\mathcal{Y}}
\def\CL{\mathcal{L}}

\def\pH#1{H_{#1;n}^{x, \eta}}
\def\pHa#1{H_{#1;n_1}^{x, \eta}}
\def\pHb#1{H_{#1;n_2}^{x, \eta}}
\def\pHao#1{H_{#1;n_1}^{0, \eta}}
\def\pHbo#1{H_{#1;n_2}^{0, \eta}}
\def\pG#1{G_{#1;n}^{x, \eta}}

\def\R{\mathbf{R}}
\let\phi\varphi
\def\scal#1{\langle #1\rangle}
\def\${|\!|\!|}

\let\eps\varepsilon
\def\ex{\mathrm{ex}}

\def\PPi{\boldsymbol{\Pi}}

\def\nice{\CC_w^\star}

\def\reg{\operatorname{scal}}

\def\supp{\operatorname{supp}}
\def\one{\mathbf{1}}

\def\gap{\operatorname{gap}}
\def\rand{\mathrm{rand}} 
\def\Mrand{\CM_{\rand}(\CT)}

\def\defeq{\eqdef}

\def\N{\mathbf{N}}
\let\f\frac

\def\dash{\leavevmode\unskip\kern0.18em--\penalty\exhyphenpenalty\kern0.18em}
\def\slash{\leavevmode\unskip\kern0.15em/\penalty\exhyphenpenalty\kern0.15em}

\makeatletter

\DeclareRobustCommand{\TitleEquation}[2]{\texorpdfstring{\StrLeft{\f@series}{1}[\@firstchar]$\if%
		b\@firstchar\boldsymbol{#1}\else#1\fi$}{#2}}

\makeatother

\begin{document}	
	\title{The BPHZ Theorem for Regularity Structures\\ via the Spectral Gap Inequality}
	\author{Martin Hairer$^{1,2}$ \orcidlink{0000-0002-2141-6561} and Rhys Steele$^1$ \orcidlink{0000-0001-6546-5800}}
	\institute{Imperial College London, London SW7 2AZ, United Kingdom \and École Polytechnique Fédérale de Lausanne, 1015 Lausanne, Switzerland \\
		\email{m.hairer@imperial.ac.uk, r.steele18@imperial.ac.uk}}
	
	\maketitle
	
	\begin{abstract}
		We provide a relatively compact proof of the BPHZ theorem for regularity structures of decorated trees 
		in the case where the driving noise satisfies a suitable spectral gap property, as
		in the Gaussian case. This is inspired by the recent work \cite{Felix} in the multi-index setting, but our proof
		relies crucially on a novel version of the reconstruction theorem for a space of ``pointed 
		Besov modelled distributions''. As a consequence, the analytical core of the proof is
		quite short and self-contained, which should make it easier to adapt the proof to 
		different contexts (such as the setting of discrete models).
	\end{abstract}
	
	\setcounter{tocdepth}{2}
	\tableofcontents
	
	\section{Introduction}
	
	Over the past decade, there has been substantial progress in the application of
	pathwise techniques to solution theories for singular stochastic PDEs. Besides
	the original theory of rough paths \cite{Terry,TerryBook,Book}, these include paracontrolled
	calculus \cite{GIP15} and regularity structures \cite{Hai14}. A common thread 
	is that the probabilistic
	aspects of a problem are encoded in the construction of a random element of
	some nonlinear space of distributions \slash functions (a ``model'' in the language 
	of \cite{Hai14}), while the solution theory itself
	is then purely deterministic, once a realisation of the model is fixed. 
	
	The BPHZ theorem obtained in \cite{ChHa16} is a crucial ingredient in this programme
	since it allows for the construction of random models in a wide variety of contexts.
	Unfortunately, while this result is extremely flexible, its proof is rather long and 
	difficult to follow. Recently, Otto and coauthors obtained a version of the BPHZ theorem
	in a specific context that, instead of the Feynman diagram techniques used in \cite{ChHa16},
	makes use of Malliavin calculus and spectral gap estimates \cite{Felix}. 
	(See also \cite{FelixAlg} for a description of the algebraic structure underpinning their approach.) 
	We refer the reader also to \cite{FG19} where a spectral gap based approach was used to obtain similar estimates in a more limited setting.
	
	In the present article, building on some of the analytic results from \cite{Cyril}, we extend and 
	simplify this technique to obtain
	a relatively short and self-contained proof of the BPHZ theorem for arbitrary regularity 
	structures of the type constructed in \cite{BHZ} and naturally appearing in the study of semilinear singular SPDEs. While it is restricted to noises satisfying a spectral gap condition
	(therefore mainly Gaussian noises), our conditions appear to be optimal in that context.
	
	Comparing to \cite{Felix}, the main advantages of our approach are:
	\begin{enumerate}
		\item We make no particular structural assumption on the class of stochastic PDEs that can in principle be covered by our result, except for local subcriticality. In particular, we cover 
		all the situations covered by the general machinery developed in \cite{BHZ,Ilya}.
		\item The analytic ``core'' of our proof is very short. This is especially the case when 
		all the driving noises are more regular than space-time white noise, in which case
		Section~\ref{section: Pointed Modelled Distributions} on pointed modelled distributions contains pretty much all of 
		the required analytic ingredients. This core generalises naturally to a variety of 
		slightly different setups, although the technicalities arising from the rest of the proof may require
		non-trivial adaptation.
		\item We obtain not only uniform bounds on suitable smoothened models but their convergence
		to a unique limit, which is a crucial ingredient when using these techniques to prove
		``weak universality'' type results. 
	\end{enumerate}
	
As this article was nearing completion, the preprint \cite{YvainIsmael} was released. There, the authors
recover the convergence of the model associated to the geometric stochastic
heat equation studied in \cite{BGHZ22}. This is achieved by combining the spectral gap inequality with the
algebraic structure described in \cite{Nadeem} and the analytic techniques developed in \cite{HQ18}.
These bounds are obtained by treating the specific trees appearing in this problem ``by hand'',
so it is not clear whether their approach could treat the level of generality considered here.
	
\subsection{Structure of proof}

We now give a short description of the overall structure of our proof, which is 
heavily inspired by \cite{Felix,FelixAlg}. This subsection assumes that the reader is somewhat familiar with 
the context of \cite{Hai14}. For notational simplicity, we assume here that there is only one
driving noise.

Since we are interested in a result applicable to singular stochastic PDEs, we consider regularity
structures with basis vectors given by certain labelled trees as in \cite{Hai14,BHZ}. We are then
interested in obtaining uniform in $\eps$ bounds on $\hat \Pi_x^\eps \tau$, where $\hat \Pi_x^\eps$
denotes the BPHZ lift (see \cite{ChHa16,BHZ}) of the driving noise $\xi$ to the ambient regularity
structure and $\tau$ denotes one of its basis vectors. 

Write then $T_k$ for the regularity structure generated by the set of basis vectors which 
include at most $k$ copies of the noise. The proof then goes by estimating the action
of $\hat \Pi_x^\eps$ on $T_k$ by induction over $k$. 

For $k=1$, one uses the spectral gap inequality and a Kolmogorov-type criterion to conclude that the 
driving noises $\xi_\ft$ do almost surely belong to a suitable Besov--Hölder space as desired.
For larger values of $k$, we remark that the derivative $D_h \hat \Pi_x^\eps \tau$ of $\hat \Pi_x^\eps \tau$
with respect to the noise $\xi$ in the direction $h$ is a sum of multilinear terms in $\xi$ 
of degree at most $k-1$. One can therefore hope to be able to find a modelled distribution
$H_x^{\tau,h}$ with values in $T_{k-1}$ (where we assume that we already have good bounds on
the renormalised model $\hat \Pi$) such that $D_h \hat \Pi_x^\eps \tau = \CR H_x^{\tau,h}$
with $\CR$ the reconstruction operator.
The spectral gap inequality then states that bounds on $\|(\CR H_x^{\tau,h})(\phi_x^\lambda)\|_{L^p(\Omega)}$ and
on $\bE (\hat \Pi_x^\eps \tau)(\phi_x^\lambda)$ yield bounds on $\|(\hat \Pi_x^\eps \tau)(\phi_x^\lambda)\|_{L^p(\Omega)}$.
Since one can hope to control $\bE \hat \Pi_x^\eps \tau$ as a consequence of the centring condition
defining the BPHZ lift, this then provides the desired induction step.

Here are some of the hurdles that need to be overcome along the way:
\begin{enumerate}
\item At the analytic level, one needs to find a suitable space of modelled distributions which on the 
one hand is large enough to contain the desired $H_x^{\tau,h}$ (which is itself relatively easy
to ``guess'' by induction) in a way that's uniform over $h$'s belong to the relevant space appearing
in the spectral gap condition (the Cameron--Martin space in the case of Gaussian noises), but on
the other hand is small enough so that the bounds on $\CR H_x^{\tau,h}$ match the bounds
appearing in the definition of a model. It turns out that this can only be achieved by looking
at a suitable family of $x$-dependent ``pointed'' Besov-type spaces of modelled distributions.
These are similar to the spaces appearing in \cite{Cyril,Josef}, but the required optimal bounds on their
reconstruction are much tighter than the bounds obtained there.
\item While it is straightforward to ``guess'' an inductive expression for $H_x^{\tau,h}$
such that the identity $D_h \Pi_x^\eps \tau = \CR H_x^{\tau,h}$ holds
for the canonical lift $\Pi$ (i.e.\ without any renormalisation), it is not obvious a priori
that the same expression also holds for renormalised models.
\item While the BPHZ model specifies that $\PPi \tau$ has vanishing expectation for basis vectors
$\tau$ of negative degree, it is not obvious in general how to use this to control 
$\bE \big(\hat \Pi_x^\eps \tau\big)(\phi_x^\lambda)$. Instead, we introduce a different
centring condition (which we call the $\bBPHZ$ model) which \textit{does}
allow to obtain such a control, and we then show a posteriori that control on the 
$\bBPHZ$ model implies control on the BPHZ model.
\end{enumerate}

\subsection{Article Structure}
	
	We begin this paper by gathering the relevant algebraic framework that forms the setting of our main result. In particular, Section~\ref{sec: set-up} is devoted to a high level description of the regularity structures and models appearing in \cite{BHZ}, alongside the renormalisation groups that act on them. Additionally, in that section we fix our probabilistic assumptions on the driving noise and state the main result of this paper, which is the content of Theorem~\ref{theo: BPHZ}.
	
	Section~\ref{sec: Modelled Distributions} is then devoted to the study of spaces of modelled distributions; especially the pointed modelled distributions described in the previous subsection. The tools developed in this section are not restricted to the regularity structures appearing in \cite{BHZ}, however their application in this setting forms the analytic core of our proof.
	In a final preparatory section, Section~\ref{sec: Frechet}, we identify the Fr\'echet 
derivative (with respect to the driving noises) 
of a renormalised model as the reconstruction of a particular family of modelled distributions of the type 
studied in Section~\ref{sec: Modelled Distributions}.

	With these tools in hand, the core of our proof is then contained in the remaining sections. In Section~\ref{sec: uniform bounds}, we derive uniform bounds on the $\bBPHZ$ model. In Section~\ref{sec: convergence}, we then show that this family converges as mollification is removed. Finally, in Section~\ref{sec: moving between models}, we then transfer our results 
to the context of the BPHZ model of \cite{BHZ}.
	
	In the appendices, for the reader's convenience, we gather some relatively elementary technical tools that are used throughout the paper.

	\subsection{Notation and Conventions}\label{s: notation}
	Throughout this article, we will consider an integer dimension $d \ge 1$ of space-time to be fixed. In addition, we will fix a d-dimensional space-time scaling $\fs$ which is nothing but a multi-index $\fs = (\fs_i)_{i=1}^d$ whose entries are positive real numbers. We write $|\fs| \defeq \sum_{i=1}^d \fs_i$. 
	
	Given such a scaling and $z = (z_i)_{i = 1}^d \in \bR^d$, we define 
	\begin{equs}
		|z|_\fs \defeq \sum_{i = 1}^d |z_i|^{1/\fs_i}.
	\end{equs}
	We note that $|\cdot|_\fs$ does not define a true norm on $\bR^d$ but it does induce a translation invariant metric. We will nonetheless often refer to this quantity as the ($\fs$-scaled) norm and will write $B_\fs(x,r)$ for the corresponding ball around $x$ of radius $r$.
	
	A choice of scaling also induces a scaled degree for multi-indices $k \in \bZ^d$ defined by setting
	\begin{equ}
		|k|_\fs \defeq \sum_{i=1}^d k_i \fs_i.
	\end{equ}
	When talking about the degree of a polynomial, we shall always mean the notion of degree for our scaling $\fs$ that is induced by the above definition.
	
	In addition, a choice of scaling defines an appropriate way of rescaling test functions that respects that scaling. Concretely, given a function $\psi: \bR^d \to \bR$, $\lambda \in (0,1]$ and $x \in \bR^d$, we define $\psi_x^\lambda(y) \defeq \lambda^{-|\fs|} \psi(\lambda^{- \fs} (y-x))$ where $\lambda^{-\fs} z \defeq (\lambda^{-\fs_i}z_i)_{i=1}^d$. In the case where $\lambda = 2^{-k}$, we will often write as a shorthand $\psi_x^k \defeq \psi_x^\lambda$ and it should always be clear from context which of these is meant.
	
	The spaces of distributions of interest to us will then be a scale of local Besov spaces that respect our given scaling. For $\alpha < 0$, $p, q \in [1, \infty]$ and $r > |\alpha|$, we let $\CB_{p,q}^\alpha$ be the space of distributions $\zeta$ on $\bR^d$ such that for each compact set $\ck \subseteq \bR^d$
	\begin{equs}
		\|\zeta\|_{\CB_{p,q}^\alpha; \ck} \eqdef \bigg \| \Big \| \sup_{\eta \in \CB_r} \lambda^{-\alpha} |\scal{\zeta, \eta_x^\lambda}| \Big \|_{L^p(\ck; dx)} \bigg \|_{L_\lambda^q}<\infty\;,
	\end{equs}
	where $L_\lambda^q = L^q((0,1); \lambda^{-1} d\lambda)$ and $\CB_r$ is the same distinguished set of test functions as appearing in \cite[Definition 3.7]{Hai14}. 
	For $\alpha \ge 0$, we let $\CB_{p,q}^\alpha$ be the space of distributions $\zeta$ on $\bR^d$ such that,
	for every compact set $\ck\subseteq \bR^d$, 
	\begin{equs}
		\|\zeta\|_{\CB_{p,q}^\alpha; \ck} \eqdef \Big \| \sup_{\eta \in \CB_r} |\scal{\zeta, \eta_x}| \Big \|_{L^p(\ck; dx)} + \bigg \| \Big \| \sup_{\eta \in \CB^{\lfloor \alpha \rfloor}_r} \frac{|\scal{\zeta, \eta_x^\lambda}|}{\lambda^\alpha} \Big \|_{L^p(\ck; dx)} \bigg \|_{L_\lambda^q}<\infty\;.
	\end{equs}	
	For the particular case where $p = q = \infty$, we write $\CC^\alpha \eqdef \CB_{\infty, \infty}^\alpha$.
	For an alternative characterisation of these spaces, we refer the reader to Appendix~\ref{appendix: Besov Spaces}.
	
	Finally, we let $H^{s} = H^s(\bR^d)$ be the $(L^2)$ Sobolev space of regularity $s$ which is nothing but the space of distributions $\zeta$ such that $\|\zeta\|_{\CB_{2,2}^s; \bR^d} < \infty$ with the topology induced by this choice of norm.
	
\subsection*{Acknowledgements}

{\small
MH gratefully acknowledges financial support from the Royal Society via a research professorship,
RP\textbackslash R1\textbackslash 191065. Likewise, RS gratefully acknowledges financial support from the EPSRC via the grant 	EP \textbackslash W522673 \textbackslash 1.
}

\subsection*{Data Availability Statement}
{ \small	Data sharing is not applicable to this article as no datasets were generated or analysed during the current study.}

	\section{Setting and Main Result}\label{sec: set-up}
	
Since the aim of this article is to obtain stochastic estimates for the BPHZ model defined in \cite{BHZ}, we
begin this section by briefly recalling the relevant algebraic set-up from that paper and providing references for the reader who is interested in more detail.
	
	\subsection{Regularity Structures of Decorated Trees}
	
	We begin by recalling the definition of a regularity structure given in \cite[Definition 2.1]{Hai14}.
	
	\begin{definition}\label{def: regularity structure}
		A regularity structure is a triple $(\CT, \CA, \CG)$ where
		\begin{enumerate}
			\item The index set $\CA \subset \bR$ is locally finite and bounded below.
			\item The model space $\CT = \bigoplus_{\alpha \in \CA} \CT_\alpha$ is an $\CA$-graded vector space
			such that each $\CT_\alpha$ is a Banach space. 
			\item The structure group $\CG$ acts linearly on $\CT$ such that for every $\Gamma \in \CG$ and $\tau \in \CT_\alpha$, $\Gamma \tau - \tau \in \CT_{< \alpha} \eqdef \bigoplus_{\beta < \alpha} \CT_\beta $$.
			$			\end{enumerate}

	\end{definition}
	
	\begin{assumption}\label{ass:poly_structure}
		In addition to the above definition, we will assume throughout that our regularity structures contain the polynomial structure in the sense that for $k \in \bN$, $\CT_k$ is isomorphic to $(\Tpoly)_k$ and that there exists a group morphism $\CG \to \cG_\mathrm{poly}$ such that the action of $\cG$ on $\bigoplus_{k \in \bN} \CT_k$ coincides with the pullback of the action of $\cG_\mathrm{poly}$ under this group morphism.
	\end{assumption}

	For our purposes, the main family of regularity structures of interest will be the reduced regularity structures described in \cite[Section 6.4]{BHZ} with model spaces consisting of linear combinations of decorated trees. 
	
	\subsubsection{Typed and Decorated Trees}\label{ss: typed forests}
	
	\begin{definition} \label{def: trees}
		A rooted tree is a connected acyclic graph $T = (N_T, E_T)$ with a distinguished vertex $\rho_T$ called the root.		
		A rooted forest is a graph in which each connected component is a rooted tree.
	\end{definition}
	
	We remark that a rooted tree has a natural partial order on vertices where $u \le v$ if and only if $u$ lies on the unique path from $v$ to the root. This extends in the obvious way to a partial order on rooted forests. This partial order induces maps $e \mapsto e_p$ and $e \mapsto e_c$ where $(e_p, e_c)$ are the pair of vertices incident to $e$ and $e_p < e_c$. In turn, this induces a partial order on edges by writing $e \le \tilde{e}$ if and only if $e_p \le \tilde{e}_p$.
	
	We then fix a finite set $\Lab$ of types which we assume comes equipped with a partition $\Lab = \Lab_+ \sqcup \Lab_-$ where we call elements of $\Lab_+$ kernel types and elements of $\Lab_-$ noise types. The purpose of this set is to allow us to associate to each edge in the tree an integral kernel or driving noise.

	\begin{definition}
		A degree assignment is a map $|\cdot|_\fs : \Lab \to \bR$ such that $|\ft|_\fs > 0$ for all $\ft \in \Lab_+$.
	\end{definition}
	Throughout this paper, we assume that a choice of degree assignment is fixed. One should think of a degree assignment as quantifying for $\ft \in \Lab_+$ the degree of regularisation of the corresponding integral kernel and for $\ft \in \Lab_-$ as quantifying the regularity (or lack thereof) of the corresponding noise.
	
	\begin{definition} \label{def: typed trees}
		A rooted tree $T$ is said to be typed if it is equipped with a map $\ft: E_T \to \Lab$ such that if $\ft(e) \in \Lab_-$ then $e$ is maximal in the partial order on $E_T$ and for distinct edges $e, e'$ with $\ft(e), \ft(e') \in \Lab_-$ one has that $e_p \neq e'_p$.
		Typed (rooted) forests are defined analogously.
	\end{definition}
	
	We will henceforth implicitly assume that all trees and forests appearing are typed and rooted, so we drop these qualifiers. We will call edges with labels in $\Lab_+$ kernel edges and edges with labels in $\Lab_-$ noise edges.
	
	\begin{definition}\label{def: decorated trees}
		A decorated forest is a forest $F$ equipped with a node label $\fn : N_F \to \bN^d$ and an edge label $\fe : E_F \to \bN^d$. Given a forest $F$ and decorations $\fn, \fe$ on $F$, we denote the corresponding decorated forest by $F_\fe^\fn$.
	\end{definition}
	
	A degree assignment on the type set $\Lab$ induces a corresponding notion for semi-decorated forests by setting
	\begin{equ}[e:defDeg]
		|F_\fe^\fn|_\fs \eqdef \sum_{e \in E_F} (|\ft(e)|_\fs - |\fe(e)|_\fs) + \sum_{u \in N(F)} |\fn(u)|_\fs\;,
	\end{equ}
	where $N(F) \eqdef N_F \setminus e_c\left [\{e \in E_F : \ft(e) \in \Lab_-\} \right ]$ is the set of nodes that will later correspond to integration variables when we interpret trees as functions.

	\begin{remark}
		In \cite{BHZ}, an additional decoration $\fo$ appears which tracks suitable information from contracted subtrees during the recursive renormalisation process. For the majority of this paper, we will work exclusively on the reduced regularity structure of \cite[Section 6.4]{BHZ} where this decoration vanishes and hence we will not include it. 
	\end{remark}

	\subsubsection{The Reduced and Extended Structures of \cite{BHZ}}\label{section: The Reduced and Extended Regularity Structures}
	
	It is unfortunately not the case that the vector space generated by the set $\fH$ of all decorated trees for a given set of types $\Lab$ forms a regularity structure. The reason is that whilst $\fH$ is naturally graded 
	by \eqref{e:defDeg}, the homogeneities arising in this way are typically neither locally finite, nor bounded from below.
	
	The usual solution to this problem is to restrict the set of trees that are allowed in the basis of the model space via a prescription for which combinations of edges can appear at any given vertex known as a rule \cite[Definition 5.8]{BHZ}.
	
	Throughout this paper we will assume that a so-called complete, normal, subcritical rule $R$ has been fixed. For precise definitions we refer the reader to \cite[Definitions 5.7, 5.14, 5.22]{BHZ}. For our purpose, the precise definition will not matter except insofar as it is required for the constructions of \cite{BHZ} to yield a regularity structure. Since the details of that construction are not important to us, we will satisfy ourselves here with a high-level description of their output which will be sufficient for the remainder of the paper. 
	
	The upshot of the constructions of \cite{BHZ} is that the regularity structures obtained in this way always come with a distinguished basis (elements of which are represented by decorated trees) as well as a partial product (defined by joining trees at their roots and setting the node decoration of the new root to be the sum of the
	node decorations of the original roots) and a finite number of abstract integration operators of the form $\CI_k^\fl$ for $k \in \bN^d$ and $\fl \in \Lab_+$ (obtained by grafting a tree onto a ``trunk'' consisting of
	an edge of type $\fl$ and edge label $k$).
	
	Every basis vector can then be obtained from the elementary ones (corresponding to trees consisting only of a single noise edge or of just the root) by repeated applications of the product and operations of the form $\CI_k^\ft$. The only purpose of the rule $R$ is to restrict the trees considered for inclusion as basis vectors of the regularity structure to avoid the problem mentioned for the naive approach.
	
	We will assume that the choice of degree on $\Lab$ and of the rule $R$ is such that the resulting regularity structure has finite dimensional components of each degree and for convenience we will assume that the choice of norm on each homogeneous component is such that the basis trees are orthonormal (so that in particular, the norm on each homogenous component comes from an inner product).
	
	For $k \in \bN^d$ we will often write $X^k$ for the tree consisting of a single vertex $\rho$ with label $\fn(\rho) = k$ and for $\ft \in \Lab_-, k \in \bN^d$ we will write $\Xi_\ft^k$ for the tree consisting of a single edge of type $\ft$ with with edge decoration $k$ and vanishing node label. We will often write $\Xi_\ft \eqdef \Xi_\ft^0$. We will also assume that for $\ft \in \Lab_+$ we have that $\CI_k^\ft X^k = 0$. Following the notation of Definition~\ref{def: regularity structure}, we write $\CT$ for the resulting model space and $B$ for the corresponding basis of trees.
	
	It remains to the describe the structure group associated with this regularity structure, since the details of the construction are important to ensure that various operations involved in our main argument respect the structure of our induction. We will give a description closer in spirit to that of \cite[Section 8]{Hai14}, however we note that the object constructed is ultimately the same as the one in \cite{BHZ} (see Section~6.4 there) and so only coincides with the precise description of \cite[Section 8]{Hai14} up to a change of basis.
	
	We write $\CT_+$ for the linear span of formal expressions of the type 
\begin{equ}[e:setTplus]
	\Big\{X^k \prod_i \CJ_{k_i}^{\ft_i} \tau_i : \tau_i \in B, \tau_i \neq X^k \text{ and }|\CJ^{\ft_i}_{k_i} \tau_i|_\fs > 0\Big\}\;,
\end{equ}	
where the product ranges over a finite index set and where $|\CJ^{\ft}_k \tau|_\fs \eqdef |\tau|_\fs + |\ft|_\fs - |k|_\fs$. We endow $\CT_+$ with its natural structure as a commutative algebra which is 
	suggested by the notation \eqref{e:setTplus}. 
	
	We then have a natural family of maps $\bar{\CJ}^\ft_k: \CT \to \CT_+$ given by setting for a semi-decorated tree $\tau$
	\begin{equ}
		\bar{\CJ}^\ft_k \tau =
		\begin{cases}
			\CJ^\ft_k \tau \text{ if } |\CJ^\ft_k \tau|_\fs > 0, \\ 0 \text{ otherwise}
		\end{cases}
	\end{equ}
	and extending linearly and multiplicatively.
	
	Since the meaning will be clear from context, we will usually denote this map also by $\CJ_k^\ft$.
	Additionally, we define a map $\Delta: \CT \to \CT \otimes \CT_+$ by setting
	\begin{equ}\label{eq: Delta def base}
		\Delta 1 = 1 \otimes 1, \quad \Delta X_i = X_i \otimes 1 + 1 \otimes X_i \;,
	\end{equ}
	as well as requiring that $\Delta$ is multiplicative and satisfies 
	\begin{equ}\label{eq: Delta def int}
		\Delta \CI_k^\ft \tau = (\CI_k^\ft \otimes 1) \Delta \tau + \sum_{l,m} \frac{X^l}{l!} \otimes \frac{X^m}{m!} \tilde{\CJ}_{k+l+m}^\ft \tau\;,
	\end{equ}
	where $\tilde{\CJ}_{k}^\ft \tau \eqdef \sum_{|m|_\fs < |\CJ_k^\ft \tau|_\fs} \frac{(-X)^m}{m!} \CJ_{k+m}^\ft \tau$. The structure group is then the character group of $\CT_+$ acting on $\CT$ via $\Gamma_f \tau = (1 \otimes f) \Delta \tau$.
	
\begin{remark}
The reason for introducing the operators $\tilde{\CJ}_{k}^\ft$ is to simplify the expressions for the 
characters in \eqref{eq: f_x neg hom int} 
without having to introduce the twisted antipode from \cite{BHZ}. This is also
closer to the original setup of \cite{Hai14}, see \cite[Section~6.4]{BHZ}.
\end{remark}
	
	We note that unlike in \cite[Section 8]{Hai14}, we have not provided a definition of $\Delta \Xi_\ft$ for $\ft \in \Lab_-$ as part of the base case given in \eqref{eq: Delta def base} of the recursive construction for $\Delta$. The reason for this is that we write $\Xi_\ft = \CI^\ft 1$ so that this case is covered by \eqref{eq: Delta def int}. In \cite{Hai14}, only the case of noises of negative degree is considered 
	in which case \eqref{eq: Delta def int} yields $\Delta \Xi_\ft = \Xi_\ft\otimes 1$ which coincides with the one given there. However, in the case where $|\ft|_\fs > 0$, additional terms appear which correspond to the need for ``positive renormalisation'' that will appear in the treatment of such symbols.
	
	As an addition to the construction of \cite{BHZ}, we will assume that the reduced regularity structure is truncated to include only symbols of degree at most $R$ for some $R > 0$ chosen to be sufficiently large. This is sufficient for all current analytic applications of the results of that paper and will be a key simplifying assumption since it implies that $\CA$ is finite.
	
	At times, we will refer also to the extended regularity structure of \cite{BHZ}. The main difference between this structure and the reduced structure is that the trees forming a basis of the model space $\CT^\mathrm{ex}$ of the extended structure have an additional decoration $\fo: N_T \to \bZ^d \oplus \bZ(\Lab)$ which is used to keep track of the information on components of the tree that get contracted during the recursive renormalisation procedure that will eventually be performed. The model space and structure group then admit an analogous description to the one given for the reduced structure. 
	
	Since a full description of these details is lengthy and unnecessary for the core of our argument, we refer the interested reader to \cite[Section 5]{BHZ}.

	\begin{assumption}\label{ass:rule}
		For the remainder of the paper, we will assume that each regularity structure appearing is an instance of (a truncation of) either the reduced or extended regularity structure for some complete, normal, subcritical rule.
	\end{assumption}

	For the reader not familiar with \cite{BHZ}, this should be treated as a technical assumption 
	required to use the constructions of that paper which is satisfied by all the rules arising
	naturally from subcritical stochastic PDEs. It guarantees on one hand that the structure is finite-dimensional
	at homogeneities below each level and on the other hand that the renormalisation group 
	acting on the structure is sufficiently rich.

	\subsection{Models and Renormalisation}
	
		The key analytic objects that we wish to study are models on the given regularity structure. Here, we write our definition in a form that is amenable to the treatment of renormalised models given in the case of the reduced\slash extended structures since it will later make introduction of notation more convenient.
	
	\begin{definition}\label{def: model}
		A smooth model on $\mathbb{R}^d$ with scaling $\fs$ on the reduced\slash extended structure $(\CT, \CA, \CG)$ consists of a pair of maps $(\Pi, \Gamma)$ where
		\begin{enumerate}
			\item There exist characters $(f_x)_{x \in \bR^d}$ on $\CT_+$ (or $\CT_+^\mathrm{ex}$ in the case of the extended structure) such that $\Gamma_{xy} = F_x^{-1} F_y$ where $F_x = (1 \otimes f_x) \Delta$,
			\item $\Pi: \mathbb{R}^d \times \CT \to \CC^\infty(\mathbb{R}^d)$ is such that $\Pi_x \Gamma_{xy} = \Pi_y$.
		\end{enumerate}
		We additionally assume that the following analytic bounds hold for all $\gamma > 0$ and compact sets $\ck \subseteq \mathbb{R}^d$:
		\begin{equs}
			\|\Pi\|_{\gamma; \ck} & \eqdef \sup_{\varphi \in \CB_r} \sup_{\alpha < \gamma} \sup_{x \in \ck} \sup_{\lambda \in (0,1]} \sup_{\tau \in \CT_\alpha} \lambda^{-\alpha} \|\tau\|^{-1} |\Pi_x \tau (\varphi_x^\lambda) | < \infty,
			\\
			\|\Gamma\|_{\gamma; \ck} & \eqdef \sup_{\beta < \alpha < \gamma} \sup_{x,y \in \ck} \sup_{\tau \in \CT_\alpha} \|x-y\|^{\beta-\alpha} \|\tau\|^{-1} \|\Gamma_{xy} \tau \|_\beta < \infty.
		\end{equs}
	\end{definition}
	
\begin{remark}\label{rem:topology}
These ``semi-norms'' yield a natural topology on the space of models. When showing convergence,
we will always restrict ourselves to some fixed but arbitrary $\gamma > 0$ and assume
then without loss of generality that our regularity structure is finite-dimensional.
\end{remark}	

	In addition to the above requirements, we will want to ensure that our models correctly encode the action of integration against the desired integration kernels. To this end, we fix a choice of kernel for each label $\fl \in \Lab_+$
	
	\begin{definition}
		A kernel assignment is a tuple $K = (K_\ft)_{\ft \in \Lab_+}$ such that for each $\ft \in \Lab_+$, $K_\ft$ satisfies \cite[Assumption 5.1]{Hai14} with $\beta = |\ft|_\fs$.
	\end{definition}

	\begin{definition}\label{def: admissibility}
			A model $Z = (\Pi, \Gamma)$ is admissible if
			\begin{enumerate}
				\item For $k \in \bN^d$, $\Pi_z X^k(\bar{z}) = (\bar{z} - z)^k$ and for $\tau \in \CT$ such that $X^k \tau \in \CT$, $\Pi_z X^k \tau = \Pi_zX^k \cdot \Pi_z \tau$.
				\item For $\ft \in \Lab_+$, $k \in \bN^d$ and $\tau \in \CT$ such that $\CI_k^\ft \tau \in \CT$, we have that
				$$\Pi_z \CI_k^\ft \tau (\bar{z}) = D^k K^\ft \ast \Pi_z \tau (\bar{z}) - \sum_{|j|_\fs < |\CJ_k^\ft \tau|_\fs} \frac{(\bar{z} - z)^j}{j!} D^{k+j} K^\ft \ast \Pi_z \tau(z).$$
			\end{enumerate}
	\end{definition}
	
    As is usual, we will without loss of generality assume that for an admissible smooth model $(\Pi, \Gamma)$, the characters $f_x$ satisfy 
\begin{equ}[e:deffx]
    f_x(X_i) = -x_i, \qquad f_x(\tilde{\CJ}_k^\ft \tau) = - (D^k K \ast \Pi_x \tau)(x)
\end{equ}
when $\ft \in \Lab_+$ is such that $|\ft|_\fs + |\tau|_\fs - |k|_\fs > 0$ and that when $\ft \in \Lab_-$ is such that $|\ft|_\fs - |k|_\fs > 0$ (which is possible since we allow noise edges of positive degree), we have that 
    \begin{equ}\label{eq: f_x neg hom int}
    	f_x (\tilde{\CJ}_k^\ft 1 ) = - D^{k} \xi_\ft (x)
    \end{equ}
	for some family of functions $\xi_\ft \in \CC^{|\ft|_\fs}$.

	It then follows that $\PPi = \Pi_x F_x^{-1}$ is independent of $x$ and we may equivalently represent the given model in the form $(\PPi, f)$. As such, we will in what follows interchange between these two descriptions without comment. 
	
	\begin{remark}\label{remark: positive degree noises}
		In general, there is no completely canonical way to associate a pair $(\PPi, f)$ to the pair $(\Pi, \Gamma)$. In the above we obtain a fixed definition only because we have made particular choices for how to define $f_x(X_i)$ and $f_x(\tilde{\CJ}_k^\ft 1)$, which correspond to insisting that $\PPi X_i = x_i$ and $\PPi \Xi_\ft = \xi_\ft$.
		
		However, for any constants $c_i$ and polynomials $\{P_\ft: \ft \in \Lab_-, |\ft|_\fs >0\}$ such that $P_\ft$ has degree less than $|\ft|_\fs$, there exists a choice of characters $\tilde{f}_x$ such that $\tilde{\PPi} \eqdef \Pi_x \tilde{F}_x^{-1}$ is such that $\tilde{\PPi} X_i = - x_i + c_i$ and $\tilde{\PPi} \Xi_\ft = \xi_\ft + P_\ft$, where $\tilde{F}_x$ is the map corresponding to $\tilde{f}_x$. We would then have $\tilde{\PPi} \tilde{F}_x = \Pi_x = \PPi F_x$ so that we see that our choice of $(\PPi, f)$ is far from canonical.
		
		However, in this paper our perspective is always that we start with fixed noises $\xi_\ft \in \CC^{|\ft|_\fs}$ and consider models constructed from them. This means that for us there will always be a well-defined choice of characters $f$ associated to a model $(\Pi, \Gamma)$ by \eqref{e:deffx}.
	\end{remark}
    
	We denote by $\CM_\infty(\CT)$ the space of all smooth and admissible models on $\CT$ and remark that the pseudometrics defined in Definition~\ref{def: model} induce a metric topology on $\CM_\infty(\CT)$. We denote the completion of $\CM_\infty(\CT)$ under this topology by $\CM_0(\CT)$ and note that it includes models which are distributional in nature, rather than function-valued.
	
	We will however be particularly interested in models that are constructed in a canonical way from a smooth driving noise and those that can be obtained by renormalisation of such a model. 
	
	\begin{definition}\label{def: smooth noise}
		A smooth noise is a tuple $\xi = (\xi_\ft)_{\ft \in \Lab_-}$ such that $\xi_\ft \in \CC^\infty(\bR^d)$ for every $\ft \in \Lab_-$. As in \cite{ChHa16}, we denote by $\Omega_\infty$ the collection of all smooth noises.
	\end{definition}
	
	\begin{definition}\label{def: canonical lift}
		The canonical lift $\CL(\xi)$ of a smooth driving noise $\xi$ is the unique admissible model such that
		\begin{enumerate}
			\item If $\tau_1, \tau_2 \in \CT$ are trees such that $\tau_1 \tau_2 \in \CT$, then $\Pi_z \tau_1 \tau_2 = \Pi_z \tau \cdot \Pi_z \tau_2$.
			\item For $\ft \in \Lab_-$, $\Pi_z \Xi_\ft = \xi_\ft - \sum_{|k|_\fs < |\ft|_\fs} \frac{D^k \xi_\ft(z)}{k!} (\cdot - z)^k$.
		\end{enumerate}
	\end{definition}
	
	As is by now well known, it is typically not the case that if $\xi_\varepsilon$ denotes a sequence of smooth approximations of a distributional driving noise $\xi$ then $\CL(\xi_\varepsilon)$ converges as $\varepsilon \to 0$. Instead, it is in general necessary to deform the definition of the product so as to produce a convergent sequence of models from these canonical lifts. This is the procedure of renormalisation of models that we will address in the next subsection.

	\subsubsection{The Renormalisation Group}\label{section: Renormalisation Group}
	
	Thus far in the regularity structures literature, there are two distinct objects that have been given the name ``renormalisation group'' for a given instance of the reduced \slash extended regularity structure. In this paper, we will have reason to work with both objects and thus in this section we will clarify the relation between them and prepare notation. 
	
	In particular, we identify nested groups $\CG_- \subset \fR$ described in \cite{BHZ,Hai14} respectively, together with their action on $\CM_\infty(\CT)$ that will encode the renormalisation that will later appear.
	
	The larger group $\fR$ is the renormalisation group constructed in \cite[Section 8]{Hai14}. Whilst the definition of this group entails fewer properties for its elements, we will have cause to work with an element of $\fR \setminus \CG_-$ later and thus will require it.
	
	\begin{definition}
		$\fR$ is the group of linear maps $M: \CT \to \CT$ such that $M$ commutes with the operators $\CJ_k^\ft$ for $\ft \in \Lab_+$, $M X^k = X^k$ and such that 
		$$\Delta^M \tau = \tau \otimes 1 + \sum \tau^{(1)} \otimes \tau^{(2)}$$ where $|\tau^{(1)}|_\fs > |\tau|_\fs$ and $\Delta^M$ is constructed in \cite[Proposition 8.36]{Hai14}. 
	\end{definition}
	
	\begin{remark}
		\cite{Hai14} also requires upper triangularity for another map $\hat{\Delta}^M$. However by \cite[Theorem B.1]{HQ18}, this condition actually follows automatically from the other conditions imposed in the construction. 
	\end{remark}
	
The action of $\fR$ on the space $\CM_\infty(\CT)$ is given by setting $M (\Pi, \Gamma) = (\Pi^M, \Gamma^M)$ where $\Gamma_{xy}^M = (F_x^M)^{-1} F_y^M$,
	$$\Pi_x^M \tau = (\Pi_x \otimes f_x) \Delta^M \tau, \qquad F_x^M \tau = (1 \otimes f_x \hat{M}) \Delta \tau\;,$$
	and $\hat{M}$ is constructed in \cite[Prop.~8.36]{Hai14}. 

	At times, the renormalisation group of \cite{Hai14} will be unwieldy for our purposes. At these times, we will restrict our statements to elements of the group $\CG_-$ described in \cite[Sec.~6.4]{BHZ}, which is a subgroup of $\fR$ by \cite[Thm~6.36]{BHZ}. 
	
	As mentioned previously, the construction of \cite{BHZ} also yields a regularity structure $\mathscr{T}^\mathrm{ex}$ whose model space $\CT^\mathrm{ex}$ consists of trees with an additional node decoration $\fo: N_T \to \bZ^d \oplus \bZ^d(\Lab)$. Additionally, \cite{BHZ} provides a Hopf algebra $\CT_-^\mathrm{ex}$ of trees of negative degree whose character group $\CG_-^\mathrm{ex}$ is the renormalisation group on the extended structure. An element $g \in \CG_-^\mathrm{ex}$ acts on models on the extended structure via precomposition with $M_g^\mathrm{ex} \eqdef (g \otimes 1) \Delta_\mathrm{ex}^-$. Here $\Delta_\mathrm{ex}^-: \CT^\mathrm{ex} \to \CT_-^\mathrm{ex} \otimes \CT^\mathrm{ex}$ is a co-action which (when ignoring the presence of decorations) 
extracts and contracts negative degree subtrees. The precise definition of the objects in this construction won't be important to us and so we refer the interested reader to \cite[Secs~5--6]{BHZ} for more details.
	
	We will actually only make use of the subgroup $\CG_- \subset \CG_-^\mathrm{ex}$ consisting of 
characters that are reduced in the sense that they do not depend on the additional decoration $\fo$.
The importance of $\CG_-$ stems from the fact that there is a canonical embedding $\CM_0(\CT) 
\hookrightarrow \CM_0(\CT^\ex)$ whose image is preserved under the action of $\CG_-$ on
$\CM_0(\CT^\ex)$. There, it furthermore coincides with the action of $\CG_-$ on $\CM_0(\CT)$
mentioned earlier, thus justifying the abuse of notation. (One has compatible canonical inclusions
$\CG_- \hookrightarrow \CG_-^\mathrm{ex} \hookrightarrow \fR^\ex$ as well as
$\CG_- \hookrightarrow \fR \hookrightarrow \fR^\ex$.)
In what follows, we will often without comment make use of these identifications.
	
	\subsubsection{Spaces of Random Models}
	
	We now turn to the description of the spaces of random models that will be under consideration for the remainder of this paper. 
	
	As in \cite{Felix}, our main assumption on the random driving noises appearing in the construction of models in this paper is a certain spectral gap inequality. In order to formulate this inequality for a collection of driving noises indexed by $\Lab_-$, we introduce some useful function spaces.
	
	\begin{definition}
		Given a map $s : \Lab_- \to \bR$, we set
		\begin{equ}
			H^s(\Lab_-) \eqdef \prod_{\ft \in \Lab_-} H^{s(\ft)}
		\end{equ}
		where the Sobolev space $H^s$ is the space of distributions $\zeta$ that satisfy $\|\zeta\|_{\CB_{2,2}^s; \bR^d} < \infty$ equipped with the obvious choice of norm.
		
		Since the product is finite, there is no ambiguity regarding the topology on this space. However, since we will want to characterise the resulting norm by duality later, we will equip the product space with the norm resulting from the $\ell^1$-product.
	\end{definition}
	
	Fix now a dimension $d$ and write $\CD'(\Lab_-) = \CD'(\bR^d)^{\Lab_-}$, with
	$\CD'(\bR^d)$ the space of Schwartz distributions on $\bR^d$. 
	Recall that a function $F \mathpunct{:}\CD'(\Lab_-) \to \bR$ is said to be 
	\textit{cylindrical} if there exist $k \in \bN$, $\allowbreak \psi_1, \dots, \psi_k \in \CD(\bR^d)$, $\ft_1, \dots, \ft_k \in \Lab_-$ and a smooth function $\bar{F}$ in $k$ real variables such that $F(\xi) = \bar{F}(\xi_{\ft_1}(\psi_1), \dots, \xi_{\ft_k}(\psi_k))$.
		For such an $F$, we define its \textit{gradient} by
		\begin{equ}
			\frac{\partial F}{\partial \xi}(\xi)_{\ft} = \sum_{i=1}^k \partial_i \bar{F} (\xi_{\ft_1}(\psi_1), \dots, \xi_{\ft_k}(\psi_k)) \,\psi_i\,\delta_{\ft_i, \ft}\;.
		\end{equ}
		This is interpreted as an element of $H^s(\Lab_-)$ for any choice of $s: \Lab_- \to \bR$.

	\begin{definition}\label{def: SGap}
		Given $s : \Lab_- \to \bR$, we say that a $\CD'(\Lab_-)$-valued random variable $\xi$ satisfies the spectral gap inequality for $s$ if the bound
		\begin{equ}\label{eq: SGap}
			\bE \left [ |F(\xi)|^2 \right ]^{1/2} \le C_2 \Bigg (|\bE \left [ F(\xi) \right ]| + \bE \left [ \left \| \frac{\partial F}{\partial \xi} \right \|_{H^s(\Lab_-)}^2 \right ]^{1/2} \Bigg)
		\end{equ}
holds uniformly over all cylindrical functions $F$.
	\end{definition}
\begin{remark}\label{remark: SGap upgrade}
	By applying the spectral gap inequality to $F^2$ for $F$ cylindrical, one concludes from it that, for $p=4$,
	there exists a constant $C_4$ such that the bound
	\begin{equ}\label{eq: pSGap}
		\bE \left [ |F(\xi)|^{p} \right ]^{1/p} \le C_p \Bigg (|\bE \left [ F(\xi) \right ]| + \bE \left [ \left \| \frac{\partial F}{\partial \xi} \right \|_{H^s(\Lab_-)}^p \right ]^{1/p} \Bigg)
	\end{equ}
	holds uniformly over smooth cylindrical functions with $p = 4$. Iterating this argument, one obtains the same inequality for all $p$ of the form $2^k$ with $k \in \bZ_+$. (See for example \cite[Lem.~2]{GNO15} for a similar argument.)
\end{remark}
\begin{remark}
	In the Gaussian setting, it is well known that the spectral gap inequality holds so long as the Sobolev norm appearing on the right hand side is the Cameron--Martin norm of the corresponding Gaussian measure (see \cite[Theorem 5.5.1]{Bogachev}).
	
	We will often require a slightly different choice of norm (as dictated by Assumption~\ref{ass:alg}), however it will typically be the case that our choice of norm will be stronger than the corresponding Cameron--Martin norm so that our assumption is weaker than the usual condition for Gaussian measures. For example, in the case of space-time white noise, the Cameron--Martin norm is nothing other than the $L^2$-norm. However, the corresponding part of Assumption~\ref{ass:alg} corresponds to requiring $s$ to be a small but positive real number in this case. 
	
	This is done to avoid working with Besov spaces of integer regularity, which frequently exhibit different behaviour to their non-integer regularity counterparts.
\end{remark}
	
It is unfortunately not the case that, for a given model $(\Pi, \Gamma)$ on a regularity structure $\CT$, the
random variables $\Pi_x \tau (\psi)$ are cylindrical as a function of $\xi$, where $\tau \in \CT$ and $\psi \in \cD(\bR^d)$. Nonetheless, we will throughout apply the spectral gap inequality to such functions of the noise without further ado. This can always be justified by an application of the following technical lemma, for which we only provide a sketch proof since it is not the main conceptual novelty of our approach.

In order to state the lemma, we first fix some notation. We fix the weight function 
$w(x) = (1+|x|^2)^{|\fs|/2+1}$ on $\bR^d$ and write $\CC^{-L}_w(\bR^d)$\label{defCw}
for the corresponding weighted $\CC^{-L}$ space, namely distributions in $\CC^{-L}_w$ are such that
\begin{equ}
	|\zeta(\psi_x^\lambda)| \le C w(x) \lambda^{-L}\;,
\end{equ}
uniformly over $x \in \bR^d$, $\psi \in \CB^r$ and $\lambda \in (0,1]$.
These weighted spaces will not be important in the rest of this paper, and appear only to avoid the 
need for some ad hoc notion of derivative on a Fréchet space.

\begin{lemma}
Let $s$ be as in Definition~\ref{def: SGap} and let
\begin{equ}
\CC_w \eqdef \prod_{\ft \in \Lab_-} \CC^{- s(\ft) - {|\fs|\over 2} - \kappa}_w(\bR^d)\;,
\end{equ}
for some (fixed but arbitrary) $\kappa > 0$.
Suppose that $\xi$ is a centred $\CD'(\Lab_-)$-valued random variable satisfying the 
corresponding spectral gap inequality \eqref{eq: SGap}.
	
Then $\xi$ admits a version taking values in $\CC_w$ with moments of all orders. Let furthermore
$
F\colon \CC_w \to \bR
$
be continuously Fr\'echet differentiable 
such that both $F$ and its Fr\'echet derivative are of polynomial growth.
Then $F$ also satisfies \eqref{eq: SGap}.
\end{lemma}
\begin{proof}[Sketch of proof]
Let $\psi$ be a smooth compactly supported test function. 
By applying \eqref{eq: pSGap} to $F[\xi] = \xi_\ft(\psi_x^\lambda)$, we obtain for each $p$ of the form $2^k$
\begin{equ}
\|\xi_\ft(\psi_x^\lambda)\|_{L^p} \lesssim \|\psi_x^\lambda\|_{H^{s(\ft)}(\Lab_-)} \lesssim \lambda^{- s(\ft) - \f{|\fs|}2}\;,
\end{equ}
where the last bound follows from the scaling properties of $H^{s(\ft)}$ and the implicit constant is allowed to depend on $p$.

As a result, the fact that $\xi$ admits a version in $\CC_w$ with moments of all orders then follows from a simple variant of 
Kolmogorov's continuity test in the form of Theorem~\ref{theo: Kolmogorov Criterion}, using the fact that $\sum_{x \in \bZ^d} 1/w(x) <\infty$.

	For $\zeta \in \CC^{-L}_w(\bR^d)$, we define 
	\begin{equ}[e:approx]
		\zeta^\varepsilon \eqdef \sum_{x \in \Lambda_\varepsilon \cap B(\eps^{-1})} \frac{1}{\|\theta_x^\varepsilon\|_{L^1}} \zeta(\theta_x^\varepsilon) \theta_x^\varepsilon
	\end{equ}
	where $\{\theta_x^\varepsilon: x \in \Lambda_\varepsilon \eqdef (\varepsilon \bZ)^d\}$ is a smooth partition of unity such that $\theta_x^\varepsilon$ is supported in $B_\fs(x, c\varepsilon)$ for a positive constant $c$ chosen such that the resulting $c\eps$-balls cover the full space. One then checks that $\zeta^\eps \to \zeta$ in $\CC_w$ as $\eps \to 0$.
	
	It then follows from continuous Fr\'echet differentiability of $F$, that if $\xi^\eps \eqdef  (\xi_\ft^\eps)_{\ft \in \Lab_-}$ then $F[\xi^\eps] \to F[\xi]$ in $\bR$ and $dF[\xi^\eps] \to dF[\xi]$ in $H^s(\Lab_-)$ (where in the latter case, we make use of the continuous embedding $(\CC_w)^* \hookrightarrow H^s(\Lab_-)$). Therefore, since $\scal{f, g^\eps} = \scal{f^\eps, g}$, we can write $dF^\eps[\xi](\eta) = dF[\xi^\eps][\eta^\eps] = \scal{dF[\xi^\eps]^\eps, \eta}$ where we interpret $dF[\xi^\eps]$ as an element of $H^s(\Lab_-)$ to conclude that for fixed $\xi \in \CC_w$, $dF^\eps[\xi] \to dF[\xi]$ in $H^s(\Lab_-)$-norm.
		The result then follows by an application of the dominated convergence theorem.
\end{proof}
	
	We will now assume that a map $\reg : \Lab_- \to \bR$ is fixed for the remainder of this paper. When we say that an $\Omega_\infty$-valued random variable satisfies the spectral gap inequality without specifying $s$ we will always mean that it satisfies the spectral gap inequality for $s = - \reg$. 
Here the choice of sign is to simplify notation in our later arguments where we will treat the norm in the derivative term by duality so that $\reg$ corresponds to the regularity of the space we end up working with. 

The map $\reg$ captures the amount by which the scaling of our driving noise differs from that of white noise,
modulo some small perturbation which we will make for purely technical reasons in order to avoid 
integer regularities in our Besov spaces.
	
	\begin{definition}\label{def:smoothNoise}
		We let $\CM(\Omega_\infty)$ denote the set of $\Omega_\infty$-valued random variables $\xi$ such that
		\begin{enumerate}
			\item The law of $\xi$ is invariant under the natural action of $\bR^d$ on $\Omega_\infty$ given by translations.
			\item $\xi$ satisfies the spectral gap inequality with $s = -\reg$. 
			\item For every $x \in \bR^d$, $k \in \bN^d$ and $\ft \in \Lab_-$, $D^k \xi_\ft(x)$ has finite moments of all orders and has vanishing first moment.
		\end{enumerate}
	\end{definition}
	\begin{remark}
		This definition is more restrictive than the one found in \cite[Definition 2.13]{ChHa16}. The difference is that that paper aims to estimate norms of models via moment-cumulant techniques and thus later makes assumptions on certain norms that control cumulants of the noise. We will instead replace this control on cumulants with the spectral gap assumption and thus we encode that at the level of the definition here.
	\end{remark}
	
	The noises we are interested in do not satisfy Definition~\ref{def:smoothNoise} due to their lack of 
	regularity. However, they will satisfy the following, which implies that their convolutions with a mollifier satisfies Definition~\ref{def:smoothNoise}.
	\begin{definition}\label{def: SGap noises}
		We let $\CM(\Omega_0)$ be the set of all $\CD'(\Lab_-)$-valued random variables $\xi$ 
		that are centred and that satisfy the analogues of points 1 and 2 in Definition~\ref{def:smoothNoise}.
	\end{definition}

	Given $\xi \in \cM(\Omega_\infty)$, applying the action of the renormalisation group $\cG_-$ to the canonical lift $\CL(\xi)$ yields a collection of random models. We will now identify several spaces of random models that contain such models that will be required in the sequel. 
	
	\begin{definition}\label{def:stationary}
		We denote by $\bar{\CM}_{\rand}(\CT)$ the space of all $\fM_0(\CT)$-valued random variables that are stationary in the sense that there exists an action $\tau$ of $\bR^d$ on the underlying probability space by measure preserving maps such that for $x,y, h \in \bR^d$, the identities
		\begin{equ}\label{eq: translation invariance}
			T_h \Pi_{x+h}(\omega) = \Pi_x(\tau_h \omega), \qquad \Gamma_{x+h, y+h}(\omega) = \Gamma_{x,y}(\tau_h \omega)
		\end{equ}
		hold almost surely, where $T_h \Pi_x \tau (y) = \Pi_x \tau (y+h)$.
		
		We let $\Mrand$ denote the subspace of $\bar{\CM}_{\rand}(\CT)$ consisting of random models $Z$ such that there exists $M \in \CG_-$ and $\xi \in \CM(\Omega_\infty)$ such that $Z = M \CL(\xi)$. 
	\end{definition}
	
	\begin{remark}
		It is straightforward to see from \cite[Equation 6.20]{BHZ} that all models constructed from the lift of a stationary noise via the action of $\CG_-$ are stationary and hence lie in $\Mrand$.
	\end{remark}
	
	In what follows, we will assume that all random models under consideration are stationary in the sense of Definition~\ref{def:stationary}.	
	As mentioned previously, our aim is to consider random models constructed via renormalisation of the canonical lift of mollifications of $\xi \in \CM(\Omega_0)$. Given $\xi \in \Omega_0$ and a sequence $M^n \in \CG_-$, we obtain a family of models $Z^n = (\Pi^n, \Gamma^n)$ in $\Mrand$ defined by setting $Z^n = M_n \CL(\xi_n)$ where $\xi_n = \varrho^n \ast \xi$ for some mollifier $\varrho$.
	
	Our goal in this paper is then to show convergence for the sequence of models given by this construction for a distinguished sequence of elements $M_n$ known as the \textit{BPHZ renormalisation} which is constructed in \cite{BHZ}.
	
	For brevity, we will not delve into the precise details of the construction of \cite{BHZ} here. Rather, we will satisfy ourselves with a description of its defining properties. 
	\begin{definition}\label{def: BPHZ}
		For $\xi \in \CM(\Omega_\infty)$, the BPHZ renormalised model corresponding to $\xi$ is the unique model $(\hat{\Pi}, \hat{\Gamma})$ in $\CG_-(\CL(\xi))$ such that for each tree $\tau \in \CT$ of negative degree one has that
		$$\bE[\hat{\PPi} \tau(h)] = 0$$
		where $\hat{\PPi} = \hat\Pi_x \,\hat F_x^{-1}$.
		
		Given a noise $\xi \in \CM(\Omega_0)$ and a mollifier $\varrho$, we will often write $\hat{Z}^n = (\hat{\Pi}^n, \hat{\Gamma}^n)$ for the BPHZ renormalised model corresponding to $\xi_n = \varrho^n \ast \xi$.
	\end{definition}
	Existence and uniqueness of a renormalised model with these properties follows from the results of \cite[Section 6]{BHZ} and in particular Theorem~6.18 there. For an explicit description of the action of $\hat{\Pi}^n$ on trees, we refer the reader to \cite[Equations 3.5 and 4.27]{ChHa16}.
	
	In fact, we will proceed by first considering a slight variant of the BPHZ renormalisation which we call the $\bBPHZ$ renormalisation that is better adapted to our approach for estimating the first term on the right hand side of the spectral gap inequality.
	
	\begin{definition}\label{def: bBPHZ}
		For $\xi \in \CM(\Omega_\infty)$, the $\bBPHZ$ renormalised model corresponding to $\xi$ is the unique model $(\bar{\Pi}, \bar{\Gamma})$ in $\CG_-(\CL(\xi))$ such that for each tree $\tau \in \CT$ of negative degree one has that
		\begin{equ}\label{eq: renorm recentering}
			\bE[\bar{\Pi}_0 \tau(\phi)] = 0
		\end{equ}
		where $\phi$ is as in Definition~\ref{def: semigroup kernel}.
		
		Given a noise $\xi \in \CM(\Omega_0)$ and a mollifier $\varrho$, we will often write $\bar{Z}^n = (\bar{\Pi}^n, \bar{\Gamma}^n)$ for the $\bBPHZ$ renormalised model corresponding to $\xi_n = \varrho^n \ast \xi$.
	\end{definition}

	Since we will only ever consider models that are stationary in the sense of Definition~\ref{def:stationary} the arbitrary choice of $0$ as a distinguished base point appearing in the above definition is not of particular importance.
	
	Existence and uniqueness of such a model follows from straightforward adaptation of the proof of \cite[Theorem 6.18]{BHZ} by replacing their character $g^-(\PPi)$ with the character $g^-(\Pi) = \bE[\Pi_0 \tau(\phi)]$. The specific construction will not be of particular importance outside of the property \eqref{eq: renorm recentering} and as a result we once again avoid introducing the lengthy notation required to discuss the constructions of \cite{BHZ} in detail.
	
	\subsection{Assumptions and Main Results}
	
	Our main goal in this paper is to show that if we fix a driving noise $\xi$ satisfying the spectral gap inequality and let $\xi_n = \varrho^n \ast \xi$ then $M_n^{\mathrm{BPHZ}} \CL(\xi_n)$ is a convergent sequence of models. We split this task into two parts. We begin by establishing uniform bounds on this sequence of models and only once these bounds are established do we turn to the convergence result. This division provides an opportunity to demonstrate the core of our argument without the technicalities required for convergence and also allows us to simplify the exposition of the proof of convergence. 
	
	In order to obtain these results, we will apply a Kolmogorov criterion for models (a statement and proof of which can be found in Appendix~\ref{appendix: Kolmogorov}) which reduces the task of obtaining uniform bounds to the task to verifying that the assumption \eqref{eq: Kolmogorov single point ass} is satisfied. Similarly, the task of obtaining convergence is reduced to showing that for each $\varepsilon > 0$ there exists an $N \in \bN$ such that for $n, m \ge N$, the pair of models $M_n^{\mathrm{BPHZ}} \CL(\xi_n), M_m^{\mathrm{BPHZ}} \CL(\xi_m)$ satisfy assumption \eqref{eq: Kolmogorov multi point ass} uniformly in $n,m$.
	
	Whilst such a result is more narrow than the estimates obtained in \cite{ChHa16}, it does cover their main application. 
	
	However in order to formulate such a result, we will need to make some restrictions on the trees that appear in the reduced regularity structure. We assume that
	
	\begin{assumption}\label{ass:alg}
		For every $\ft \in \Lab_-$, we assume that $-|\fs| < |\ft|_\fs < \reg \ft - |\fs|/2$ 
		and that $|\ft|_\fs + |\fs|/2$ and $\reg \ft$ have the same sign.
		
		We  also assume that if $n_\tau$ denotes the number of noise edges in $\tau$ then $|\tau|_\fs = |\bar{\tau}|_\fs$ implies that $n_\tau = n_{\bar \tau}$ and that, for every tree $\tau \in \CT$ containing 
		at least two edges, we have
		$|\tau|_\fs > - \frac{|\fs|}{2}$.
	\end{assumption}
		
	\begin{remark}
		These restrictions should be compared with the corresponding \cite[Assumption 2.24 and Definition 2.28]{ChHa16}. The first and last of our assumptions do appear there also, however our second assumption which is adapted to the spectral gap inequality, is replaced by a more complex assumption which ensures that non-renormalisable cumulants aren't formed when taking expectations of the resulting Feynman diagrams. 
		
		Since the assumptions of \cite{ChHa16} on the trees appearing in the reduced regularity structure are not automatic from the constructions of \cite{BHZ}, applications of that framework require checking these assumptions. This is significantly simpler with the assumptions on trees made here since our only assumptions that aren't posed at the level of the labels themselves are relatively straightforward ones.
	\end{remark}
	
	Additionally, we will make infinitesimal losses in regularity at various stages in our (inductive) proof. This is typically accounted for by the fact that the natural value for the degree of the noise under consideration is an ``infinitesimal amount'' less than some fixed value. Since this ``infinitesimal amount'' doesn't have a natural fixed value, we can alter it throughout the argument to allow ourselves the wiggle room to accommodate these losses.
	
	This phenomenon already occurs in \cite{ChHa16} and is accounted for by the supposition that there is a second degree assignment on the structure that does not alter its core algebraic construction but does assign every negative degree tree a degree that is at least $\kappa > 0$ lower than the original assignment. 
	
	In that paper, this assumption is only needed for one further degree assignment since estimates are established for each tree individually and the wiggle room is only needed to ensure that the estimates are good enough to apply a Kolmogorov criterion of the same type as Theorem~\ref{theo: Kolmogorov Criterion}. 
	
	However, we will need to make use of the wiggle room more often since our proof is inductive in the number of edges so that the bounds on the model at the previous stage are fed as input into the argument that the bounds at the current stage hold. At each stage then, we will want to apply Theorem~\ref{theo: Kolmogorov Criterion} to convert our bounds into control on the norm of the model and this means that at each stage we will make a loss of the same type as in \cite{ChHa16}. As a result, we need a sequence of degree assignments, each of which encodes some amount of wiggle room. 
	
	We define degree assignments for $k \in \bN$ by
	\begin{equs}\label{eq: shifted hom assignments}
		|\tau|_\fs^{(k)} \eqdef |\tau|_\fs + k n_\tau \kappa
	\end{equs}
	where $\kappa> 0$ is a sufficiently small fixed constant. Note that the case $k = 0$ corresponds to our original degree assignment on the reduced regularity structure.
	In particular, it is a consequence of Assumptions~\ref{ass:rule} and~\ref{ass:alg} that for each $N \in \bN$ there exists a $\kappa > 0$ (sufficiently small) such that for each $k \le N$
	\begin{enumerate}
		\item Assumption~\ref{ass:alg} is also satisfied with $|\cdot|_\fs$ replaced by $|\cdot|_\fs^{(k)}$.
		\item The orders on trees induced by $|\cdot|_\fs$ and $|\cdot|_\fs^{(k)}$ coincide.
		\item The rule $R$ generating our regularity structure is also complete and subcritical with respect to
		the degree assignment $|\cdot|_\fs^{(k)}$.
	\end{enumerate}  
		
	With these assumptions the place, the main result of this paper can then be stated as follows.
	
	\begin{theorem}\label{theo: BPHZ}
		Suppose that $\mathscr{T} = (\CT, \CA, \CG)$ is an instance of the reduced regularity structure satisfying Assumptions~\ref{ass:rule} and~\ref{ass:alg} and that $\xi \in \CM(\Omega_0)$. 
		Then, for any mollifier $\varrho$, if $M_n^\mathrm{BPHZ} \in \fR$ denotes the BPHZ choice of renormalisation corresponding to $\varrho^n \ast \xi$, then the sequence of models $\hat Z^n \eqdef M_n^\mathrm{BPHZ} \CL(\varrho^n \ast \xi)$ converges in probability as $n \to \infty$, for the topology on the space of admissible models
		discussed in Remark~\ref{rem:topology}.
	\end{theorem}
	\begin{remark}
		It also follows from a minor modification of our proof that the limit is independent of the choice of mollifier $\varrho$. This is because, given a second mollifier $\tilde{\varrho}$, all of our estimates can also be applied to the sequence of models $$\tilde{Z}^n = \begin{cases}
			M_n^\mathrm{BPHZ} \CL(\varrho^n \ast \xi), \quad n \text{ even}
			\\
			M_n^\mathrm{BPHZ} \CL(\tilde{\varrho}^n \ast \xi), \quad n \text{ odd}			
		\end{cases}
		$$
		to obtain convergence of that sequence of models.
	\end{remark}
	
	\section{Spaces of Modelled Distributions}\label{sec: Modelled Distributions}
	
	We now turn to establishing the core analytic tools that will be used in our proof of Theorem~\ref{theo: BPHZ}. Our eventual goal is to describe the Fr\'echet derivative of a model in terms of certain modelled distributions that have properties that are sufficient to capture the core analytic information for our proof. To that end, we introduce a series of results regarding various spaces of modelled distributions in this section.
	
	\subsection{Besov Modelled Distributions}
	
	In this subsection, we recall the definitions of Besov modelled distributions $\cD^\gamma_{p,q}$ 
	from \cite{Cyril}. We remark that the definitions and subsequent results will differ very slightly from \cite{Cyril} since we will work with local rather than global bounds in the definitions of all of our objects. 
	(For conciseness, we will not use the convention of adding `loc' to the names of the local spaces since all the spaces appearing here are local.) Despite this, we will not include proofs of the direct analogues of the results of \cite{Cyril} in this paper since they follow from the precisely the same techniques: one only needs to keep track of the domain on which bounds are required to hold.

	\begin{definition}\label{def: Dgamma}
		For $\gamma \in \bR$ and a model $Z = (\Pi, \Gamma)$, let $\cD^\gamma_{p,q} = \cD^\gamma_{p,q}(Z)$ be the Fr\'echet space of all measurable maps $f: \bR^d \rightarrow \cT_{<\gamma}$ such that, for all $\zeta\in\cA_\gamma$ and for all compact sets $\ck \subseteq \bR^d$ , we have:\begin{enumerate}
			\item Local bound:
			\begin{equ}
				\big\| \big| f(x) \big|_\zeta \big\|_{L^p(\ck; dx)} < \infty\;,
			\end{equ}
			\item Translation bound:
			\begin{equ}
				\int_{B(0,1)} \bigg\| \frac{\big| f(x+h)-\Gamma_{x+h,x} f(x) \big|_\zeta}{|h|_\fs^{\gamma-\zeta}} \bigg\|_{L^p(\ck; dx)}^q \frac{dh}{|h|_{\fs}^{|\fs|}}< \infty\;.
			\end{equ}
		\end{enumerate}
		We write $\$f\$_{p, q, \gamma; \ck}$ for the corresponding family of seminorms. Since we often restrict to the case $q = \infty$, we will write $\CD_p^\gamma \eqdef \CD_{p,\infty}^\gamma$ and $\$f\$_{p, \gamma; \ck} \eqdef \$f\$_{p, \infty, \gamma; \ck}$.
		
		Additionally, given a second model $\bar{Z} = (\bar{\Pi}, \bar{\Gamma})$, $f \in \cD_{p, q}^\gamma(Z)$ and $\bar{f} \in \cD_{p,q}^\gamma(\bar Z)$ we define
		\begin{equs}
			 \$ f, \bar{f} \$_{p, q, \gamma; \ck} =& \sup_{\zeta < \gamma} \| |f(x) - \bar{f}(x)|_\zeta \|_{L^p(\ck; dx)} \\ &+ \sup_{\zeta < \gamma} \left ( \int_{B(0,1)} \left \| \frac{\big| \Delta_h f(x) - \bar{\Delta}_h \bar{f}(x) \big|_\zeta}{|h|_\fs^{\gamma-\zeta}} \right \|_{L^p(\ck; dx)}^q \frac{dh}{|h|_\fs^{|\fs|}} \right )^{1/q}
		\end{equs}
		where $\Delta_h f(x) \defeq f(x+h) - \Gamma_{x+h, x} f(x)$ and $\bar{\Delta}_h \bar{f}(x) \defeq \bar{f}(x+h) - \bar{\Gamma}_{x+h, x} \bar{f}(x)$. 
	\end{definition}
	
	\begin{remark}\label{remark: projection remark}
		As usual, there is some ambiguity as to the choice of representative of an element $\CD_{p, q}^\gamma$. In the above definition we have chosen a unique representative by insisting that these elements take values in $\CT_{< \gamma}$, however it is often natural to say that functions $f: \bR^d \to \CT$ that contain terms of degree $\gamma$ or higher are elements of this space. When we do so, it should be understood that we mean that $\CQ_{< \gamma} f \in \CD_{p, q}^\gamma$ in the sense of the above definition.
		
		It is easy to check that if $\bar{\gamma} > \gamma$ and $f: \bR^d \to \CT_{\bar{\gamma}}$ satisfies the local bounds for all $\zeta < \bar{\gamma}$ and the translation bound (with no projection $\CQ_{< \gamma}$ present) for all $\zeta < \gamma$ then, provided that $q = \infty$ or $\gamma \not \in \CA$, one has $\CQ_\gamma f \in \CD_{p, q}^\gamma$ thus justifying our convention. 
		
		This observation leads us to make the following assumption.
	\end{remark}

	\begin{assumption}	We will make it a standing assumption that when we consider the case $q < \infty$, we have that $\gamma \not \in \CA$, and will often use this observation without comment in what follows.
	\end{assumption}

	The main result of \cite{Cyril} is the following version of the reconstruction theorem (here stated in its local form). We remark that the proof given in \cite{Cyril} relies on tools from wavelet analysis and so makes a slightly more restrictive assumption on the choice of scaling than we have done. A wavelet free proof of a more general result is provided in \cite[Theorem 3.2]{BL21} which is certainly sufficient for our purposes. 
	
	\begin{theorem}\label{theo: Reconstruction}
		Let $(\cT,\CA, \cG)$ be a regularity structure and $Z = (\Pi,\Gamma)$ be a model. Let $\gamma \in \R_+\backslash\N$, and set $\alpha=\min(\cA\backslash\N) \wedge \gamma$. If $q = \infty$, let $\bar{\alpha} = \alpha$. Otherwise suppose that $\bar{\alpha} < \alpha$. 
		
		Then, for $\gamma > 0$, there exists a unique continuous linear map $\cR:\cD^\gamma_{p,q}(Z) \rightarrow\cB^{\bar \alpha}_{p,q}$ such that
		\begin{equs}\label{Eq:BoundRecons}
			\Bigg \| 2^{n \gamma} \Bigg \| \sup_{\eta \in \CB^r} \left |\langle \cR f - \Pi_x f(x), \eta_x^n \rangle \right | \Bigg \|_{L^p(\ck; dx)} & \Bigg \|_{\ell^q(n)} \\& \lesssim \$ f\$_{p, q, \gamma; \bar{\ck}} \|\Pi\|_{\gamma; \bar{\ck}} (1+\|\Gamma\|_{\gamma; \bar{\ck}})\;,
		\end{equs}
		uniformly over all $f\in\cD^\gamma_{p,q}$, all compact subsets $\ck \subseteq \bR^d$ and all models $(\Pi,\Gamma)$. Here $\bar{\ck}$ denotes the 2-fattening of $\ck$.
		
		Given a second model $\bar{Z} = (\bar{\Pi}, \bar{\Gamma})$, the bound
		\begin{equs}\label{Eq:BoundReconsDiff}
			\left \| 2^{n\gamma} \left \| \sup_{\eta \in \CB^r} \left |\langle \cR f  - \bar{\cR} \bar{f} - \Pi_x f(x) + \bar{\Pi}_x \bar{f}(x), \eta_x^n \rangle \right | \right \|_{L^p(\ck; dx)} \right \|_{\ell^q(n)} \\ \lesssim  \$ f; \bar{f} \$_{p, q, \gamma; \bar{\ck}} +  \|\Pi - \bar{\Pi}\|_{\gamma; \bar{\ck}} + \|\Gamma - \bar{\Gamma}\|_{\gamma; \bar{\ck}};,
		\end{equs}
		holds uniformly over compact sets $\ck \subseteq \bR^d$ and choices of models and modelled distributions satisfying $\$f\$_{p,q,\gamma; \bar{\ck}} + \$\bar f \$_{p,q,\gamma; \bar \ck} + \|Z\|_{\gamma; \bar \ck} + \|\bar{Z}\|_{\gamma; \bar{\ck}} \le C$.
	\end{theorem}

	\begin{remark}
		It is in principle possible to establish existence of a reconstruction operator on $\CD_{p,q}^{\gamma}$ for $\gamma \le 0$. Later, we will in fact apply the reconstruction bound in that case. However, we will always then have a particular value of the reconstruction in mind and, since uniqueness fails when $\gamma \le 0$, the existence part of the statement is then of no use to us. As a result, we don't pursue this route further.
	\end{remark}
	
	At this point, we observe that by Besov embedding, it follows that if $f \in \CD_{p,q}^\gamma$ for $\gamma > 0$ and $f$ takes values in a sector of regularity $\alpha \in \bR$ then $\CR f \in \CC^{\alpha - |\fs|/p}$. On the other hand, if $f$ is sufficiently regular (i.e.\ if $\gamma$ is sufficiently large) one would expect its reconstruction
	to actually belong to $\CC^{\alpha}$.
	
	Our next result will provide for $\gamma < \alpha + |\fs|/p$ an interpolation between these two settings. For convenience, we restrict to the case $q = \infty$ since that is the only setting in which we will apply the result.
	
	\begin{theorem}\label{theo:betterReconstr}
		Fix a model $Z = (\Pi, \Gamma)$ and let $f \in \CD^\gamma_{p}(Z)$ take values in a sector $V$ and let $\alpha \eqdef \min(\CA_V \setminus \bN)$. 
		Let $\gamma$ be such that $0 < \gamma < \alpha + |\fs|/p$ and suppose that $\gamma - |\fs|/p \not \in \bN$. Then the reconstruction $\CR f$
		from Theorem~\ref{theo: Reconstruction} takes values in the Besov--Hölder space $\CC^{\gamma - |\fs|/p}$.
		
		Additionally, we have for any compact set $\ck \subseteq \bR^d$ the bound
		\begin{equ}[ReconBound]
			\|\cR f\|_{\CC^{\gamma - |\fs|/p}; \ck} \lesssim \$ f \$_{p, \gamma; \bar{\ck}} \| \Pi \|_{\gamma; \bar{\ck}} (1 + \| \Gamma \|_{\gamma; \bar{\ck}}).
		\end{equ}
		
		Given a second model $\bar{Z} = (\bar{\Pi}, \bar{\Gamma})$ and $\bar{f} \in  \CD^\gamma_{p,\infty}(\bar{Z})$ taking values in the same sector as $f$, we also have that 
		\begin{equs}[ReconBoundDiff]
			\|\cR f - \bar{\cR} \bar{f} \|_{\CC^{\gamma - |\fs|/p}; \ck} & \lesssim \$ f; \bar{f} \$_{p, \gamma; \bar{\ck}} + \|\Pi - \bar{\Pi}\|_{\gamma; \bar{\ck}} + \|\Gamma - \bar{\Gamma}\|_{\gamma; \bar{\ck}}
		\end{equs}
		uniformly over compact sets $\ck \subseteq \bR^d$ and choices of models and modelled distributions satisfying $\$f\$_{p,q,\gamma; \bar{\ck}} + \$\bar f \$_{p,q,\gamma; \bar \ck} + \|Z\|_{\gamma; \bar \ck} + \|\bar{Z}\|_{\gamma; \bar{\ck}} \le C$.
	\end{theorem}
	
	\begin{proof}	
		For $\gamma > |\fs|/p$, by (a local analogue of) \cite[Theorem 4.1]{Cyril} we have $\CD_{p,q}^\gamma \subset \CD_{\infty, \infty}^{\gamma - |\fs|/p}$ so that $\CR f \in \CC^{\gamma - |\fs|/p}$ by Theorem~\ref{theo: Reconstruction}. As a result, we will assume without loss of generality that $\bar \gamma = \gamma - |\fs|/p < 0$. We fix $\phi, \rho$ as in Definition~\ref{def: semigroup kernel}.
		
		To obtain the bound \eqref{ReconBound}, by Theorem~\ref{theo: kernel swapping}, it suffices to show that for $n \ge 0$, we have $|\cR f (\phi_y^{n})| \lesssim 2^{-n \bar \gamma}$ where the implicit constant depends on the model and choice of $f$ only through the right-hand side of \eqref{ReconBound}. The bound \eqref{ReconBoundDiff} 
		follows via similar ideas that are only notationally more complex, so we only include a proof of the first bound. 
		
		We first note that by the convolution semigroup type property \eqref{e:convolProp} of $\phi$, we can write
		\begin{equ}
			|\scal{\CR f, \phi_y^{n}}| = \left|\int\scal{\CR f, \phi_x^{n+1}} \rho_y^{n}(x)\,dx\right|\;,
		\end{equ}
		which in turn is bounded by
		\begin{equ}[e:breakTwo]
			\left |\int\scal{\CR f - \Pi_x f(x), \phi_x^{n+1}} \rho_y^{n}(x)\,dx\right |
			+  \left|\int\scal{\Pi_x f(x), \phi_x^{n+1}} \rho_y^{n}(x)\,dx\right|\;.
		\end{equ}
		Regarding the first term, Theorem~\ref{theo: Reconstruction}
		implies that the $L^p$ norm (in $x$) of $\scal{\CR f - \Pi_x f(x), \phi_x^{2\lambda}}$
		is of order $2^{-n \gamma}$ (with a constant of the correct type). Applying Hölder's inequality, this immediately yields the desired bound of order
		$2^{-n \left (\gamma-|\fs|/p \right )}$ without any assumption other than $\gamma > 0$.
		
		To bound the second term, we use again \eqref{e:convolProp} to write 
		\begin{equ}
			\int\scal{\Pi_x f(x), \phi_x^{n+1}} \rho_y^{n}(x)\,dx
			= \iint\scal{\Pi_x f(x), \phi_z^{n+2}} \rho_x^{n+1}(z) \rho_y^{n}(x)\,dx\,dz\;.
		\end{equ}
		Note that the outer integral is restricted to a ball of radius $\CO(2^{-n})$ around $y$
		(which is fixed once and for all) and, for fixed $y,z$, the function 
		$x \mapsto \rho_x^{n+1}(z) \rho_y^{n}(x)$
		is of the form $2^{-n |\fs|} \psi_z^{n+1}$ for some nice test function $\psi$
		which still depends on $y,z$ itself, but in a way that's uniformly bounded in $\CC^r$ (for any fixed $r > 0$)
		and with uniformly bounded supports.
		
		This shows that it is sufficient to get a bound of the form
		\begin{equ}[e:wantedBound]
			\Big|\int\scal{\Pi_x f(x), \phi_z^{n+2}}\, \psi_z^{n+1}(x)\,dx\Big|
			\lesssim \$ f \$_{p, \gamma; \bar{\ck}} \| \Pi \|_{\gamma; \bar{\ck}} (1 + \| \Gamma \|_{\gamma; \bar{\ck}}) 2^{-n \bar \gamma}\;.
		\end{equ}
		We then rewrite $\psi_z^{n+1}$ as a telescopic sum:
		\begin{equ}
			\psi_z^{n+1} = \psi_z^{0} + \sum_{k=0}^{n} \bar{R}_z^k
		\end{equ}
		where $\bar{R}^k = \psi^{k+1} - \psi^k$.
		To bound the first term, we write $\Pi_x f(x) = \Pi_z \Gamma_{zx} f(x)$ and exploit the fact that $f_\zeta \in L^p$ for every $\zeta < \gamma$
		by definition, so that
		\begin{equs}
			\left |\int\scal{\Pi_x f(x), \phi_z^{n+2}}\, \psi_z^{0}(x)\,dx \right|
			& \lesssim \sum_{\zeta = \alpha \wedge 0}^\gamma 2^{-n\zeta} \|\psi_z^{0}\|_{L^{p'}} \$ f \$_{p, \gamma; \bar{\ck}} \| \Pi \|_{\gamma; \bar{\ck}} (1 + \|\Gamma\|_{\gamma; \bar{\ck}})
			\\ & \lesssim 2^{-n (\alpha \wedge 0)} \$ f \$_{p, \gamma; \bar{\ck}} \| \Pi \|_{\gamma; \bar{\ck}}(1 + \|\Gamma\|_{\gamma; \bar{\ck}}) \\ & \lesssim 2^{-n \bar \gamma} \$ f \$_{p, \gamma; \bar{\ck}} \| \Pi \|_{\gamma; \bar{\ck}} (1 + \|\Gamma\|_{\gamma; \bar{\ck}})\;,
		\end{equs}
		where we used the fact that $\bar \gamma \le \alpha$ by assumption and wrote $p'$ for the exponent conjugate to $p$.
		
		For the remaining terms, we use the fact that $\bar R$ integrates to $0$ to write
		\begin{equs}
			\Bigg|\int\scal{\Pi_x f(x), \phi_z^{n+2}}\, \bar R_z^{k}&(x)\,dx\Bigg|
			\\ & =  \left |\iint\scal{\Pi_u (f(u) - \Gamma_{ux} f(x)), \phi_z^{n+2}}\, \bar R_z^{k}(x) \chi^{k}_z(u)\,dx\,du\right |
			\\ & \lesssim \sum_{\alpha \wedge 0 \le \beta < \gamma} 2^{-n \beta} \left | \iint F_\beta(x, h) \bar R_z^k(x) \chi_z^k(x+h) \, dx \, dh \right | 
		\end{equs}
		where $\chi$ is an arbitrary test function integrating to $1$.
		
		Applying H\"older's inequality in $x$, we obtain a bound of order $$\sum_{\alpha \wedge 0 \le \beta < \gamma} 2^{-n \beta} 2^{-k (\gamma - \beta)} 2^{k |\fs|/p} \$ f \$_{p, \gamma; \ck}$$ which has a sum in $k \le n$ of the correct order since $\gamma - \beta - |\fs|/p < 0$ by assumption.
	\end{proof}

	Since by now multiplication of modelled distributions and Schauder estimates for modelled distributions of positive regularity are well understood, we will omit discussion of such results in this paper (though we will later include similar statements for variants of these spaces). However, we will also be interested in a version of the usual Schauder estimates for Besov modelled distributions of negative regularity. 
	
	The only real barrier for such results is that in this regime the reconstruction is not unique and therefore the usual definition of the abstract integral has some ambiguity due to its dependence on the choice of reconstruction. However, it is the case that for any suitable choice of reconstruction, analogues of the usual Schauder results hold.

	\begin{definition}\label{def: reconstruction candidate}
		For $\gamma \in \bR$, $p \in [1, \infty]$, a candidate for the reconstruction of $f \in \CD_{p}^\gamma$ is a distribution $\CR f$ such that for each compact subset $\ck \subseteq \bR^d$ there exists a $C(f; \ck) > 0$ such that
		\begin{equ}[e:boundCandidate]
			\Big\| \sup_{\eta \in \CB^r} \big|\langle \cR f - \Pi_x f(x), \eta_x^\lambda\rangle \big| \Big\|_{L^p(\ck; dx)}  \le C(f; \ck) \lambda^{\gamma} 
		\end{equ}
		for all $\lambda \in (0,1]$ and compact sets $\ck \subseteq \bR^d$.
		
		A candidate for the reconstruction operator on $\CD_{p}^\gamma$ is a map $\CR: \CD_{p}^\gamma \to \CD'(\bR^d)$ such that for each $f \in \CD_{p}^\gamma$, $\CR f$ is a candidate for the reconstruction of $f$.
	\end{definition}
	
	Given $\gamma \leq 0$ and a candidate $\CR f$ for the reconstruction operator of $f \in \CD_{p}^\gamma$, it is now straightforward to mimic the usual definition of the abstract integration operators. That is, we define for $\ft \in \Lab_+$ the abstract integration operators $\CK_\gamma^\ft$ via the same formula as in the case $\gamma > 0$ given in \cite[Equation (5.15)]{Hai14}, where appearances of the reconstruction of $f$ there are replaced by our candidate for its reconstruction.
	We then have the following analogue of the Schauder estimates.
	
	\begin{theorem}
		Let $\gamma \leq 0$ and suppose that $\CR f$ is a candidate for the reconstruction of $f \in \CD_{p}^\gamma$.
		Then $\CK_\gamma^\ft f \in \CD_{p}^{\gamma + |\ft|_\fs}$ and satisfies the bound
		\begin{equs}
			\$ \CK_\gamma^\ft f \$_{p, \gamma + |\ft|_\fs; \ck} \lesssim C(f; \bar{\ck}) + \$ f \$_{p, \gamma; \ck} ( 1+ \|\Pi\|_{\gamma; \bar{\ck}}).
		\end{equs}
		
		Furthermore, $K_\ft \ast \CR f$ is a candidate for the reconstruction of $\CK_\gamma^\ft f$. In particular when $\gamma + |\ft|_\fs > 0$ so that the reconstruction operator $\tilde{\CR}$ on $\CD_{p}^{\gamma + |\ft|_\fs}$ is uniquely defined, we have that $\tilde{\CR} \CK_\gamma^\ft f = K_\ft \ast \CR f$.
		
		Additionally, given a second model $\bar{Z}$ and a candidate $\bar{\CR} \bar{f}$ for the reconstruction operator of $\bar{f} \in \CD_{p}^{\gamma}(\bar{Z})$ such that for each compact $\ck$ there exists $C(f,\bar{f}; \ck) > 0$ such that
		\begin{equs}
			\Big\| \sup_{\eta \in \CB^r} \big|\langle \cR f - \bar{\cR} \bar{f} - \Pi_x f(x) + \bar{\Pi}_x f(x), \eta_x^\lambda\rangle \big| \Big\|_{L^p(\ck; dx)}  \le C(f, \bar{f}; \ck) \lambda^{\gamma},
		\end{equs}
		we have that
		\begin{equs}
			\$ \CK_\gamma f ; \bar{\CK}_\gamma \bar{f} \$_{p, \gamma + |\ft|_\fs; \ck} \lesssim C(f, \bar{f}; \bar{\ck}) + \|\Pi - \bar{\Pi}\|_{\gamma; \bar{\ck}} + \| \Gamma - \bar{\Gamma} \|_{\gamma; \bar{\ck}} + \$ f; \bar{f} \$_{p, \gamma; \ck}
		\end{equs}
		where for any $C > 0$, the implicit constant can be chosen uniformly over $\| Z \|_{\gamma; \bar \ck} + \|\bar{Z}\|_{\gamma; \bar \ck} + \$ f \$_{p, \gamma; \ck} + \$ \bar{f} \$_{p, \gamma; \ck} + C(f; \ck) + C(\bar{f}; \ck) \le C$.
	\end{theorem} 
	\begin{proof}
		The proof is almost line by line the same as the proof of \cite[Theorem 5.1]{Cyril} with the only significant difference outside of adapting to the local setting being that one replaces applications of the reconstruction theorem by references to \eqref{e:boundCandidate}, so one has to be a bit more careful to correctly track the constants showing up in the proof.
	\end{proof}
	
	\subsection{Pointed Modelled Distributions}\label{section: Pointed Modelled Distributions}
	
	Whilst Besov modelled distributions are a powerful tool for describing generic distributions at the level of the regularity structure, they do not capture some important properties for our purposes. In particular, given a model $(\Pi, \Gamma)$, we are interested in describing the Fr\'echet derivative in the driving noise of terms of the form $\Pi_x \tau(\phi_x^\lambda)$. The most important feature of such terms (and their derivatives) is that they behave as if they were a more regular distribution locally around the point $x$. 
	
	To this end, we introduce a novel space of \textit{pointed modelled distributions} which are precisely those modelled distributions which behave like Besov modelled distributions, but exhibit more regular behaviour around some distinguished point. 
	In this paper we consider spaces of pointed modelled distributions whose Besov nature has parameter $q = \infty$. This is mainly for convenience and isn't a serious technical requirement, however such results are sufficient for all of our needs and this assumption does simplify some proofs.
	
	We begin by defining 
	\begin{equ}
		\|f\|_{\lambda,x,p} \eqdef \Bigg ( \int_{|y-x| \le \lambda} |f(y)|^p\,dy \Bigg )^{1/p} .
	\end{equ} 
	Here and in what follows, we consider the value $p \in [1, \infty)$ to be fixed and so we will often suppress it in the notation and simply write $\|f\|_{\lambda, x}$ for the above quantity. In our applications we will always
	take $p=2$, but the results of this section hold for arbitrary values.
	
	\begin{definition}\label{def: pointed modelled distribution}
		Given $\gamma, \nu \in \bR$, a model $Z = (\Pi, \Gamma)$ and $x \in \bR^d$, we let $\CD^{\gamma,\nu;x}_{p}$ be the set of elements $f \in \CD_{p}^\gamma$ that admit the additional bounds
		\begin{equ}[e:assumx]
			\||f|_\zeta\|_{\lambda,x} \lesssim \lambda^{\nu-\zeta}\;,\qquad 
			\||f(y+h) - \Gamma_{y+h,y}f(y)|_\zeta\|_{\lambda,x; dy} \lesssim |h|_\fs^{\gamma-\zeta} \lambda^{\nu-\gamma}\;.
		\end{equ}
		uniformly over $\lambda \in (0,1]$ and $|h|_\fs \le \lambda$.
		We call the first of these bounds the \emph{local bound} and the second the \emph{translation bound}.
		
		We then define $ \$ f \$_{p, \gamma, \nu; x}$ to be the smallest constant implicit in 
		the notation $\lesssim$ such that the local and translation bounds for $f$ hold.
		
		Given a second model $\bar{Z}$ and $f \in \CD^{\gamma,\nu,x}_{p}(Z), \bar{f} \in \CD^{\gamma,\nu,x}_{p}(\bar{Z})$ we then write
		\begin{equs}
			\$ f, \bar{f} \$_{p, \gamma, \nu; x}  = &  \sup_{\zeta < \gamma} \sup_{\lambda \in (0,1]} \frac{\| |f - \bar{f}|_\zeta \|_{\lambda, x}}{\lambda^{\nu - \gamma}} \\& + \sup_{\lambda \in (0,1]} \sup_{\|h\|_\fs < \lambda} \sup_{\zeta < \gamma} \frac{ \||\Delta_h f(y) - \bar{\Delta}_h \bar{f}(y) |_\zeta\|_{\lambda,x; dy}}{|h|_\fs^{\gamma-\zeta}\lambda^{\nu-\gamma}}.
		\end{equs}
	\end{definition}
	
	\begin{remark}
		Whilst typically we will be most interested in the case $\nu > \gamma$, we remark that Definition~\ref{def: pointed modelled distribution} adds non-trivial constraints beyond the definition of $\CD_p^\gamma$ even in the case $\nu \le \gamma$. 
		
Indeed, whilst in that regime the translation bound is immediate from the definition of $\CD_p^\gamma$, the local bound is non-trivial since it places constraints stronger than mere $L^p$ integrability on 
components with $\zeta < \nu$.
	\end{remark}
	
	The idea here is that $f \in \CD_p^{\gamma, \nu;x}$ if it belongs to $\cD^{\gamma}_{p}$ but behaves ``better'' by an order $\nu - \gamma$ near the point $x$. Our main technical result is that the reconstruction of such distributions behaves ``as if'' it were in $\CC^{\nu - |\fs|/p}$ (as opposed to $\CC^{\gamma - |\fs|/p}$
	as in Theorem~\ref{theo:betterReconstr}), provided that it is tested against rescaled test functions centred around $x$, which translates this intuition at the level of modelled distributions into a precise statement for their reconstructions.
	
	Our proof of this reconstruction result will proceed by obtaining a standard Besov modelled distribution with certain norm behaviour by localising a pointed modelled distribution around $x$ and then applying Theorem~\ref{theo:betterReconstr}. As a result, before discussing reconstruction of pointed modelled distributions, it is natural for us to discuss their products.

	In order to cleanly state our multiplication result, we begin by fixing some context and notation. 
	Throughout this subsection $Z = (\Pi, \Gamma)$ and $\bar{Z} = (\bar{\Pi}, \bar{\Gamma})$ will be models and we will sectors $V_1$ and $V_2$ of $\mathscr{T}$ of regularity $\alpha_1$ and $\alpha_2$ respectively. We will assume that $\mathscr{T}$ comes with a product $\star$ such that $(V_1, V_2)$ is $\gamma$-regular for $\gamma = (\gamma_1 + \alpha_2)\wedge(\gamma_2 + \alpha_1$).
	
	We will write $\CD_p^{\gamma, \nu; x}(V; Z)$ for the set of elements of $\CD_p^{\gamma, \nu; x}(Z)$ that take values in the sector $V$. Given $f_1 \in \CD_{p_1}^{\gamma_1, \nu_1; x}(V_1; Z)$ and $f_2 \in \CD_{p_2}^{\gamma_2, \nu_2; x}(V_2; \bar{Z})$ our goal will be to show that the product
	\begin{equ}[e:defProductProj]
		f \eqdef f_1 \star_{\gamma} f_2 \eqdef \CQ_\gamma (f_1 \star f_2)
	\end{equ}
	lies in a suitable space of pointed modelled distributions.
	
	We first remark that it follows from the multiplication result for Besov modelled distributions that we will only have to consider the pointed bounds in what follows. We will often suppress the notation $\star$, simply replacing it either by $\cdot$ or simply by writing $\tau_1 \tau_2 \eqdef \tau_1 \star \tau_2$ since the presence of the product will always be clear from context.
	
	\begin{theorem}\label{theo: pointed multiplication}
		In the context described above, suppose that $f_1 \in \cD_{p_1}^{\gamma_1, \nu_1; x}(V_1; Z)$ and $f_2 \in \CD_{p_2}^{\gamma_2, \nu_2; x}(V_2; \bar{Z})$ where $\gamma_i, \nu_i \in \bR$ and $p_i \in [1,\infty]$. Let $\gamma = (\gamma_1 + \alpha_2) \wedge (\gamma_2 + \alpha_1)$ and let $p$ satisfy $\frac{1}{p} = \frac{1}{p_1} + \frac{1}{p_2}$.
		
		Then $f \eqdef f_1 \star_{\gamma} f_2 \in \cD_p^{\gamma, \nu_1 + \nu_2; x}$ and the bound
		\begin{equ}
			\$ f \$_{p, \gamma, \nu; x} \lesssim \$ f_1 \$_{p_1, \gamma_1, \nu_1; x} \cdot \$ f_2 \$_{p_2, \gamma_2, \nu_2; x}
		\end{equ}
		holds.
		
		Additionally, if $\bar{f}_1 \in \cD_{p_1}^{\gamma_1, \nu_1; x}(V_1; \bar{Z})$ and $\bar{f}_2 \in \cD_{p_2}^{\gamma_2, \nu_2;x}(V_2; \bar{Z})$ then for $\bar{f} \eqdef \bar{f}_1 \star_{\gamma} \bar{f}_2$ we have the bound
		\begin{equs}
			\$ f, \bar{f} \$_{p, \gamma, \nu; x} \lesssim \$ f_1 , \bar{f}_1 \$_{p_1, \gamma_1, \nu_1; x} + \$ f_2 , \bar{f}_2 \$_{p_2, \gamma_2, \nu_2; x} + \| \Gamma - \bar{\Gamma} \|_{\gamma_1 \vee \gamma_2; B_x}
		\end{equs}
		uniformly over choices of $f_i, \bar{f}_i, Z, \bar{Z}$ such that $$\$ f_i \$_{p_1, \gamma_1, \nu_1; x}, \$ \bar{f}_i \$_{p_1, \gamma_1, \nu_1; x},  \| \Gamma \|_{\gamma_1 \vee \gamma_2; B_x}, \| \bar{\Gamma} \|_{\gamma_1 \vee \gamma_2; B_x} \leq C.$$
	\end{theorem}
	
	\begin{remark}
		In the application we have in mind, we will always apply this result with $p_1 = p$ (typically equal to $2$) and $p_2 = \infty$, but since the general result is no more difficult to obtain we include its proof here.
	\end{remark}
	
	\begin{proof}
		We first consider the case of a single model and begin by verifying the local bound. By H\"older's inequality, we have that
		\begin{equs}
			\| |f|_\zeta \|_{\lambda,x,p} & \leq \sum_{\alpha_1 \le \beta \le \zeta-\alpha_2} \| |f_1|_\beta \|_{\lambda,x,p_1} \||f_2|_{\zeta - \beta}\|_{\lambda,x,p_2}
			\\ & \lesssim \sum_{\alpha_1 \le \beta \le \zeta-\alpha_2} \lambda^{\nu_1 - \beta} \lambda^{\nu_2 - \zeta + \beta} \$f_1\$_{p_1, \gamma_1, \nu_1; x} \$ f_2 \$_{p_2, \gamma_2, \nu_2; x}
			\\ & \lesssim \lambda^{\nu_1 + \nu_2 - \zeta} \$f_1\$_{p_1, \gamma_1, \nu_1; x} \$ f_2 \$_{p_2, \gamma_2, \nu_2; x}
		\end{equs}
		as desired.
		
		We now move to the translation bound. We write as usual
		\begin{equs}
			f(y+h) -  \Gamma_{y+h,y}&f(y) = [f_1(y+h) - \Gamma_{y+h, y} f_1(y)] f_2(y + h) 
			\\ & + f_1(y+h) [f_2(y+h) - \Gamma_{y+h,y} f_2(y)]
			\\ & -  [f_1(y+h) - \Gamma_{y+h, y} f_1(y)] \cdot [f_2(y+h) - \Gamma_{y+h,y} f_2(y)].
		\end{equs}
		To control the first term, we note that $\| | \Delta_h f_1(y) f_2(y+h) |_\zeta \|_{\lambda, x, p; dy}$ is bounded by
		\begin{equs}
			 & \sum_{\alpha_2 \le \beta \le \zeta-\alpha_1} \| |\Delta_h f_1(y) |_{\zeta - \beta} \|_{\lambda, x, p_1; dy} \| |f_2(y+h)|_{\beta} \|_{\lambda, x, p_2; dy}
			\\ & \lesssim \sum_{\alpha_2 \le \beta \le \zeta-\alpha_1} |h|_\fs^{\gamma_1 - \zeta + \beta} \lambda^{\nu_1 - \gamma_1} \lambda^{\nu_2 - \beta} \$f_1\$_{p_1, \gamma_1, \nu_1; x} \$ f_2 \$_{p_2, \gamma_2, \nu_2; x}
			\\ & \lesssim \sum_{\alpha_2 \le \beta \le \zeta-\alpha_1} |h|_\fs^{\gamma_1 + \alpha_2 - \zeta} |h|_\fs^{\beta - \alpha_2} \lambda^{\nu_1 - \gamma_1} \lambda^{\nu_2 - \alpha_2} |h|_\fs^{\alpha_2 - \beta} \$f_1\$_{p_1, \gamma_1, \nu_1; x} \$ f_2 \$_{p_2, \gamma_2, \nu_2; x}
			\\ & \lesssim |h|_\fs^{\gamma_1 + \alpha_2 - \zeta} \lambda^{\nu_1 + \nu_2 - \gamma_1 - \alpha_2} \$f_1\$_{p_1, \gamma_1, \nu_1; x} \$ f_2 \$_{p_2, \gamma_2, \nu_2; x}
		\end{equs}
		Switching the roles of $f_1$ and $f_2$ in the above yields the desired bound on the second term and we are left to consider $\| |\Delta_h f_1(y) \Delta_h f_2(y)]|_\zeta \|_{\lambda, x, p; dy}$. We bound this term by
		\begin{equs}
		&  \sum_{\alpha_1 \le \beta \le \zeta-\alpha_2} \| |\Delta_h f_1(y)|_\beta \|_{\lambda, x, p_1; dy} \| |\Delta_h f_2(y)|_{\zeta - \beta} \|_{\lambda, x, p_2; dy} 
			\\ & \lesssim \sum_{\alpha_1 \le \beta \le \zeta-\alpha_2} |h|_\fs^{\gamma_1 + \gamma_2 - \zeta} \lambda^{\nu_1 + \nu_2 - \gamma_1 - \gamma_2} \$f_1\$_{p_1, \gamma_1, \nu_1; x} \$ f_2 \$_{p_2, \gamma_2, \nu_2; x}
			\\ & \lesssim |h|_\fs^{\gamma - \zeta} \lambda^{\nu_1 + \nu_2 - \gamma} \$f_1\$_{p_1, \gamma_1, \nu_1; x} \$ f_2 \$_{p_2, \gamma_2, \nu_2; x}.
		\end{equs}	
		
		We turn now to the bounds for $f - \bar{f}$. The local bounds work in much the same way as in the case of a single pointed modelled distribution so we will consider only the translation bound. We write
		\begin{equ}
			\Delta_h f(y) - \bar{\Delta}_h \bar{f}(y) = T_1 + T_2 + T_3 + T_4 + T_5
		\end{equ}
		for 
		\begin{equs}
			T_1 & = [\Delta_h f_1 - \bar{\Delta}_h \bar{f}_1](y) \, f_2(y+h) \\
			T_2 & = \Gamma_{y+h, y} f_1(y) \, [\Delta_h f_2 - \bar{\Delta}_h \bar{f}_2](y) \\
			T_3 & = - \bar{\Gamma}_{y+h, y} [ f_1 - \bar{f}_1](y) \,  \bar{\Delta}_h \bar{f}_2(y) \\
			T_4 & = - [\bar{\Gamma}_{y+h, y} f_1(y) - \Gamma_{y+h, y} f_1(y)] \,  \bar{\Delta}_h \bar{f}_2(y) \\
			T_5 & = \bar{\Delta}_h \bar{f}_1(y) \, [ f_2(y) - \bar{f}_2(y)].
		\end{equs}
		
		The desired bounds on $T_1$ and $T_5$ follow in much the same way as the bounds in the single model case. The remaining three bounds are all obtained in a similar way to each other so in the interest of brevity we 
		only consider $T_2$.
		We write
		\begin{equs}
			\| |T_2|_\zeta \|_{\lambda, x, p} & \le \sum_{\alpha_1 \le \beta \le \zeta - \alpha_2} \| | \Gamma_{y+h, y} f_1(y) |_\beta \|_{\lambda, x, p_1; dy} \, \| | \Delta_h f_2 - \bar{\Delta}_h \bar{f}_2 |_{\zeta - \beta} \|_{\lambda, x, p_2; dy}
			\\
			& \lesssim \sum_{\alpha_1 \le \beta \le \zeta - \alpha_2} \sum_{\beta \le \eta < \gamma_1} |h|_\fs^{\eta - \beta} \||f_1|_\eta \|_{\lambda, x, p_1} \| | \Delta_h f_2 - \bar{\Delta}_h \bar{f}_2 |_{\zeta - \beta} \|_{\lambda, x, p_2}
			\\
			& \lesssim \sum_{\alpha_1 \le \beta \le \zeta} \sum_{\beta \le \eta < \gamma_1} |h|_\fs^{\eta - \beta} \lambda_1^{\nu_1 - \eta} |h|_\fs^{\gamma_2 - \zeta + \beta} \lambda^{\nu_2 - \gamma_2} \$ g, \bar{g} \$_{p_2, \gamma_2, \nu_2; x}
			\\
			& \lesssim |h|_\fs^{\gamma_2 + \alpha_1 - \zeta} \lambda^{\nu_1 + \nu_2 - \alpha_1 - \gamma_2} \$ f_2, \bar{f}_2 \$_{p_2, \gamma_2, \nu_2; x}
		\end{equs}
		which yields the desired bound.
	\end{proof}
	
	With this result in hand, we proceed with our plan to obtain a reconstruction result for pointed modelled distributions via localisation. The key observation here is the following proposition.
	
	\begin{proposition} \label{prop: localise}
		Suppose $f \in \cD_p^{\gamma, \nu; x}(Z)$ is such that $\supp f \subseteq B(x, \lambda)$ for $\lambda \in (0,1]$. Then
		\begin{equ} 
			\$ f\$_{p, \gamma; \ck} \lesssim \lambda^{\nu - \gamma} \$ f \$_{p, \gamma, \nu; x} 
		\end{equ}
		with a constant that is independent of $\lambda$ and $f$.
		
		Additionally, given a second model $\bar{Z}$ and $\bar{f} \in \cD_p^{\gamma, \nu; x}(\bar{Z})$ such that $\supp \bar{f} \subseteq B(x, \lambda)$, we have the bound $\$ f, \bar{f} \$_{p, \gamma; \ck} \lesssim \lambda^{\nu - \gamma} \$ f, \bar{f} \$_{p, \gamma, \nu; x}$.
	\end{proposition}
	\begin{proof}
		For the local bound in the definition of $\$f\$_{p, \gamma; \ck}$, we note that, 
		since $f$ is supported in $B(x, \lambda)$,  we can write for $\zeta < \gamma$
		\begin{equ}
			\| |f|_\zeta \|_{L^p(\ck)} \le \| |f|_\zeta \|_{\lambda, x}  \lesssim \lambda^{\nu - \zeta} \$f\$_{p, \gamma, \nu; x}  \lesssim \lambda^{\nu - \gamma} \$f\$_{p, \gamma, \nu; x} \;,
		\end{equ}
		as desired, where we used the definition of $\$f\$_{p,\gamma, \nu;x}$ in the second inequality.
		
		We now turn to the translation bounds in the definition of $\$f\$_{p, \gamma; \ck}$. We divide into the cases $|h|_\fs \le \lambda$ and $|h|_\fs > \lambda$ since the definition of $\$f\$_{p, \gamma, \nu; x}$ only gives us added information in the former case.
		
		First suppose that $|h|_\fs \le \lambda$. With $\Delta_h f$ as in Definition~\ref{def: Dgamma}, we write
		\begin{equ}
			\| |\Delta_h f(y)|_\zeta \|_{L^p(\ck; dy)} \le \| |\Delta_h f(y)|_\zeta \|_{2 \lambda, x} \lesssim \lambda^{\nu - \gamma} |h|_\fs^{\gamma - \zeta} \$ f \$_{p, \gamma, \nu; x}
		\end{equ}
		which is a bound of the desired type. (Recall that we are considering the case $q = \infty$.)
		In the regime $|h|_\fs > \lambda$, we write
		\begin{equs}
			\| |\Delta_h f(y)|_\zeta \|_{L^p(\ck)} & \leq \| |f|_\zeta \|_{L^p(\ck)} + \sum_{\zeta \le \beta < \gamma} |h|_\fs^{\beta - \zeta} \| |f|_\zeta \|_{L^p(\ck)}
			\\ & \lesssim \sum_{\zeta \le \beta < \gamma}  \$ f \$_{p, \gamma, \nu; x}  |h|_\fs^{\beta - \zeta} \lambda^{\nu - \beta} \lesssim  \$ f \$_{p, \gamma, \nu; x}  |h|_\fs^{\gamma - \zeta} \lambda^{\nu - \gamma}.
		\end{equs}
		
		The bound in the case of two pointed modelled distributions then follows in essentially the same way.
	\end{proof}
	
	In order to leverage this result, recall that the reconstruction operator is local 
	in the following sense
	
	\begin{lemma}\label{lemma: ReconIsLocal}
		Let $f,g \in \CD_p^{\gamma}$ with $\gamma > 0$ and let $\psi$ be a smooth test function. 
		If $f = g$ on $\supp \psi$, then $\CR f(\psi) = \CR g(\psi)$.
	\end{lemma}
	This gives us the ability to leverage Proposition~\ref{prop: localise} if we can suitably lift bump functions into an appropriate $\CD_p^\gamma$ space.
	
	We fix a smooth compactly supported function $\chi$ such that $\chi(x) = 1$ for $|x|_\fs \le 2$ and, for given $\mu > 0$, define the Taylor lift of order $\mu$ of its rescaled version by
	\begin{equ}
		\chi_\lambda(x) = \sum_{|k|<\mu} \lambda^{-|k|_\fs} (D^k\chi)(x/\lambda) \, \frac{X^k}{k!}\;.
	\end{equ}
	
	Notice that we suppress the $\mu$ dependency in the notation since the precise value of $\mu$ will not be important to us so long as it is sufficiently large.
	
	The preceding results then allow us to derive the following pointed versions of the reconstruction theorem.
	\begin{theorem}\label{theo:pointedReconstruction}
		Let  $\gamma \in (0, \alpha + |\fs|/p) \setminus \bN$ and suppose that 
		$f \in \CD^{\gamma,\nu,x}_{p}(Z)$ takes values in a sector of regularity $\alpha \le 0$. Then writing $B_x$ for the closed ball around $x$ of radius $2$, 
\begin{equ}[e:boundPointed1]
		|\scal{\CR f,\psi_x^\lambda}| \lesssim \lambda^{\nu-|\fs|/p} \$ f \$_{p, \gamma, \nu; x} \| \Pi \|_{\gamma; B_x} (1 + \| \Gamma \|_{\gamma;B_x}),
\end{equ}		
uniformly over $\lambda \in (0,1]$ and $\psi \in \CB^r$.
		
		Further, for $\delta \le \lambda$, we have that
		\begin{equ}[e:boundPointed2]
			\|\scal{ \CR f - \Pi_y f(y), \psi_y^\delta } \|_{\lambda, x, p; dy} \lesssim \lambda^{\nu - \gamma}  \delta^{\gamma} \$ f \$_{p, \gamma, \nu; x}  \| \Pi \|_{\gamma; B_x} (1 + \| \Gamma \|_{\gamma;B_x}).
		\end{equ}
		
		Given a second model $\bar{Z}$ and $\bar{f} \in \CD^{\gamma,\nu,x}_{p}(\bar{Z})$ taking values in the same sector as $f$, we have that
		\begin{equs}
			\lambda^{|\fs|/p-\nu} | & \scal{\CR f - \bar{\CR} \bar{f},\psi_x^\lambda}|  \\ & \lesssim \$ f; \bar{f} \$_{p, \gamma, \nu; x} \| \Pi \|_{\gamma; B_x} (1 + \|\Gamma\|_{\gamma; B_x}) \\ & \quad + \$ \bar{f} \$_{p, \gamma, \nu; x} \left ( \|\Pi - \bar{\Pi}\|_{\gamma; B_x} (1 + \|\Gamma\|_{\gamma; B_x}) + \|\bar{\Pi}\|_{\gamma; B_x} \|\Gamma - \bar{\Gamma}\|_{\gamma; B_x} \right ).
		\end{equs}
		and further
		\begin{equ}
			\lambda^{\gamma - \nu} \delta^{- \gamma} \| \scal{ \CR f - \bar{\CR} \bar{f} - \Pi_y f(y) + \bar{\Pi}_y \bar{f}(y), \psi_y^\delta } \|_{\lambda, x, p; dy}
		\end{equ}
		is bounded by the same quantity.
		
		In the case where $\alpha = 0$, $\bar \alpha = \min(\CA_V \setminus \bN)$, $\alpha < \gamma - |\fs|/p < \bar\alpha$ and $\gamma, \gamma - |\fs|/p \not \in \bN$, the same bounds hold for $\psi \in \CB_{\lfloor \bar\alpha \rfloor}^r$.
		
	\end{theorem}
	\begin{proof}
		We provide a proof of the bounds for a single model; those for two models following similarly. 
		
		A straightforward calculation yields that if we take $\mu = \gamma - \alpha$ then $\chi_\lambda \in \CD_\infty^{\gamma - \alpha, 0;x}$ and that $\$\chi_\lambda\$_{\infty, \gamma - \alpha, 0;x} \lesssim 1$ uniformly in $\lambda$.
		
		Therefore, by Lemma~\ref{theo: pointed multiplication}, it follows that $f_\lambda(y) = f(y) \chi_\lambda(y-x)$ defines an element of $\cD_p^{\gamma, \nu;x}$ with $\$ f_\lambda \$_{p, \gamma; B_x} \lesssim \lambda^{\nu - \gamma} \$ f \$_{p, \gamma, \nu; x}$ uniformly in $\lambda$ by Proposition~\ref{prop: localise}.
Since		$|\scal{\CR f, \varphi_x^\lambda}| = |\scal{\CR f_\lambda, \varphi_x^\lambda}|$ by Lemma~\ref{lemma: ReconIsLocal}, the bound \eqref{e:boundPointed1} follows from Theorem~\ref{theo:betterReconstr}.
		The bound \eqref{e:boundPointed2}  follows similarly from Theorem~\ref{theo: Reconstruction}.
	\end{proof}
	
	We now turn our attention to establishing analogues of the usual Schauder estimates for pointed modelled distributions. 
	
	We note that our definition of the abstract integration map here will differ from the one for the usual spaces of Besov modelled distributions. The reason for this is that in order to preserve the better behaviour around the distinguished point $x$ it turns out to be necessary to subtract off a suitable Taylor jet around that distinguished point.
	
	As usual, we will assume that a kernel assignment $K = (K^\ft)_{\ft \in \Lab_+}$ has been fixed and will fix a choice of $\ft \in \Lab_+$ and write $\beta = |\ft|_\fs$ and suppress the dependency on $\ft$ in the kernel in order to simplify the forthcoming notation.
	
	\begin{definition}\label{def: pointed integration}
		For $\gamma, \nu > 0$, $p \in [1,\infty]$ and $f \in \cD_p^{\gamma, \nu; x}$, we define the pointed abstract integral $\CK_{\gamma, \nu}^{x,p} f$ of $f$ by
		\begin{equs}
			\CK_{\gamma, \nu}^{x,p} f(y) = \CI f(y) + \CJ(y) f(y) + (\CN_\gamma f)(y) - T_{\nu,p}^x f(y)
		\end{equs} 
		where 
		\begin{equs}
			\CJ(y) \tau &= \sum_{|k|_\fs < |\tau| + \beta} \frac{X^k}{k!} \big(D^k K \ast \Pi_y \tau\big) (y) \\
			(\CN_\gamma f)(y) &= \sum_{|k|_\fs < \gamma + \beta} \frac{X^k}{k!} [D^k K \ast (\CR f - \Pi_y f(y))](y) \\
			T_{\nu,p}^x f(y) &= \CQ_{< \gamma + \beta} \Big ( \sum_{|k|_\fs < \nu + \beta - |\fs|/p}  \frac{(X+y - x)^k}{k!} \big(D^k K \ast \CR f \big)(x) \Big ) \;.
		\end{equs}
	\end{definition}
	
	\begin{remark}
		Since $D^k K$ isn't a smooth test function, it is not immediately obvious that the evaluation of
		the distribution $D^k K \ast \CR f$ at the point $x$ is a well defined operation. However, this is a consequence of Theorem~\ref{theo:pointedReconstruction} in essentially the same way as the corresponding statement in \cite[Lemma 5.19]{Hai14} is a consequence of the reconstruction theorem there.
	\end{remark}

	Later, given a kernel assignment $(K^\ft)_{\ft \in \Lab_+}$ we will write $\CK_{\gamma, \nu, \ft}^x, \CI^\ft, \CN_{\gamma, \ft}, T_{\nu, p, \ft}^x$ for the operators defined in this way for the kernels $K^\ft$.
	
	Additionally, in order to simplify upcoming expressions we will set
	\begin{equs}
		K_{n; y,x}^{k, \eta}(z) = \partial^k K_n(y-z) - \sum_{\substack{l \ge k \\ |l|_\fs \le \eta}} \frac{(y-x)^{l-k}}{(l-k)!} \partial^l K_n(x-z).
	\end{equs}
	
	We will also repeatedly use the following estimates, which are immediate to obtain from the definitions but we state separately here for the reader's convenience.
	
	\begin{lemma} \label{Lemma: Pointed Model Bound}
		Fix a model $Z = (\Pi, \Gamma)$ and let $f \in \cD_p^{\gamma, \nu; x}(Z)$. Then for $\zeta \leq \gamma$ and $|h|_\fs < \lambda$
		\begin{equs}
			\|\scal{\Pi_y \CQ_\zeta f (y), \eta_y^\varepsilon} \|_{\lambda, x, p; dy} & \lesssim \| \Pi \|_{\delta; B_x} \$ f \$_{p, \gamma, \nu; x} \varepsilon^\zeta \lambda^{\nu - \zeta} \\
			\|\scal{\Pi_y \cQ_\zeta \Delta_{h} f(y), \eta_y^\varepsilon} \|_{\lambda, x, p; dy} & \lesssim \sum_{\eta: \zeta \le \eta < \gamma} \| \Pi \|_{\delta; B_x} (1 + \| \Gamma\|_{\delta; B_x}) |h|_\fs^{\eta - \zeta} \varepsilon^\eta \lambda^{\nu - \zeta}.
		\end{equs}
		Additionally, given a second model $\bar{Z} = (\bar{\Pi}, \bar{\Gamma})$ and $\bar{f} \in \CD_p^{\gamma, \nu; x}(\bar{Z})$, we have the bounds
		\begin{equ}
			 \|\scal{\Pi_y \CQ_\zeta - \bar{\Pi}_y \CQ_\zeta \bar{f}(y), \eta_y^\varepsilon} \|_{\lambda, x, p; dy} \lesssim  \varepsilon^\zeta \lambda^{\nu - \zeta} \left ( \|\Pi - \bar{\Pi}\|_{\delta; B_x} + \$f; \bar{f} \$_{p, \gamma, \nu;x} \right )
		\end{equ}
		and
		\begin{equs}
			\|\scal{\Pi_y \cQ_\zeta \Delta_{h} f(y) - & \bar{\Pi}_y \cQ_\zeta \bar{\Delta}_{h} \bar{f}(y), \eta_y^\varepsilon} \|_{\lambda, x, p; dy}  \\ &\lesssim \sum_{\eta: \zeta \le \eta < \gamma}  |h|_\fs^{\eta - \zeta} \varepsilon^\zeta \lambda^{\nu - \eta} \left ( \|Z - \bar{Z}\|_{\delta; B_x} + \$f; \bar{f} \$_{p, \gamma, \nu;x}\right )
		\end{equs}
		uniformly over choices of $Z, \bar{Z}, f, \bar{f}$ such that $\|Z\|_{\delta; B_x} + \|\bar{Z}\|_{\delta; B_x} + \$f\$_{p,\gamma, \nu; x} + \$\bar{f}\$_{p, \gamma, \nu; x} \le C$.
	\end{lemma}
	\begin{proof}
		The results in the case of a single model follow immediately by applying the definition of the norm of the model followed by the definition of $\$f\$_{p, \gamma, \nu;x}$.
		
		In the case of two models, a similar approach works where one has to appropriately rewrite the difference of models. For example, for the first inequality in this case we would write $\Pi_y \CQ_\zeta f(y) - \bar{\Pi}_y \CQ_\zeta \bar{f}(y) = (\Pi_y - \bar{\Pi}_y) \CQ_\zeta f(y) + \bar{\Pi}_y \CQ_\zeta (f(y) - \bar{f}(y))$ and then proceed in almost the exact same way as in the case of a single model. 
		
		The second inequality for the case of two models is then similar.
	\end{proof}
	
	The main result of this subsection is the following theorem.
	\begin{theorem} \label{theo:schauder}
		Fix models $Z = (\Pi, \Gamma)$ and $\bar{Z}= (\bar{\Pi}, \bar{\Gamma})$ and for $p \in [1, \infty]$, $\gamma, \nu \in \bR$ and $x \in \bR^d$, let $\CK_{\gamma, \nu}^{x,p}, \bar{\CK}_{\gamma, \nu}^{x,p}$ be the corresponding pointed integration operators. 
		
		Then for $f \in \CD_p^{\gamma, \nu; x}(Z)$ valued in a sector $V$ such that $0 < \gamma < \bar \alpha + |\fs|/p$ where $\bar \alpha = \min(\CA_V \setminus \bN)$, one has $\CK_{\gamma, \nu}^{x,p} f \in \CD_p^{\gamma + \beta, \nu + \beta; x}(Z)$ and further
		\begin{equs}
			\$ \CK_{\gamma, \nu}^{x, p} f \$_{p, \gamma + \beta, \nu + \beta; x} \lesssim \| \Pi \|_{\gamma; B_x} (1+\|\Gamma\|_{\gamma; B_x}) \$f\$_{p, \gamma, \nu; x}
		\end{equs}
		where $\beta$ is the order to which the kernel $K$ is regularising.
		
		Additionally, given $\bar{f} \in \CD_{p}^{\gamma, \nu; x}(\bar{Z})$
		\begin{equs}
			\$ \CK_{\gamma, \nu}^{x, p} f ; \bar{\CK}_{\gamma, \nu}^{x,p} \bar{f} \$_{p, \gamma + \beta, \nu + \beta} \lesssim \|Z - \bar{Z}\|_{\delta; B_x} + \$f; \bar{f} \$_{p, \gamma, \nu;x} 
		\end{equs}
		uniformly over choices of $Z, \bar{Z}, f, \bar{f}$ such that $\|Z\|_{\delta; B_x} + \|\bar{Z}\|_{\delta; B_x} + \$f\$_{p,\gamma, \nu; x} + \$\bar{f}\$_{p, \gamma, \nu; x} \le C$.
	\end{theorem}
	\begin{proof}
		As is usual in these proofs, the bounds at non-integer homogeneities follow immediately from the definition so we will consider only the bounds at integer levels. 
		
		We let $n_0$ be the largest integer such that $2^{-n_0} \ge \lambda$. Since $|k|_\fs < \gamma + \beta$ by assumption, we can write $k! \cQ_k \CK_{\gamma, \nu}^x f(y) = \sum_{n \le n_0} I_n + \sum_{n > n_0} J_n$ where
		\begin{equs}
			I_n & \eqdef \scal{ \CR f, K_{n;y,x}^{k, \nu + \beta - |\fs|/p}} - \sum_{\zeta \le |k|_\fs - \beta} \scal{\Pi_y f_\zeta(y), \partial^k K_n(y- \cdot)}
			\\
			J_n & \eqdef \scal{\CR f - \Pi_y f(y), \partial^k K_n(y- \cdot)} + \sum_{\zeta > |k|_\fs - \beta} \scal{\Pi_y f_\zeta(y), \partial^k K_n(y- \cdot)} \\ & \quad - \sum_{\substack{l \ge k \\ |l|_\fs \le \nu + \beta - |\fs|/p}} \frac{(y-x)^{l-k}}{(l-k)!} \scal{\CR f, \partial^l K_n(x - \cdot)}
		\end{equs}
		We will first derive bounds on $\|I_n\|_{\lambda, x, p}$ for $n \le n_0$. 
		
		Lemma~\ref{Lemma: Pointed Model Bound} immediately yields a bound of order $2^{n(|k| - \beta - \zeta)} \lambda^{\nu - \zeta}$ on the summand of the second term so that we are left to consider only the first term.
		
		To bound this term, we apply \cite[Proposition A.1]{Hai14} to write
		\begin{equs}
			\| \scal{\CR f, K_{n;y,x}^{k, \nu}}\|_{\lambda, x, p} \le \sum_{l \in \partial A_{\nu, p}^k} \left \| \int_{\mathbb{R}^d} \scal{\CR f, \partial^{l+k} K_n(x+h-\cdot)} \cQ^l(y-x, dh) \right \|_{\lambda, x, p}
		\end{equs}
		where $A_{\alpha,p}^k = \{j: |k+j|_\fs < \alpha + \beta - |\fs|/p \}$, $\partial A_{\alpha, p}^k = \{j \not \in A_{\alpha, p}^k: j - e_{m(j)} \in A_\alpha^k\}$ for $m(j) = \min\{i: j_i \neq 0\}$ and $\cQ^j(y-x, \cdot)$ is a signed measure supported on $\{h: h_i \in [0,(y-x)_i]\}$ with total mass $\frac{(y-x)^j}{j!}$.
		
		Since in the domain of integration $|h|_\fs \le |y-x|_\fs \le \lambda$, we can apply Theorem~\ref{theo:pointedReconstruction} to bound $|\scal{\CR f, \partial^{l+k} K_n(x+h - \cdot)}| \lesssim 2^{n(|l+k|_\fs - \beta - \nu + |\fs|/p)}$. This yields the bound
		\begin{equs}
			\| \scal{\CR f, K_{n;y,x}^{k, \nu}}\|_{\lambda, x, p} \lesssim \sum_{l \in \partial A_{\nu, p}^k} 2^{n(|l+k|_\fs - \beta - \nu + |\fs|/p)} \lambda^{|l|_\fs + |\fs|/p}
		\end{equs} 
		so that in total
		\begin{equs}
			\sum_{n \le n_0} & \|I_n\|_{\lambda, x, p} \\ & \lesssim \sum_{n \le n_0} \left [ \sum_{\zeta \le |k|_\fs - \beta} 2^{n(|k|_\fs - \beta - \zeta)} \lambda^{\nu - \zeta} + \sum_{l \in \partial A_{\nu, p}^k} 2^{n(|l+k|_\fs - \beta - \nu + |\fs|/p)} \lambda^{|l|_\fs + |\fs|/p} \right ]
			\\ & \lesssim \lambda^{\nu + \beta - |k|_\fs}
		\end{equs}
		as required.
		
		We now turn our attention to $J_n$ for $n > n_0$. Here, applying the second statement of Theorem~\ref{theo:pointedReconstruction} to the first term, Lemma~\ref{Lemma: Pointed Model Bound} to the second term and the first statement of Theorem~\ref{theo:pointedReconstruction} to the third term yields the bound
		\begin{equs}
			\sum_{n > n_0} \|J_n\|_{\lambda, x, p} & \lesssim \sum_{n > n_0} \bigg [2^{n(|k|_\fs - \beta - \gamma)} \lambda^{\nu - \gamma}  \quad + \sum_{\zeta > |k|_\fs - \beta} 2^{n(|k|_\fs - \beta - \zeta)} \lambda^{\nu - \zeta}  \\ &  + \sum_{\substack{l \ge k \\ |l|_\fs \le \nu + \beta - |\fs|/p}} \lambda^{|l-k|_\fs} 2^{n(|l|_\fs - \beta - \nu + |\fs|/p)} \bigg ]
			\lesssim \lambda^{\nu - \gamma}\;,
		\end{equs}
		again as required.
		
		We now turn our attention to the translation bound at integer homogeneities. We fix $h \in \mathbb{R}^d$ such that $|h|_\fs \le \lambda$. We remark that a (slightly fiddly but essentially straightforward) calculation shows that $$\Delta_h T_{\nu, p}^x f (y) = \sum_{(k,j,l)} \frac{(y-x)^j}{j!} \frac{h^l}{l!} D^{k+j+l} K \ast \CR f(x) \frac{X^k}{k!} $$
		where the sum is over multi-indices $k,j,l$ such that $|k|_\fs < \gamma+ \beta, \gamma+ \beta \le |k+l|_\fs < \nu + \beta - |\fs|/p$ and $|k+j+l|_\fs < \nu + \beta - |\fs|/p$. In what immediately follows, for a fixed multi-index $k$ such that $|k|_\fs < \gamma + \beta$, we write $M_k$ for the set of $(j,l)$ such that $(k,j,l)$ satisfies the above conditions.
		
		We will split our consideration for the translation bound into 3 regimes. We let $m_0$ be the largest integer such that $2^{-m_0} \ge |h|_\fs$. Note that $n_0 \le m_0$. We then write $k! \cQ_k \Delta_{h} f(y) = \sum_{n \le n_0} I_n' + \sum_{n_0 < n \le m_0} J_n' + \sum_{n > m_0} L_n'$ where
		\begin{equs}
			L_n' & \eqdef \scal{\CR f - \Pi_{y+h} f(y+h), \partial^k K_n(y+h - \cdot)} \\ & \quad - \sum_{l: |l+k|_\fs < \gamma + \beta} \frac{h^l}{l!} \scal{\CR f - \Pi_y f(y), \partial^{k+l} K_n(y-\cdot)} \\  & \quad + \sum_{\zeta > |k|_\fs - \beta} \scal{\Pi_{y+h} \cQ_\zeta \Delta_{h}f(y), \partial^k K_n(y+h - \cdot)}
			\\ & \quad - \sum_{(j,l) \in M_k} \frac{(y-x)^j}{j!} \frac{h^l}{l!} \scal{\CR f, D^{k+j+l} K_n(x - \cdot)}
			\\ J_n' & \eqdef \scal{\CR f - \Pi_y f(y), K_{n; y+h, y}^{k, \gamma + \beta}} - \sum_{\zeta \le |k|_\fs - \beta} \scal{\Pi_{y+h} \cQ_\zeta \Delta_{h} f(y), \partial^k K_n(y+h - \cdot)}
			\\ & \quad - \sum_{(j,l) \in M_k} \frac{(y-x)^j}{j!} \frac{h^l}{l!} \scal{\CR f, D^{k+j+l} K_n(x - \cdot)}
			\\
			I_n' & \eqdef \scal{\CR f, K_{n; y+h, y}^{k, \nu + \beta - |\fs|/p}} + \sum_{l : \gamma + \beta \le |l + k|_\fs \le \nu+ \beta - |\fs|/p} \frac{h^l}{l!} \scal{\CR f, K_{n ; y, x}^{k + l, \nu + \beta - |\fs|/p}} \\ & \quad - \scal{\Pi_y f(y), K_{n; y+h, y}^{k, \gamma + \beta}}  - \sum_{\zeta \le |k|_\fs - \beta} \scal{\Pi_{y+h} \cQ_\zeta \Delta_{h}f(y), \partial^k K_n(y+h - \cdot)}.
		\end{equs}
		First, we bound $\|L_n'\|_{\lambda, x, p}$ for $n > m_0$. An application of the second claim in Theorem~\ref{theo:pointedReconstruction} to the first and second term, the first claim in that result to the final term and of Lemma~\ref{Lemma: Pointed Model Bound} to the third term immediately yields
		\begin{equs}
			\|L_n'\|_{\lambda, x, p} & \lesssim 2^{n(|k|_\fs - \beta - \gamma)} \lambda^{\nu - \gamma} + \sum_{l: |l+k|_\fs < \gamma + \beta} |h|_\fs^{|l|_\fs} 2^{n(|k+l|_\fs - \beta - \gamma)} \lambda^{\nu - \gamma} \\ & \quad + \sum_{\zeta > |k|_\fs - \beta} \sum_{\eta \ge \zeta} |h|_\fs^{\eta - \zeta} 2^{n(|k| - \beta - \eta)} \lambda^{\nu - \zeta}
			\\ & \quad + \sum_{(j, l) \in M_k} |h|_\fs^{|l|_\fs} \lambda^{|j|_\fs + |\fs|/p} 2^{-n (\nu + \beta - |\fs|/p - |k+j+l|_\fs)}.
		\end{equs}
		It is then straightforward to see that this implies the required bound.
		
		The bounds on $\|J_n'\|_{\lambda, x, p}, \|I_n'\|_{\lambda, x, p}$ follow similarly by first applying \cite[Proposition A.1]{Hai14} to any term including an appearance of an expression of the form $K_{n;z, w}^{m, \eta}$ and then applying one of Theorem~\ref{theo:pointedReconstruction} or Lemma~\ref{Lemma: Pointed Model Bound} in a very similar way to many of the bounds already derived in this proof. As a result, we omit the details for these terms.
		
		It remains to obtain the bounds for the instance in which there are two models. These bounds follow in almost the same way as above, replacing the applications of statements in Lemma~\ref{Lemma: Pointed Model Bound} and Theorem~\ref{theo:pointedReconstruction} with their analogues for differences of models and hence we again omit the details for brevity.
	\end{proof}
	
	Once again, we will also need to apply (pointed) Schauder estimates in the case where $\gamma \le 0$. As in the non-pointed variant of the result that follows, the proof of estimates in this regime follow almost line-by-line as in the $\gamma > 0$ case by replacing applications of the bound coming from the Pointed Reconstruction Theorem (Corollary~\ref{theo:pointedReconstruction}) with an assumed bound on candidates for the reconstruction operator. Therefore we will provide only a statement of our Schauder estimate in this case.
	
	\begin{definition}\label{def: pointed reconstruction candidate}
		For $\gamma, \nu \in \bR$, $x \in \bR^d$ and $p \in [1, \infty]$, a candidate for the pointed reconstruction of $f \in \CD_{p}^{\gamma, \nu; x}$ is a distribution $\CR f$ that is a candidate for the reconstruction of $f$ as an element of $\CD_{p, \infty}^\gamma$ such that the constant $C(f; B_x)$ for $f$ in Definition~\ref{def: reconstruction candidate} can be taken such that
		
		\begin{equs}
			|\scal{\CR f, \phi_x^\lambda}| & \lesssim C(f; B_x) \lambda^{\nu - |\fs|/p}
			\\
			\| \scal{\CR f - \Pi_y f(y), \phi_y^\delta} \|_{\lambda, x, p; dy} & \lesssim C(f; B_x) \lambda^{\nu - \gamma} \delta^\gamma.
		\end{equs}
		
		A candidate for the pointed reconstruction operator on $\CD_{p}^{\gamma, \nu; x}$ is then a map $\CR: \CD_p^{\gamma, \nu; x} \to \CD'(\bR^d)$ such that for every $f \in \CD_p^{\gamma, \nu; x}$, $\CR f$ is a candidate for the pointed reconstruction of $f$.
	\end{definition}
	
	Given $\gamma \leq 0$, $\nu \in \bR$ and a candidate $\CR f$ for the reconstruction of $f \in \CD_p^{\gamma, \nu; x}$ the formula for $\CK_{\gamma, \nu}^{x, p} f$ in Definition~\ref{def: pointed integration} makes perfect sense and we will take it as a definition of the pointed abstract integration in this setting. With this in mind, we have the following result.
	
	\begin{theorem}\label{theo:schauder candidate}
		Let $\gamma \leq 0$, $\nu \in \bR$ and $p \in [1,\infty]$. Fix models $Z, \bar{Z}$.
		
		If $\CR f$ is a candidate for the pointed reconstruction of $f \in \CD_p^{\gamma, \nu;x}(Z)$ then $\CK_{\gamma, \nu}^{x,p} f \in \CD_p^{\gamma + \beta, \nu + \beta; x}(Z)$ and further
		\begin{equs}
			\$ \CK_{\gamma, \nu}^{x, p} f \$_{p, \gamma + \beta, \nu + \beta; x} \lesssim C(f; B_x) +  \$f\$_{p, \gamma, \nu; x} \left ( 1+\| Z \|_{\gamma; B_x} \right )^2
		\end{equs}
		where $\beta$ is the degree to which the corresponding kernel is regularising.
		
		Additionally if $\bar{\CR} \bar{f}$ is a candidate for the pointed reconstruction of $\bar{f} \in \CD_p^{\gamma, \nu;x}(\bar{Z})$ such that
		\begin{equs}
			& |\scal{\CR f - \bar{\CR} f, \phi_x^\lambda}| \lesssim C(f,\bar{f}) \lambda^{\nu - |\fs|/p}
			\\
			& \| \scal{ \CR f - \bar{\CR} \bar{f} - \Pi_y f(y) + \bar{\Pi}_y \bar{f}(y), \varphi_y^\delta } \|_{\lambda, x, p; dy} \lesssim C(f, \bar{f}) \lambda^{\nu - \gamma} \delta^\gamma
		\end{equs}
		then
		\begin{equs}
			\$ \CK_{\gamma, \nu}^{x, p} f, \bar{\CK}_{\gamma, \nu}^{x, p} \bar{f} \$_{p, \gamma + \beta, \nu + \beta; x} \lesssim C(f, \bar{f}) + \$f, \bar{f} \$_{p, \gamma, \nu; x} + \|Z; \bar{Z}\|_{\gamma; B_x}
		\end{equs}
		uniformly over modelled distributions $f, \bar f$ and models $Z, \bar Z$ such that $\$f\$_{p, \gamma, \nu; x} + \$\bar{f}\$_{p, \gamma, \nu; x} + \|Z\|_{\gamma, B_x} + \|\bar{Z}\|_{\gamma; B_x} + C(f; B_x) + C(\bar f; B_x) \le C$.
	\end{theorem}
	
	\subsection{Modelled Distributions Measured in Negative Sobolev Norm}\label{section: Modelled Distributions with Sobolev Coefficients}

	Unfortunately, whilst the tools of the last section will be good enough for us to obtain uniform bounds on renormalised models built from mollifications of a fixed driving noise, they are not quite good enough to obtain convergence of those models.
	
	The problem that propagates throughout our constructions is essentially that if $\varrho_\varepsilon$ is a mollifier at scale $\varepsilon$ then whilst $\varrho_\varepsilon \ast f \to f$ as $n \to \infty$ for any fixed $f$ in $L^p$, this convergence is not uniform over the $L^p$ unit ball. Since that is the kind of convergence we would require (once lifted to the level of modelled distributions) our previous framework runs into a problem.
	
	However, it is \emph{nearly} the case that we do have the desired convergence. We mean this in the sense 
	that, as usual in this kind of situation, uniform convergence does occur in any sensible weaker topology,
	for example $\sup_{\|f \|_{L^p} = 1} \| f - \varrho_\varepsilon \ast f \|_{\CB_{p,p}^{- \kappa}} \to 0$ as $\varepsilon \to 0$ for every $\kappa > 0$.
	
	To this end,  we introduce in this section another variant of our reconstruction and integration results which will allow us to track the behaviour of Besov and pointed modelled distributions through negative Sobolev norms of their coefficients. Since these negative Sobolev norms are controlled by $L^p$-norms, as sets the spaces we work with will coincide with the standard Besov and pointed spaces respectively. The interesting addition will be that we will derive a number of results extending the results of the previous sections to obtain bounds on abstract integrals and reconstructions of such modelled distributions that depend on negative regularity Sobolev norms of coefficients of the modelled distributions involved.
	
	We assume throughout this section that for $\zeta \in \CA$, we have a distinguished orthonormal basis $B_\zeta$ of $\CT_\zeta$. In our application of these results we will be interested in the case of the reduced regularity structure so that $B_\zeta$ will be given by the set of trees of degree $\zeta$, however our results here are not limited to that structure. 
	
	We begin with the setting of Besov modelled distributions.
	
	\begin{definition}\label{def: sob mod norm}
		For $f \in \CD_p^\gamma$, let
		\begin{equs}
			\|f\|_{\gamma, - \kappa, p; \ck} \eqdef \sup_{\zeta < \gamma} \sup_{\tau \in B_\zeta} \|f_\tau \|_{\CB_{p,p}^{-\kappa}(\ck)}.
		\end{equs}
		where $f_\tau$ is the coefficient of $f$ in direction $\tau$.  
	\end{definition}
	
	Our strategy throughout all the proofs of this subsection will be to obtain bounds on norms of very bad regularity on the objects under consideration that depend on $\|\cdot\|_{\gamma, - \kappa, p; \ck}$ and then use interpolation to obtain a bound in a space of better regularity that will depend on a smaller power of this norm, which will be sufficient for our purposes later.
	
	Such proofs are unfortunately necessarily notationally rather messy when written out in full, but are also all conceptually very similar to each other. As a result, we will provide one proof that demonstrates how the interpolation works in some detail and then in the other proofs will omit the details of the interpolation step.
	
	Throughout this subsection, we will let $Z = (\Pi, \Gamma)$, $\bar{Z} = (\bar{\Pi}, \bar{\Gamma})$ denote an arbitrary pair of admissible models. 
	
	We first aim to establish several results that will allow us to control the reconstruction operator in terms of this Sobolev-type norm. In the result that follows, we obtain control on the $\CB_{p,\infty}^{\alpha - \varepsilon}$ seminorms of the reconstruction. We remark that it would have sufficed to consider Besov seminorms of much worse regularity since we only use this result as input for a further interpolation argument. However obtaining the regularity we list below requires no additional effort.
	
	\begin{proposition}\label{prop: firstSobReconstr}
		Fix models $Z = (\Pi, \Gamma), \bar{Z} = (\bar{\Pi}, \bar{\Gamma})$ and parameters $p \in [1, \infty]$ and $\gamma \in \bR_+ \setminus \bN$. Then for any sufficiently small $\kappa > 0$, there exists a $\theta > 0$ such that for $f_1 \in \CD_p^\gamma(Z), f_2 \in \CD_{p}^\gamma(\bar{Z})$
		\begin{equs}
			\| \CR f_1 \|_{\CB_{p, \infty}^{\alpha - \varepsilon}(\ck)} \lesssim \|\Pi\|_{\gamma; \ck} \|f_1\|_{\gamma, - \kappa, p; \bar{\ck}}^\theta \$ f_1 \$_{p, \gamma; \ck}^{1- \theta}
		\end{equs}
		Additionally, uniformly over $\$f_i\$_{p, \gamma; \bar{\ck}} + \|\Gamma\|_{\gamma; \bar{\ck}} + \|\bar\Gamma\|_{\gamma; \bar{\ck}} + \|\Pi\|_{\gamma; \bar{\ck}} + \|\bar{\Pi}\|_{\gamma;\bar{\ck}} \le C$,
		\begin{equs}
			\| \CR f_1 - \bar{\CR} f_2 \|_{\CB_{p, \infty}^{\alpha - \varepsilon}(\ck)} & \lesssim \|f_1 - f_2\|_{\gamma, -\kappa, p; \bar{\ck}}^\theta + \|\Pi - \bar{\Pi}\|_{\gamma; \bar{\ck}}^\theta + \|\Gamma - \bar{\Gamma}\|_{\gamma; \bar{\ck}}^\theta.
		\end{equs}
		
	\end{proposition}
	
	\begin{proof}
		We begin with the case of a single model. Since this proof is otherwise notationally dense, we will track only the dependency on the Sobolev type norm explicitly and will absorb the other terms into the implicit constant.
		
		We write $f \eqdef f_1$ for notational convenience. We note that it follows from the proof of \cite[Theorem 3.2]{BL21} that if $\phi$ is as in Definition~\ref{def: semigroup kernel} then we can write the $\CB_{p,\infty}^{\alpha - \varepsilon}(\ck)$ norm of $\CR f$ as
		\begin{equ}
		\sup_{n \ge 0} 2^{n (\alpha - \varepsilon)} \Big \| \sup_{\eta \in \CB^r} \left | A_n(x) \right | \,   \Big \|_{L^p(\ck)}
		\end{equ}
		where
		\begin{equ}
			A_n(x) \eqdef \int \Pi_y f(y) (\phi_y^n) \eta_x^n(y) \, dy + \sum_{k = n}^\infty \int \Pi_y f(y) (\phi_z^{k+2}) \eta_x^n(y) R^k(y - z) \, dy \, dz
		\end{equ}
	 	and $R = \rho^1 - \rho^1 \ast \rho$. 
	
		To bound the term corresponding to the first term of $A_n$, we first note that applying the definition of the model immediately yields a bound of  $2^{-n \alpha}$. However this bound has no dependency on $\|f\|_{\gamma, -\kappa, p; \bar{\ck}}$. 
		
		In order to obtain a bound that does have such a dependency we write
		\begin{equs}
			 \sup_{\eta \in \CB^r} \int \Pi_y f(y)(\phi_y^n) \eta_x^n(y) dy =  \sup_{\eta \in \CB^r} \sum_{\zeta < \gamma} \sum_{\tau \in B_\zeta} \int f_\tau(y) \Pi_y \tau (\phi_y^n) \eta_x^n(y) dy.
		\end{equs}
		
		We define $\gap(\gamma)$ to be the minimum distance between distinct elements of the degree set $\CA_\gamma$. We will assume that $\kappa < \gap(\gamma)$.
We also recall that, by the scaling properties of Besov spaces,
		$
			\sup_{\eta \in \CB_r} \| \eta_x^n \|_{\CB_{p', p'}^\kappa(\ck)} \lesssim 2^{n \kappa + n|\fs|/p}
		$,
		where $p'$ is the H\"older conjugate of $p$. 
		
		Further, for $y, \bar{y} \in B(x, 2^{-n})$, we have that
		\begin{equs}
			|\Pi_y \tau (\phi_y^n) - \Pi_{\bar{y}} \tau (\phi_{\bar{y}}^n)| & \le |(\Pi_y - \Pi_{\bar{y}})\tau(\phi_y^n)| + |\Pi_{\bar{y}} \tau (\phi_y^n - \phi_{\bar{y}}^n)|
			\\
			& = |\Pi_y (\tau - \Gamma_{y \bar{y}} \tau)(\phi_y^n)| + |\Pi_{\bar{y}} (\phi_y^n - \phi_{\bar{y}}^n)|\;.
		\end{equs}
		The first term on the right-hand side is bounded by a term of order $2^{-n \alpha}|y - \bar{y}|^{\gap(\gamma)}$. For the second term, we note that $|\phi_y(x) - \phi_{\bar{y}}(x)| \le |y - \bar{y}|$, and
		similarly for its derivatives. Therefore, there is a $\psi \in \CB^r$ such that $\phi_y^n - \phi_{\bar{y}}^n = |y-\bar{y}| \psi_{\bar{y}}^n$. It then follows that $|\Pi_{\bar{y}} \tau (\phi_y^n - \phi_{\bar{y}}^n)| \lesssim 2^{-n \alpha}|y-\bar{y}|$.
		
		Hence, $y \mapsto \Pi_y \tau(\phi_y^n)$ is H\"older continuous of order $\gap(\gamma)$ with a norm of order $2^{-n \alpha}$ on $B(x, 2^{-n})$.
		By a standard multiplication result for $\CB_{p,q}^\gamma$ spaces (see e.g. \cite[Theorem 3.11]{BL21}, we obtain that $y \mapsto \scal{\Pi_y \tau, \phi_y^n} \eta_x^n(y) \in \CB_{p',p'}^\kappa$ with a norm of order $2^{ n ( \kappa - \alpha + |\fs|/p)}$. 
		
		Hence, since the obvious bilinear map on $\CB_{p',q'}^{-\alpha} \times \CD(\bR^d)$ extends to a bounded bilinear map on $\CB_{p',q'}^{-\alpha} \times \CB_{p,q}^\alpha$, we have that
		\begin{equs}
			\Bigg | \int f_\tau(y) \scal{\Pi_y \tau, \phi_y^n} \eta_x^n(y) dy \Bigg | \lesssim 2^{n (\kappa - \alpha + |\fs|/p)} \|f\|_{\gamma, - \kappa, p; \bar{\ck}}.
		\end{equs}
		
		By choosing $\theta$ such that $\theta (\kappa + |\fs|/p) < \varepsilon$ and interpolating between our order $2^{-n \alpha}$ bound and this bound we obtain
		\begin{equs}
			2^{n(\alpha - \varepsilon)} \left \| \int \Pi_y f(y)(\phi_y^n) \eta_x^n(y) dy \right \|_{L^p(\ck; dx)} \lesssim \|f\|_{\gamma, - \kappa, p; \bar \ck}^\theta.
		\end{equs}
		We are then only left to consider the sum in $k$.
		
		First, we obtain a bound that does not depend on our Sobolev norm which we will use for interpolation. We write
		\begin{equs}\label{eq: Sobolev Reconstruction 1}
			& \left | \int \Pi_yf(y)(\phi_z^{k+2}) R^k(y-z) \eta_x^n(y) dy dz \right | \\ &\le \left | \int \Pi_z f(z) (\phi_z^{k+2}) R^k \ast \eta_x^n(z) dz \right | + \left | \int \Pi_z (f(z) - \Gamma_{zy} f(y))(\phi_z^{k+2}) R^k(y-z) \eta_x^n(y) dy dz \right |.
		\end{equs}
		
		We then have
		\begin{equ}
			\left | \int \Pi_z f(z)(\phi_z^{k+2}) R^k \ast \eta_x^n(z) dz \right | \lesssim 2^{- k\alpha} \|R^k \ast \eta_x^n \|_{L^1}.
		\end{equ}
		Since convolution with $R^k$ annihilates polynomials of up to a suitably high degree, we can replace $\eta_x^n$ with its Taylor jet to degree $r$ to find that  $\| R^k \ast \eta_x^n \|_{L^1} \lesssim 2^{-kr} 2^{nr}$ uniformly in $\eta$. In total, we obtain that
		\begin{equ}
			\left | \int \Pi_z f(z)(\phi_z^{k+2}) R^k \ast \eta_x^n(z) dz \right | \lesssim 2^{-k (r + \alpha)} 2^{n r}.
		\end{equ} 
		For the second term on the right hand side of \eqref{eq: Sobolev Reconstruction 1}, applying the definition of the model and of $\$f\$_{p, \gamma; \ck}$ immediately yields a bound of order $2^{-k \gamma}$. 
		
		Therefore 
		\begin{equs}
			\left \| \sup_\eta \int \Pi_y f(y) (\phi_z^{k+2}) \eta_x^n(y) R^k(y-z) dy dz \right \|_{L^p(\ck)} \lesssim 2^{-k \gamma} + 2^{-k (r+ \alpha)} 2^{nr}
		\end{equs}
		
		We now turn to obtaining a bound that depends on our Sobolev type norm. Again, we expand $f(y)$ in terms of our basis to write
		\begin{equs}
			\Bigg | \int \Pi_y f(y) (\phi_z^{k+2}) & R^k(y-z) \eta_x^n(y) dy dz \Bigg | \\ & \le \sum_{\zeta < \gamma} \sum_{\tau \in B_\zeta} \left | \int f_\tau(y) \Pi_y \tau (\phi_z^{k+2}) R^k(y-z) \eta_x^n(y) dy dz \right |
		\end{equs}
		
		By a straightforward estimate, one finds that $\|R_z^k\|_{\CB_{\infty, \infty}^\kappa(B)} \lesssim 2^{k(\kappa + |\fs|)}$ and $\| \eta_x^n \|_{\CB_{p', q'}^\kappa(B)} \lesssim 2^{n (\kappa + |\fs|/p)}$ where $B \eqdef B(z, 2^{-k})$.
		
		Additionally, for $y, \bar{y} \in B$, we have that
		\begin{equ}
			|\scal{\Pi_y \tau - \Pi_{\bar{y}} \tau, \phi_z^{k+2}}| = |\scal{\Pi_y [\tau - \Gamma_{y \bar y} \tau], \phi_z^{k+2}}| \lesssim 2^{-k \alpha} |y-\bar{y}|^{\gap(\gamma)}
		\end{equ}
		so that $y \mapsto \Pi_y \tau(\phi_y^{k+2})$ is $\gap(\gamma)$-H\"older with a norm of order $2^{-k \alpha}$ on $B$. Therefore, again applying Besov multiplication and duality results and not being careful to obtain the sharpest estimates since we will anyway interpolate these bounds, we get that
		\begin{equs}
			\Bigg | \int f_\tau(y) \Pi_y \tau (\phi_z^{k+2}) & R^k(y-z) \eta_x^n(y) dy dz \Bigg |  \\ & \lesssim \| f\|_{\gamma, -\kappa, p; \ck} \int_{|z-x| \lesssim 2^{-n}} 2^{-k \alpha} 2^{k (2 \kappa + |\fs|(1 + 1/p))} dz 
			\\
			& \lesssim 2^{k (2 \kappa - \alpha + |\fs|(1 + 1/p))} \| f\|_{\gamma, -\kappa, p; \bar \ck}
		\end{equs}
		which yields the same bound on the $L^p(\ck; dx)$ norm up to an irrelevant multiplicative constant. 
		
		We define $c = 2 \kappa + |\fs|(1 + 1/p)$. Interpolating between our bounds, we have that for every $\theta \in [0,1]$,
		\begin{equs}
			\Bigg \| \sup_\eta \int \Pi_y f(y) (\phi_z^{k+2}) & \eta_x^n(y) R^k(y-z) dy dz \Bigg \|_{L^p(\ck)} \\ & \lesssim \left [2^{-k \gamma(1- \theta)} 2^{ \theta k(c - \alpha)} + 2^{-k \alpha} 2^{(1- \theta) k r} 2^{k \theta c} \right ] \|f\|_{\gamma, - \kappa, p; \bar \ck}^\theta.
		\end{equs}
		By choosing $\theta$ to be sufficiently small, we can ensure that the term in brackets is summable in $k$ to a term of order $C(n)$ such that $|C(n) 2^{n(\alpha - \varepsilon)}| \lesssim 1$ which yields the desired result. 
		
		The case of two models follows via essentially the same techniques but is even more notationally dense and so we omit its proof.
	\end{proof}
	
	\begin{theorem}\label{theo: SobCoeffReconstr}
		Fix models $Z = (\Pi, \Gamma), \bar{Z} = (\bar{\Pi}, \bar{\Gamma})$ and parameters $p \in [1, \infty]$, $\alpha \le 0$ and $\gamma \in \bR_+ \setminus \bN$ such that $0 < \gamma < \alpha + |\fs|/p$. Then for any $\kappa > 0$, there exists a $\theta > 0$ such that for $f_1 \in \CD_p^\gamma(Z)$ valued in a sector $V$ of regularity $\alpha$, the following bound holds
		\begin{equs}
			|\CR f_1(\psi_x^\lambda)| \lesssim \lambda^{\gamma - |\fs|/p - \kappa} \|f_1\|_{\gamma, - \kappa, p; \bar{\ck}}^\theta \$f_1\$_{p, \gamma; \bar{\ck}}^{1- \theta} \|\Pi\|_{\gamma; \ck} ( 1 + \| \Gamma \|_{\gamma; \ck})
		\end{equs}
		uniformly over $\lambda \in (0,1]$ and $\psi \in \CB^r$.
		
		Given $f_2 \in \CD_p^\gamma(\bar{Z})$ valued also in $V$, we also have the bound
		\begin{equs}
			|\CR f_1(\psi_x^\lambda) - & \bar{\CR} f_2(\psi_x^\lambda)| \\ & \lesssim \lambda^{\gamma - |\fs|/p - \kappa} \left ( \|f_1 - f_2\|_{\gamma, -\kappa, p; \bar{\ck}}^\theta + \|\Pi - \bar{\Pi}\|_{\gamma; \bar{\ck}}^\theta + \|\Gamma - \bar{\Gamma}\|_{\gamma; \bar{\ck}}^\theta \right )
		\end{equs}
		uniformly over $\$f_i\$_{p, \gamma; \bar{\ck}} + \|\Gamma\|_{\gamma; \bar{\ck}} + \|\bar\Gamma\|_{\gamma; \bar{\ck}} + \|\Pi\|_{\gamma; \bar{\ck}} + \|\bar{\Pi}\|_{\gamma;\bar{\ck}} \le C$.
		
		In the case where $\alpha = 0$, let $\bar \alpha = \min (\CA_V \setminus \bN)$. Then for $0 < \gamma < \bar \alpha + |\fs|/p$, assuming that $\gamma - |\fs|/p \not \in \bN$, the same bounds hold uniformly over $\psi \in \CB_{\lfloor \bar \alpha \rfloor}^r$.
	\end{theorem}
	\begin{proof}
		By Proposition~\ref{prop: firstSobReconstr} and Besov embedding, there exists a $C > 0$ such that 
		\begin{equ}
			\sup_{x \in \ck} \sup_{\psi \in \CB_r} \CR f(\psi_x^\lambda) \lesssim \lambda^{-C} \|f\|_{\gamma, - \kappa, p; \bar{\ck}}^{\bar{\theta}} \$ f \$_{p, \gamma; \ck}^{1- \bar{\theta}} \|\Pi\|_{\gamma; \ck} ( 1 + \| \Gamma \|_{\gamma; \ck})
		\end{equ}
		for sufficiently small $\bar{\theta}, \varepsilon > 0$, uniformly over $\lambda \in (0,1]$. Interpolating this bound with the bound of Theorem~\ref{theo:betterReconstr} immediately yields the result for a single model. Since the only difficulty here is in notational complexity and the ideas of such an interpolation were already demonstrated in the proof of Proposition~\ref{prop: firstSobReconstr} we omit the details.
	\end{proof}
	
	We will also be interested in applying Schauder estimates to modelled distributions with coefficients that are small in a negative Sobolev norm. To this end, we have the following result.
	
	\begin{proposition}\label{prop: SobSchauder}
		Fix models $Z = (\Pi, \Gamma)$ and $\bar{Z} = (\bar{\Pi}, \bar{\Gamma})$ and suppose that $f_1 \in \CD_p^\gamma(Z), f_2 \in \CD_p^\gamma(\bar{Z})$ where $\gamma \in \bR_+ \setminus \bN$. If $\CK_\gamma, \bar{\CK}_\gamma$ are the abstract integration operators for the models $Z, \bar{Z}$ defined in \cite[Equation (5.15)]{Hai14}
		then 
		\begin{equs}
			\| \CK_\gamma f_1 \|_{\gamma + \beta, -\kappa, p; \ck} & \lesssim \left (\|f_1\|_{\gamma, -\kappa,p; \bar{\ck}} + \|f_1\|_{ \gamma, -\kappa,p; \bar{\ck}}^\theta \$ f_1 \$_{p, \gamma; \bar{\ck}}^{1- \theta} \right ) (1+ \|Z\|_{\gamma; \bar{\ck}})
			\\
			\| \CK_\gamma f_1 - \bar{\CK}_\gamma f_2 \|_{\gamma + \beta, -\kappa, p; \ck} & \lesssim \|\Pi - \bar{\Pi}\|_{\gamma; \bar{\ck}} + \|\Gamma - \bar{\Gamma}\|_{\gamma; \bar{\ck}} + \|f_1 - f_2\|_{\gamma, -\kappa,p; \bar{\ck}} \\ & + \|f_1 - f_2\|_{\gamma, -\kappa,p; \bar{\ck}}^\theta \|f_1 - f_2\|_{p, \gamma; \bar{\ck}}^{1 - \theta} 
		\end{equs}
		where the second bound has an implicit constant which can be chosen uniformly over $\$f_i\$_{p, \gamma; \bar{\ck}} + \|\Gamma\|_{\gamma; \bar{\ck}} + \|\bar\Gamma\|_{\gamma; \bar{\ck}} + \|\Pi\|_{\gamma; \bar{\ck}} + \|\bar{\Pi}\|_{\gamma;\bar{\ck}} \le C.$
	\end{proposition}
	\begin{proof}
		As usual, the bounds at non-integer homogeneities follow immediately from the definition so that we are left only to consider the terms valued in the polynomial part of the regularity structure. 
		
		We note that the component of $\CK_\gamma f_1$ in the direction $X^k$ for $|k|_\fs < \gamma$ can be written as
		\begin{equs}
			\sum_{n \geq 0} \Big ( \sum_{\zeta > |k|_\fs - \beta} \scal{\Pi_x \CQ_\zeta f_1 (x), \partial^k K_n(x - \cdot)} + \scal{\CR f_1 - \Pi_x f_1(x), \partial^k K_n(x- \cdot)} \Big ).
		\end{equs}
		We apply the definition of a model to write
		\begin{equs}
			\left \| \scal{\Pi_x \CQ_\zeta f_1(x), \partial^k K_n(x - \cdot)} \right\|_{\CB_{p,p}^{-\kappa}(\ck)} \le \sum_{\tau \in B_\zeta} \| \CQ_\tau f_1 \|_{\CB_{p,p}^{-\kappa}(\ck)} 2^{-n(\zeta + \beta - |k|_\fs)}
		\end{equs}
		to see that the first term has a sum in $n$ of the correct order.
		
		For the second term, we must again interpolate between two bounds. On the one hand, since the $\CB_{p,p}^{-\kappa}(\ck)$ norm is controlled by the $L^p(\ck)$ norm, Theorem~\ref{theo: Reconstruction} immediately yields a bound of order $2^{-n(\gamma + \beta - |k|_\fs)}$. On the other hand, we can treat the reconstruction operator and the model appearing in this term separately. 
		
		Proposition~\ref{prop: firstSobReconstr} yields the bound
		$${\|\scal{\CR f_1, \partial^k K_n(x - \cdot)} \|_{L^p(\ck)}} \lesssim 2^{-n(\alpha - \varepsilon + \beta - |k|_\fs)} \| f \|_{\gamma, -\kappa, p; \bar{\ck}}^{\bar{\theta}}.$$
		Further, the same treatment as used to control the first term yields a bound of order
		\begin{equ}
			\|f\|_{\gamma, - \kappa, p; \bar{\ck}} 2^{-n (\alpha + \beta - |k|_\fs)}
		\end{equ} 
		on $\| \scal{\Pi_x f_1(x), \partial^k K_n(x - \cdot)} \|_{\CB_{p,p}^{- \kappa}(\ck)}$.
		
		Since $\gamma + \beta - |k|_\fs > 0$, interpolating between these bounds will yield a bound that is summable in $n$ to a term of the correct order.
		
		The bounds in the case of multiple models then follow in much the same way.
	\end{proof}
	
	We will also require the pointed analogues of all of these statements. Again, the proofs here would be tedious but are conceptually no different to the non-pointed versions. As a result we only provide very brief sketches of the proofs for the following two results.
	
	\begin{theorem}\label{theo: PointedSobReconstr}
		Let $\gamma \in \bR_+ \setminus \bN$ be such that $0 < \gamma < \alpha + |\fs|/p$ and suppose that 
		$f \in \CD^{\gamma,\nu,x}_{p}(Z)$ takes values in a sector of regularity $\alpha \le 0$. Then writing $B_x$ for the closed ball around $x$ of radius $2$, for every $\varepsilon > 0$ sufficiently small, there exists $\theta > 0$ such that $$|\scal{\CR f,\psi_x^\lambda}| \lesssim \lambda^{\nu-|\fs|/p - \varepsilon} \$ f \$_{p, \gamma, \nu; x}^{1- \theta} \|f\|_{\gamma, -\kappa, p; B_x}^\theta \| \Pi \|_{\gamma; B_x} (1 + \| \Gamma \|_{\gamma;B_x}),$$ uniformly over $\lambda \in (0,1]$ and $\psi \in \CB^r$.
		
		Given a second model $\bar{Z}$ and $\bar{f} \in \CD^{\gamma,\nu,x}_{p}(\bar{Z})$ taking values in the same sector as $f$, we have that
		\begin{equs}
			\lambda^{|\fs|/p-\nu + \varepsilon} | & \scal{\CR f - \bar{\CR} \bar{f},\psi_x^\lambda}|  \\ & \lesssim \|f_1 - f_2\|_{\gamma, -\kappa, p; B_x}^\theta + \|\Pi - \bar{\Pi}\|_{\gamma; B_x}^\theta + \|\Gamma - \bar{\Gamma}\|_{\gamma; B_x}^\theta.
		\end{equs}
		uniformly over $\$f_i\$_{p, \gamma, \nu; x} + \|\Gamma\|_{\gamma; \bar{\ck}} + \|\bar\Gamma\|_{\gamma; \bar{\ck}} + \|\Pi\|_{\gamma; \bar{\ck}} + \|\bar{\Pi}\|_{\gamma;\bar{\ck}} \le C.$
		
		In the case where $\alpha = 0$, if $\bar \alpha = \min(\CA_V \setminus \bN)$ and $0 < \gamma < \bar \alpha + |\fs|/p$ then the same bounds hold uniformly over $\psi \in \CB_{\lfloor \bar \alpha \rfloor}^r$.
	\end{theorem}
	\begin{proof}
		This follows by interpolating between the bounds of Proposition~\ref{prop: firstSobReconstr} and Theorem~\ref{theo:pointedReconstruction}.
	\end{proof}
	
	\begin{theorem}\label{theo: PointedSobSchauder}
		Fix models $Z = (\Pi, \Gamma)$ and $\bar{Z}= (\bar{\Pi}, \bar{\Gamma})$ and for $p \in [1, \infty]$ let $\CK_{\gamma, \nu}^{x,p}, \bar{\CK}_{\gamma, \nu}^{x,p}$ be the corresponding pointed integration operators. Suppose that $f_1 \in \CD_p^{\gamma, \nu; x}(Z)$ and $f_2 \in \CD_p^{\gamma, \nu; x}(\bar Z)$ are valued in the sector $V$, $\gamma, \gamma - |\fs|/p \not \in \bN$ and that $0 < \gamma < \bar \alpha + |\fs|/p$ for $\bar \alpha = \min(\CA_V \setminus \bN)$.
		
		Then we have the bounds
		\begin{equs}
			\| \CK_{\gamma, \nu}^{x,p} f_1 \|_{\gamma + \beta,-\kappa, p; \ck} & \lesssim \left (\|f_1\|_{\gamma,-\kappa,p; \bar{\ck}} + \|f_1\|_{\gamma,-\kappa,p; \bar{\ck}}^\theta \$ f_1 \$_{p, \gamma; \bar{\ck}}^{1- \theta} \right ) (1+ \|Z\|_{\gamma; \bar{\ck}})
			\\
			\| \CK_{\gamma, \nu}^{x,p} f_1 - \bar{\CK}_{\gamma, \nu}^{x,p} f_2 \|_{\gamma + \beta,-\kappa, p; \ck} & \lesssim \|\Pi - \bar{\Pi}\|_{\gamma; \bar{\ck}} + \|\Gamma - \bar{\Gamma}\|_{\gamma; \bar{\ck}} + \|f_1 - f_2\|_{\gamma, -\kappa,p; \bar{\ck}} \\& \quad + \|f_1 - f_2\|_{\gamma, -\kappa,p; \bar{\ck}}^\theta \|f_1 - f_2\|_{-\kappa,p; \bar{\ck}}^\theta
		\end{equs}
		where the second bound has an implicit constant which can be chosen uniformly over $\$f_i\$_{p, \gamma; \bar{\ck}} + \|\Gamma\|_{\gamma; \bar{\ck}} + \|\bar\Gamma\|_{\gamma; \bar{\ck}} + \|\Pi\|_{\gamma; \bar{\ck}} + \|\bar{\Pi}\|_{\gamma;\bar{\ck}} \le C.$
	\end{theorem}
	\begin{proof}
		By Proposition~\ref{prop: SobSchauder}, it suffices to consider $\| T_{\nu, p}^x f \|_{\CB_{p,p}^{-\kappa}(\ck)}$. This term is valued in the polynomial structure and has coefficients that depend on pointwise evaluations of the reconstruction of $f$. Therefore control on this term follows in a very similar fashion as to control on the terms considered in the proof of Proposition~\ref{prop: SobSchauder} by making use of Theorem~\ref{theo: PointedSobReconstr} at the appropriate point.
	\end{proof}

	\section{Fr\'echet Derivatives of Renormalised Models}\label{sec: Frechet}
	
	In order to bound the derivative term on the right-hand side of the spectral gap inequality applied to $\Pi_x \tau(\psi)$, we will want to describe the Fr\'echet derivative of a model with respect to the underlying noise 
	in terms of a pointed modelled distribution so as to obtain bounds from Theorem~\ref{theo:pointedReconstruction}.
	
	In this section we recursively define a family of pointed modelled distributions for this task. Essentially this is done by observation in the base case, noting that there is only one plausible definition for a tree of the form $\tau = \CI^\ft \sigma$, and then by postulating that the obvious analogue of the Leibniz rule holds for products.
	
	We will also introduce a second family of modelled distributions on a larger regularity structure. The idea here is that whilst the first family is well-adapted to the task of obtaining analytic bounds, it is more difficult to see that its reconstruction really coincides with the Fr\'echet derivative of the model. The second family of modelled distributions is better suited to this task since the augmented regularity structure we construct is exactly formed so as to allow us to encode the action of shifts in the driving noise. In doing this we follow the sketch given in \cite[Section 5]{Jonathan}, however there is an error in that work that we must correct for.
	
	The goal will then be to relate these two constructions so as to see that the former family of modelled distributions does suitably describe the Fr\'echet derivative of the model.
	
	We will let $\gamma_\tau = \alpha_\tau + |\fs|/2 - n_\tau \bar \kappa$ where $\alpha_\tau$ is the lowest degree of all non-polynomial terms appearing in the smallest sector containing $\tau$, $n_\tau$ is the number of noise edges in $\tau$ and $\bar \kappa > 0$ is a sufficiently small constant (to be fixed later) such that there is no integer $k \in \bZ$ with $\gamma_\tau \le k < \alpha_\tau + |\fs|/2$. We also write $\deg_2 \tau = \deg \tau + |\fs|/2$
	and we adopt the shorthand $\gamma_\ft \eqdef \gamma_{\Xi_\ft}$. We adopt the convention that if $\tau \in \Tpoly$ then $\alpha_\tau = \infty$.
	
	\begin{remark}
		The inclusion of a $n_\tau \bar \kappa$ term in the definition of $\gamma_\tau$ is necessary since our 
		reconstruction result and hence also our Schauder estimate apply only to the regime $0 < \gamma < 
		\alpha_\tau + |\fs|/2$. Whilst in what follows it would seem natural to take $\gamma_\tau = \alpha_\tau 
		+ |\fs|/2$, this would prevent us from applying these results, so we must allow for a small 
		loss of regularity. Since we will only use analytic estimates for the resulting pointed modelled distribution around their distinguished point, this is not an issue.
	\end{remark}
	
	Given an element $\eta \in \CH \eqdef H^{- \reg}(\Lab_-)^* = \prod_{\ft \in \Lab_-} H^{\reg \ft}(\Lab_-)$, our goal is now to inductively define a collection $H^{x,\eta}_\tau$ of modelled distributions
	whose reconstruction will be seen to coincide with the derivative of the renormalised model
	$\Pi_x \tau$ in the direction $\eta$ of the underlying driving noise, so long as $\eta$ is sufficiently smooth. 
	
	Since we will eventually have in mind the case where a family of models is built from mollifications at dyadic scales of the underlying noise, we will write the construction with this in mind here.
	
	We first define modelled distributions $f^\tau_z$ 
	by $f^\tau_z(\bar z) = \Gamma_{\bar zz} \tau$ and then set
	\begin{equs}\label{e:baseH1}
		H^{x,\eta}_{\Xi_\ft; n}(y) & = \sum_{|k|_\fs < \gamma_\ft}{X^k \over k!} \Big [(D^k\varrho^n \ast \eta_\ft)(y) - P_x^{|\ft|_\fs}[D^k \varrho^n \ast \eta_\ft](y)\Big ] \\ \label{e:baseH2} 
		H^{x,\eta}_{\one; n}(y) & = H^{x,\eta}_{X_i; n}(y) = 0\;,
	\end{equs}
	where $P_x^a[f]$ denotes the Taylor polynomial of $f$ at base point $x$ up to degree $a$.
	
	We also set
	\begin{equ}[e:defProdH]
		H^{x,\eta}_{\tau\bar \tau; n}(y) = f^\tau_x(y)\,H^{x,\eta}_{\bar\tau; n}(y)
		+ H^{x,\eta}_{\tau; n}(y)\, f_x^{\bar\tau}(y)\;,
	\end{equ}
	and 
	\begin{equ}[e:int]
		H^{x,\eta}_{\CI^\ft \tau; n}(y) = \bigl(\CK_{\gamma_\tau, \deg_2 \tau; \ft}^{x, 2} H^{x,\eta}_{\tau; n}\bigr)(y)\;,
	\end{equ}
	in the case where $\gamma_\tau > 0$. We will handle integration in the case $\gamma_\tau \le 0$ by hand later, which will complete the recursive definition of this family of pointed modelled distributions.
	
	We remark that \eqref{e:defProdH} also implies the corresponding Leibniz rule for multiple products
	as a consequence of the fact that 
	$f^{\tau\bar \tau}_z = f^{\tau}_zf^{\bar \tau}_z$ and the associativity of our product.
	
	\begin{remark}
		In principle, planted trees of the form $\CI_k^\ft \tau$ for $k \in \bN^d \setminus \{0\}$ also show up in the reduced regularity structure. For such trees, we note that if we denote by $\CT_{\text{plant}}$ the linear subspace of $\CT$ generated by planted trees then there is a natural abstract gradient $D_i : V \to V$ on the sector $V = \CT_{\text{plant}} \sqcup \Tpoly$ given by setting $D_i X_j = \delta_{i,j}$ and extending to the polynomial structure by the Leibniz rule and by setting $D_i \CI_k^\ft \tau = \CI_{k + e_i}^\ft \tau$. We refer the reader to \cite[Section 5.4]{Hai14} for a definition of an abstract gradient. 
		
		The upshot of this construction in our case is that since $D_i : V_\alpha \to V_{\alpha - \fs_i}$ and commutes with the action of the structure group, it is immediate from the definitions that if $f \in \CD_p^\gamma$ then $D_i f \in \CD_p^{\gamma - \fs_i}$ (with a similar result in the pointed case). Additionally, since for admissible models we have that $\Pi_x D_i \tau = D_i \Pi_x \tau$, it is the case that $\CR D_i f = D_i \CR f$. Therefore, one can define $\pH{\CI_k^\ft \tau} \eqdef D^k \pH{\CI^\ft \tau}$ and all of our desired results automatically extend along this definition.
		
		As a result, we will assume without loss of generality that all planted trees are of the form $\CI^\ft$ for the remainder of this paper in order to simplify notation.
	\end{remark}
	
	We begin by verifying that it is indeed the case that $\pH{\Xi_\ft} \in \CD_2^{\gamma_\ft, \gamma_\ft; x}$.
	
	\begin{lemma}\label{lemma: H base check}
		Let $\ft \in \Lab_-$. Then for $\eta \in \CH$ we have that $H_{\Xi_\ft; n}^{x,\eta} \in \CD_2^{\gamma_\ft, \gamma_\ft; x}$. Further, we have the bound
		\begin{equ}
			\sup_{\|\eta\| \le 1} \$H_{\Xi_\ft; n}^{x, \eta} \$_{2, \gamma_\ft, \gamma_\ft; x} \lesssim 1
		\end{equ}
		where the norm appearing in the set over which the supremum is taken is the $\CH$ norm.
	\end{lemma}
	\begin{proof}
		Since $H_{\Xi_\ft; n}^{x, \eta} = 0$ when $\gamma_\ft < 0$, we may assume without loss of generality that $\gamma_\ft > 0$.
		We first note that $\pH{\Xi_\ft} \in \CD_2^{\gamma_\ft}$ with the desired control on the norm. This is essentially immediate from \cite[Proposition A.5]{BL21} in combination with the embeddings $\CB_{2,2}^\gamma \hookrightarrow \CB_{2, \infty}^\gamma \hookrightarrow \CB_{\infty, \infty}^{\gamma - |\fs|/2}$.
		
		Therefore we turn to considering the pointed bounds required in the definition of $\cD_2^{\gamma_\ft, \gamma_\ft}$. In fact, since $\gamma = \nu$, the pointed translation bound follows immediately from its non-pointed variant so that we are left only to consider the local bound.
		
		Here we treat the regimes $|\ft|_\fs > 0$ and $|\ft|_\fs < 0$ differently, though we again only demonstrate the bound at degree $0$. In the latter regime, we choose $p \in [1, \infty)$ such that $|\ft|_\fs + |\fs|/2 < |\fs|\frac{p-2}{2p} < \reg \ft$. Then, by H\"older's inequality, we write
		$$\|\rho^n \ast \eta_\ft\|_{\lambda, x} \le \|\rho^n \ast \eta_\ft \|_{L^p}^{1/2} \| 1_{B(x, \lambda)} \|_{L^{\frac{p}{p-2}}}^{1/2}.$$
		By Besov embedding, $\eta_\ft \in \CB_{2,2}^{\reg \ft} \subseteq \CB_{p,p}^{\reg \ft - |\fs|\frac{p-2}{2p}}$. Since this is a positive regularity Sobolev space by our choice of $p$, the above yields a bound of order $\lambda^{|\fs|\frac{p-2}{2p}}$ which is sufficient, again by our choice of $p$.
			
		It remains to consider the case $|\ft|_\fs > 0$. Here we use that $\eta_\ft \in \CB_{\infty, \infty}^{|\ft|_\fs}$ by Besov embedding so that $\| \rho^n \ast \eta_\ft - P_x^{|\ft|_\fs}[\rho^n \ast \eta_\ft]\|_{\lambda,x}$ can be estimated in the right way by using Taylor's theorem and applying an $L^\infty$ bound to the resulting integrand.
	\end{proof}
	
	Unfortunately, \eqref{e:baseH1}, \eqref{e:baseH2}, \eqref{e:defProdH} and \eqref{e:int} do not quite provide a complete definition of this family of modelled distributions since it is not guaranteed that $\gamma_\tau > 0$. As a result, it may be that \eqref{e:int} is not well-defined by an application of Theorem~\ref{theo:schauder}. In this case, we will want to make use of Theorem~\ref{theo:schauder candidate} which means that we will require an a priori candidate for the pointed reconstruction of $H_{\tau; n}^{x,\eta}$ for each $\tau$ such that $\gamma_\tau \le 0$. For this, we first identify the circumstances in which $\gamma_\tau$ can be non-positive.
	
	\begin{lemma}\label{lemma: bad tree identification}
		Suppose that $\gamma_\tau \le 0$. Then $\tau$ is of the form $\tau = \Xi_\ft \cdot X^k \cdot \big (\prod_{i = 1}^n \CI^{\fl_i} \sigma_i \big )$ where $|\ft|_\fs < - |\fs|/2$ and for each $i$, $\alpha_{\sigma_i} + |\fl_i|_\fs + |\ft|_\fs + |\fs|/2 > 0$. 
	\end{lemma}
	\begin{proof}
		If $\gamma_\tau \le 0$ then $\alpha_\tau < -|\fs|/2$. By Assumption~\ref{ass:alg}, this is only possible if there exists $\ft \in \Lab_-$ with $|\ft|_\fs < -|\fs|/2$ such that $\Xi_\ft$ lies in the smallest sector containing $\tau$. By definition of the map $\Delta$ this is only possible if $\tau$ is of the form specified in the statement.
		
		To obtain the inequality for $\alpha_{\sigma_i}$, note that there exists a subtree $\tilde{\sigma}_i$ of $\sigma_i$ such that $\alpha_{\sigma_i} + |\fl_i|_\fs = |\CI^{\fl_i} \tilde{\sigma}_i|_\fs$. It then follows that $\Xi_\ft \CI_k^{\fl_i} \tilde{\sigma}_i$ is a tree in the reduced regularity structure so that the inequality follows by Assumption~\ref{ass:alg}.
	\end{proof}
	
	Given this restriction on the kind of problematic trees that can arise, we have hope to provide a candidate for the pointed reconstruction by hand in these instances. We first show that the obvious candidate for the reconstruction of $\pH{\Xi_\ft}$ in the case where $\gamma_\ft \le 0$ (namely simply $\eta_\ft^n = \rho^n*\eta_\ft$) is suitable. We emphasise that in all of our discussion of candidates for the reconstruction in the case where $\gamma_\tau \le 0$, the bounds that we obtain will be independent of the choice of $n$.
	
	\begin{lemma}\label{lemma: eta candidate}
		Given $\ft \in \Lab_-$ such that $|\ft| \le - |\fs|/2$ and $\eta \in \CH$, $\eta_\ft^n = \varrho^n \ast \eta_\ft$ is a candidate for the $\CD_2^{\gamma_\ft, \gamma_\ft; x}$ pointed reconstruction of $\pH{\Xi_\ft}$. Furthermore, we have the bound
		\begin{equ}
			C(f; \ck) \lesssim \|\eta\|_{\CH}.
		\end{equ}
	\end{lemma}
	\begin{proof}
		The pointed bounds follow in this case from their non-pointed variants so that it suffices to show that $\eta_\ft^n$ is a candidate for the $\CD_{2}^{\gamma_\tau}$ reconstruction of $\pH{\Xi_\ft}$.
		
		Since $\gamma_\tau \le 0$, the desired reconstruction bound is of the form 
		\begin{equ}
			\| \scal{\eta_\ft^n, \phi_x^\lambda} \|_{L^p(\ck; dx)} \lesssim C(f, \ck) \lambda^{\gamma_\tau}.
		\end{equ} 
		Therefore the result is immediate from the embedding $\CB_{2,2}^{\gamma_\tau} \hookrightarrow \CB_{2, \infty}^{\gamma_\tau}$.
	\end{proof}
	
	Next, we turn to the more complex task of providing a candidate for the pointed reconstruction of each tree $\tau$ with more than one edge such that $\gamma_\tau \le 0$. Since multiplication with terms of the form $X^k$ is a straightforward operation, it will suffice for us to consider trees $\tau$ that have a vanishing node label at the root.
	
	\begin{lemma}\label{lemma: pointed candidate exists}
		Suppose $\gamma_\tau \le 0$ and $\tau = \Xi_{\ft} \cdot \left ( \prod_{i=1}^n \CI^{\fl_i} \sigma_i \right ) = \Xi_\ft \bar \tau$. Then there exists $\gamma > 0$ such that $f_x^{\Xi_\ft} \cdot \pH{\bar \tau} \in \CD_2^{\gamma, \deg_2 \tau; x}$. 
		Additionally, $\eta_\ft^n \cdot \Pi_x \bar \tau$ is a candidate for the $\CD_2^{\gamma_\tau, \nu_\tau; x}$ pointed reconstruction of $f_x^{\bar \tau} \cdot \pH{\Xi_\ft} = 0$. 
		
		In particular, it follows that $\CR (f_x^{\Xi_\ft} \cdot \pH{\bar \tau}) + \eta_\ft^n \cdot \Pi_x \bar \tau$ is a candidate for the $\CD_2^{\gamma_\tau, \deg_2 \tau; x}$ pointed reconstruction of $\pH\tau$, where $\CR$ is the reconstruction operator on $\CD_2^{\gamma, \deg_2 \tau; x}$. Furthermore, if $C(\pH{\tau})$ is the corresponding constant in Definition~\ref{def: pointed reconstruction candidate} then 
		\begin{equ}
			C(\pH{\tau}) \lesssim (\$\pH{\bar\tau}\$_{2, \gamma_{\bar \tau}, \deg_2 \bar \tau; x} + \|\eta\|_\CH) (1+ \|Z\|_{V_{\bar \tau}; B_x} ).
		\end{equ}
	\end{lemma}
	\begin{proof}
		The first statement is a straightforward consequence of Theorem~\ref{theo: pointed multiplication} and Lemma~\ref{lemma: bad tree identification}.
		
		Since the final statement is an immediate consequence of the first two, it remains to see that $\eta_\ft^n \cdot \Pi_x \bar \tau$ is a suitable candidate for the pointed reconstruction of $f_x^{\bar \tau} \cdot \pH{\Xi_\ft} = 0$. (The reason why it vanishes is the presence
		of the projection in \eqref{e:defProductProj}.) The bound on $C(\pH{\tau})$ follows by tracking the constants in the bounds obtained throughout this proof. This is straightforward but notationally messy and so we omit this detail.
		
		To obtain suitable bounds on $\big(\eta_\ft^n \cdot \Pi_x \bar{\tau}\big) (\psi_x^m)$, inspired by \cite[Theorem 3.11]{BL21},
we set 
\begin{equ}
F_y(z) = \sum_{|k|_\fs < \beta} \frac{(z-y)^k}{k!} \partial^k \Pi_x \bar \tau (y)\;
\end{equ}
where $\beta = \min_i \alpha_{\sigma_i} + |\fl_i|_\fs$.
($\Pi_x \bar \tau$ is guaranteed to be a function in $\CC^{\beta}$ by Lemma~\ref{lemma: bad tree identification}.)

Choosing $\phi$ as in Definition~\ref{def: semigroup kernel} and using \eqref{e:convolProp} in the last step, one has
\begin{equs}
\big(\eta_\ft^n \Pi_x &\bar{\tau}\big) (\psi_x^m) = \int \big(\eta_\ft^n \Pi_x \bar{\tau}\big)(y)\,\psi_x^m(y)\,dy
= \lim_{N \to \infty} \int \big(\eta_\ft^n \Pi_x \bar{\tau}\big)(\phi_y^N)\,\psi_x^m(y)\,dy \\
&= \int \eta_\ft^n( F_y \cdot \phi_y^m) \psi_x^m(y) dy 
+ \sum_{k \ge m} \int \eta_\ft^n( F_y \cdot (\phi_y^{k+1} - \phi_y^{k})) \psi_x^m(y) dy \\
&= \int \eta_\ft^n( F_y \cdot \phi_y^m) \psi_x^m(y) dy 
  + \sum_{k \ge m} \iint \eta_\ft^n (F_y \cdot \phi_z^{k+2}) R_y^k(z) \psi_x^m(y) dy dz
		\end{equs}
		where $R = \rho^1 - \rho^1 \ast \rho$.
		
		The first term in this expression is bounded by 
		\begin{equs}
			\Big | \int \eta_\ft^n (F_y \cdot \phi_y^m)  \psi_x^m(y) dy \Big | & \lesssim \sum_{|k|_{\fs} < \beta} \Big | \int \eta_\ft^n \left( (\cdot - y)^k \phi_y^m  \right) \partial^k \Pi_x \bar \tau (y) \psi_x^m(y) dy \Big |
			\\
			& \lesssim 2^{-m ( \reg \ft + |\bar \tau|_\fs - |\fs|/2)} \lesssim 2^{-m |\tau|_\fs}\;,
		\end{equs}
		since $\reg \ft - |\fs|/2 > |\ft|_\fs$. Here the second inequality was obtained by hitting the model with an $L^\infty$ bound and applying Cauchy--Schwarz in $L^2$ to what is left.
		
		We now turn our attention to the summand, which we rewrite as
		\begin{equs}
			\iint \eta_\ft^n & (F_z \cdot \psi_z^{k+2})  R^k \ast \psi_x^m(z) dz \\ & + \iint \eta_\ft^n \left ( (F_{z+h} - F_z) \phi_z^{k+2} \right ) R^k(-h) \psi_x^m(z + h) dz dh.
		\end{equs}
		For the first term write, 
		\begin{equs}
			\Bigg | \int \eta_\ft^n & (F_z \cdot \phi_z^{k+2})  R^k \ast \psi_x^\lambda(z) dz \Bigg | \\ & \lesssim \sum_{|j|< \beta} \int \eta_\ft^n \left ( (\cdot - z)^j \phi_z^{k+2} \right ) \partial^j \Pi_x \bar \tau (z) R^k \ast \psi_x^m(z) dz
			\\
			& \lesssim \sum_{|j|_\fs < \beta} 2^{-m(|\bar \tau|_\fs - |j|_\fs)} 2^{-k( \reg \ft + |j|_\fs)} 2^{-kr} 2^{mr + m |\fs|/2}
		\end{equs}
		where the last inequality follows again by applying an $L^\infty$ bound to the model, Cauchy--Schwarz to the integral and using the improved bound on the convolution $R^k \ast \psi_x^m$ that follows from the fact that $R^k$ annihilates polynomials so that a Taylor jet of $\psi_x^m$ can be inserted. Since we can choose $r$ to be as large as we like, this bound is summable in $k$ to a bound of order $2^{-m( \reg \ft + |\bar \tau|_\fs - |\fs|/2)}$ which is of the correct order as before.
		
		It remains to consider 
		\begin{equ}
			\iint\eta_\ft^n \left ( (F_{z+h} - F_z) \phi_z^{k+2} \right ) R^k(-h) \psi_x^m(z + h) dz dh.
		\end{equ}
		Writing $T_f^\gamma(x, h) \eqdef f(x+h) - \sum_{|l|_\fs < \gamma} \frac{h^l}{l!} \partial^l f(x)$, we have that
		\begin{equ}
			(F_{z+h} - F_z)(y)  = - \sum_{|j|_\fs < \beta} \frac{(y - z)^j}{j!} T_{\partial^j \Pi_x \bar \tau}^{\beta - |j|_\fs}(z, h).
		\end{equ} 
		Therefore, by an application of \cite[Theorem A.1]{Hai14}, we have that the remaining term is
		\begin{equ}
			\sum_{|j|_\fs < \beta} \sum_{l \in \partial \bN_{\beta}^j} \iint \eta_\ft^n ((\cdot - z)^j \phi_z^{k+2}) \partial^{l+j} \Pi_x \bar \tau(z + u) R^k(-h) \phi_x^m(z+h) \CQ^{l}(du, h) dh dz
		\end{equ}
		where $\partial \bN_{\beta}^j = \{l \in \bN^d: |l + j|_\fs > \beta \text{ and there exists } i \text{ such that } |l + j - e_i|_\fs < \beta\}$.
		By first bounding the appearance of the model by an $L^\infty$ estimate and performing the integral in $u$, we obtain a bound of order
		\begin{equ}
			\sum_{|j|_\fs < \beta} \sum_{l \in \partial \bN_\beta^j} 2^{-m (|\bar \tau| - |l + j|_\fs)} 2^{-k |l|_\fs} \iint (\rho^n \ast \eta_\ft) ((\cdot - z)^j \phi_z^{k+2}) R^k(-h) \phi_x^m(z+h) dz dh.
		\end{equ} 
		
		Applying Cauchy--Schwarz to the integral in $z$, and using the fact that $R^k$ integrates to $1$, we obtain a bound of order $2^{-m (|\bar \tau|_\fs - \beta - |\fs|/2)} 2^{-k (\beta + \reg \ft)}$. Since $\beta = \min_i \alpha_{\sigma_i} + |\fl_i|_\fs$ and $\reg \ft > |\ft|_\fs + |\fs|/2$, Lemma~\ref{lemma: bad tree identification} implies that $\beta + \reg \ft > 0$, so that this bound is summable in $k$ to a bound of the correct order. 			
	\end{proof}
	
	The family of modelled distributions $\pH{\tau}$ is now uniquely defined by the specifications \eqref{e:baseH1},\eqref{e:baseH2}, \eqref{e:defProdH}, \eqref{e:int}, as well as the candidate for the pointed reconstruction of $\pH{\tau}$ in the case $\gamma_\tau \le 0$ provided by Lemma~\ref{lemma: pointed candidate exists}. Additionally, the proof of the following result is now automated by the machinery of Section~\ref{section: Pointed Modelled Distributions} and Lemma~\ref{lemma: pointed candidate exists}.
	
	We remind the reader that we have restricted our consideration to a finite dimensional regularity structure so that in particular there is a finite collection of trees $T \subseteq \CT$ such that $\CT$ is spanned by $T$. 
	
	\begin{proposition}\label{prop: H control}
		There exists $k \in \bN$ and a compact set $\ck$ containing $x$ such that
		\begin{equ}[e:mainBound]
			\max_{\tau \in T} \sup_{\|\eta\|_\CH \le 1} \$ H_{\tau; n}^{x, \eta} \$_{2, \gamma_\tau, \deg_2 \tau; x} \lesssim (1+ \|Z\|_{\CT; \ck})^k.
		\end{equ} 
	\end{proposition}
	\begin{proof}
		With all the machinery in hand, this result now follows by a straightforward induction in the number of edges using Lemma~\ref{lemma: H base check} to handle the base case where $\tau = \Xi_\ft$, Theorem~\ref{theo: pointed multiplication} to handle the case $\tau = \tau_1 \cdot \tau_2$ in the inductive step, Theorem~\ref{theo:schauder} to handle integration in the inductive step in the case $\gamma_\tau > 0$ and finally using Theorem~\ref{theo:schauder candidate} and Lemma~\ref{lemma: pointed candidate exists} to handle integration in the case $\gamma_\tau \le 0$. 
	\end{proof}
	
	It remains to show that this family of pointed modelled distributions is related via the reconstruction operator (or our inserted candidate for the reconstruction) to the Fr\'echet derivative of the given model in direction $\eta$. For this, we will restrict ourselves to slightly smoother directions which will make no difference in our eventual analysis by virtue of \eqref{e:mainBound} and a density argument.
	
	We now write $\reg_\theta \ft = \reg \ft + \theta$ and define the space of sufficiently smooth directions
	\begin{equ}
		\nice = \prod_{\ft \in \Lab_-} \CC_w^{\reg_\theta \ft}(\bR^d)\;, 
	\end{equ}
	with $\CC_w^\alpha$ as on page~\pageref{defCw},
	where $\theta > 0$ is a fixed constant which is such that $\reg_\theta : \Lab_- \to \bR_+$. Later, in Section~\ref{section: Identification of Derivatives}, it will be the case that we will want to take $\theta$ to be sufficiently large so that a certain augmented structure has nice properties, but for the time being this assumption is not needed.
	
	Since we wish to harvest this additional smoothness, we will want to include more terms in the expansion than appear in the definition of the $H$'s. In order to avoid confusion as to which reconstruction operator is in play or where truncations are implicitly present, we introduce a second family of functions $\pG{\tau}: \bR^d \to \CT_{\bar{\gamma}_\tau}$ for some $\bar{\gamma}_\tau > 0$ whose projection onto $\CT_{< \gamma_\tau}$ will agree with $\pH{\tau}$ and whose reconstruction will agree with that of $\pH{\tau}$ (including when $\gamma_\tau \le 0$).
	
	\begin{proposition}\label{prop:regHtau}
		There exists a $\bar \theta > 0$ and, for every $\eta \in \nice$, a family of modelled distributions $\pG{\tau}$ such that if $\bar \gamma_\tau = (\alpha_\tau + |\fs|/2) \vee \bar \theta$ then for every
		homogeneous $\tau \in \CT$, one has that $\pG{\tau} \in \CD^{\bar \gamma_\tau} = \CD_{\infty, \infty}^{\bar \gamma_\tau}$, $\CQ_{< \gamma_\tau} \pG{\tau} = \pH{\tau}$ and when $\gamma_\tau \le 0$, $\CR \pG{\tau}$ coincides with the candidate for the reconstruction of $\pH{\tau}$ given in Lemma~\ref{lemma: pointed candidate exists}.
	\end{proposition}
	
	\begin{remark}
		In the case where $\gamma_\tau > 0$, it follows from the fact that $\CQ_{< \gamma_\tau} \pG{\tau} = \pH{\tau}$ and the uniqueness of the reconstruction operator that $\CR \pG{\tau} = \CR \pH{\tau}$.
	\end{remark}
	\begin{proof}
		Our construction will again be an inductive one (over the number of edges).
		As in the case of the $H$'s, we set $\pG{X^k} = 0$ and we define 
		\begin{equ}
			\pG{\Xi_\ft} = \sum_{|k|_\fs < \bar{\gamma}_\ft} \frac{X^k}{k!} \big [D^k \eta_\ft^n - P_x^{|\ft|_\fs}[D^k \eta_\ft^n] \big ].
		\end{equ}
		It is straightforward to see that so long as $\bar \theta < \reg_\theta \ft$, these definitions provide elements of $\CD^{\bar \gamma_\tau}$ and satisfy the desired properties.
		
		We now move to the induction step, beginning with the case of integration. We suppose that we are given a homogeneous $\tau \in \CT$ such that $\pG{\tau} \in \CD^{\bar \gamma_\tau}$ and $\CR \pG{\tau} = \CR \pH{\tau}$.
		
		We define
		\begin{equ}[e:degGInt]
			\pG{\CI^\fl \tau} = \CK_{\bar \gamma_\tau}^\fl \pG{\tau} - T_{\deg_2 \tau, 2; \fl}^x \pG{\tau}
		\end{equ} 
		where $\CK_{\bar{\gamma}_\tau}^\fl$ is the usual abstract integration map on $\CD^{\bar \gamma_\tau}$ corresponding to the integral kernel $K^{\fl}$ (see \cite[Section 5]{Hai14}) and $T_{\nu, 2; \fl}^x$ is as in Definition
		~\ref{def: pointed integration}.
		
		We remark that the latter term is well-defined since the reconstruction operator appearing coincides with $\CR\pH{\tau}$ so that its pointwise evaluations at $x$ do make sense.
		It is straightforward to check by comparing the definitions that $\CQ_{< \gamma_\tau} \pG{\CI^\fl \tau} = \pH{\CI^\fl\tau}$. Additionally, since the second term in \eqref{e:degGInt} is polynomial,  one actually obtains
		\begin{equ}[e:betterBound]
			G_{\CI^\fl \tau; n}^{x,h} \in \CD^{\bar \gamma_\tau + |\fl|_\fs}\;,
		\end{equ}
		which is an improvement since in general the inequality $\bar \gamma_{\CI^\fl\tau} \le \bar \gamma_\tau + |\fl|_\fs$ may be strict.
		
		Regarding multiplication, as for $\pH{\tau \bar \tau}$, we define $\pG{\tau \bar \tau} = f^\tau_x(y)\,G^{x,\eta}_{\bar\tau; n}(y)
		+ G^{x,\eta}_{\tau; n}(y)\, f_x^{\bar\tau}(y)$. 
		
		It follows immediately from \cite[Theorem 4.7]{Hai14} that
		\begin{equ}
			G_{\tau\bar\tau; n}^{x,h} \in \CD^{(\bar \gamma_\tau + \tilde{\alpha}_{\bar\tau}) \wedge (\bar \gamma_{\bar \tau} + \tilde{\alpha}_\tau)}
		\end{equ}
		where $\tilde{\alpha}_\tau$ is the lowest degree in the smallest sector containing $\tau$ (rather than the lowest \textit{non-polynomial} degree). Since $\alpha_{\tau \bar \tau} = (\alpha_\tau + \tilde{\alpha}_{\bar \tau})\wedge (\tilde{\alpha}_\tau + \alpha_{\bar \tau})$ we conclude that if $\alpha_{\tau \bar \tau} + |\fs|/2 \ge \bar{\theta}$, then this product does indeed belong to $\CD^{\bar \gamma_{\tau \bar \tau}}$.
		
		If this constraint on $\alpha_{\tau \bar \tau}$ fails then, by choosing $\bar\theta$ sufficiently small, we have that $\alpha_{\tau \bar \tau} < - |\fs|/2$ and Lemma~\ref{lemma: bad tree identification} implies that without loss of generality we can assume that $\tau = \Xi_\ft$ and $\bar{\tau} =  \left ( \prod_{i=1}^n \CI^{\fl_i} \sigma_i \right )$. Further, for $\bar{\theta}$ sufficiently small, for each $i$ we have that $\alpha_{\sigma_i} + |\fl_i|_\fs + |\ft|_\fs \ge - |\fs|/2 + \bar\theta$. 
		
		We then note that $f_x^\tau \pG{\bar \tau} \in \CD^{\min_i (\bar\gamma_{\sigma_i} + |\fl_i|_\fs + |\ft|_\fs)} \subseteq \CD^{\bar \theta}$ by \cite[Theorem 4.7]{Hai14}. Additionally, since for each $i$, $\alpha_{\sigma_i} + |\fl_i|_\fs \ge 0$, $f_x^{\bar \tau}$ is valued in a function-like sector and lies in $\CD^c$ for any $c > 0$, we conclude also that $f_x^{\bar \tau} \pG{\tau} \in \CD^{\bar \theta}$. In total, this implies that $\pG{\tau \bar \tau}$ lies in the correct space.
		
		It follows immediately from the definitions and the induction hypothesis that $\CQ_{< \gamma_{\tau \bar \tau}} \pG{\tau \bar \tau} = \pH{\tau \bar \tau}$ so that it only remains to check that in the case where $\gamma_{\tau \bar \tau} \le 0$, the reconstruction of $\pG{\tau \bar \tau}$ coincides with our candidate for the reconstruction of $\pH{\tau \bar \tau}$.
		
		Without loss of generality, we will again assume that $\tau = \Xi_\ft$ and $\bar{\tau} =\left ( \prod_{i=1}^n \CI^{\fl_i} \sigma_i \right )$. By uniqueness of the reconstruction operator and the form of our candidate for the reconstruction of $\pH{\tau \bar \tau}$, it will suffice for us to see that $\CR (f_x^{\bar \tau} \cdot \pG{\Xi_\ft}) =\eta_\ft^n \cdot \Pi_x \bar \tau$.
		
		Since $\pG{\Xi_\ft}$ is valued in $\Tpoly$ so that the model acts multiplicatively on $f_x^{\bar \tau} \pG{\Xi_\ft}$, by uniqueness of the reconstruction, in the case $|\ft|_\fs < 0$ it suffices to see that for $y \in \bR^d$
		\begin{equ}
			\bigg | \int \bigg ( \eta_\ft^n(z) - \sum_{|k|_\fs < \bar{\gamma}_\ft} \frac{(z-y)^k}{k!} D^k \eta_\ft^n(y) \bigg ) \Pi_x\bar \tau(z) \psi_y^\lambda(z) \, dz \bigg | \lesssim \lambda^{\bar \theta}.
		\end{equ}
		This is straightforward by applying a naive uniform bound on the model and using Taylor's theorem to treat the first term since $\bar \gamma_\ft = \bar \theta$ in this case. Since we have already covered several similar but less simple estimates in this section we will omit the details. The adaptations required in the case $|\ft|_\fs > 0$ are also essentially the same as in the proof of Lemma~\ref{lemma: H base check} and so we also omit repetition of those techniques here.
	\end{proof}

	\subsection{Identification of the Fr\'echet derivatives}\label{section: Identification of Derivatives}
	
	In this section we show that, for $\eta \in \nice$, we have the previously announced identity
	\begin{equ}
		D_\eta \Pi_z^n \tau = \CR H_{\tau; n}^{z,\eta}\;,
	\end{equ}
	where $D_\eta$ denotes the Fr\'echet derivative in the direction $\eta$ and $\Pi^n$ is a model constructed via the action of the renormalisation group $\CG_-$ on $\CL(\xi^n)$. By Proposition~\ref{prop:regHtau}, it suffices to show that $D_\eta \Pi_z^n \tau = \CR \pG{\tau}$ and this is the route we pursue.
	
	For the remainder of this section, for notational simplicity we will consider $n$ to be fixed and will suppress the dependence on $n$ in our models and modelled distributions.
	
	We will broadly follow the sketch of proof in \cite[Section 5]{Jonathan} (see also \cite[Section~3.2]{Giuseppe} for
	a similar construction, but in a restricted setting). However, in the notation of \cite{Jonathan} for this sentence only, it is not true that $f_z^\tau(\bar{z})$ is a linear combination of terms of $\overline{\deg}$-degree equal to $\deg \tau$ (as can be seen by considering $f_z^{\CI \tau}(z)$ since $\deg$ and $\overline\deg$ agree on polynomials). Since this claim was crucial to the correctness of their argument, we must make adaptations to correct for this error here.
	
	We consider an augmented regularity structure $\tilde{\CT}$ where we duplicate every noise $\Xi_\ft$ and assign  degree 
	$\deg\tilde \Xi_\ft = \reg_\theta \ft$ to its twin $\tilde \Xi_\ft$. 
	The augmented structure is then spanned by those trees that can be obtained from a tree in $\CT$ by substituting some of its noise edges with their respective twins.
	
	More precisely, $\tilde{\CT}$ is the reduced regularity structure for the noise labels $\widetilde{\Lab}_- = \Lab_- \cup \{\tilde{\ft} : \ft \in \Lab_-\}$ and rule $\tilde{R}$ formed by extending the rule $R$ to treat the edges with labels in the latter set in the same way as their twins. 
	
	We note that it follows from
	Assumption~\ref{ass:alg} that all trees in $\tilde{\CT}$ containing an instance of $\tilde{\Xi}_\ft$ have positive degree so long as $\theta$ is chosen sufficiently large. Hence, it follows from \cite[Theorem 5.14, Proposition 3.31]{Hai14} that for such $\theta$ (which we now assume to be fixed), given $\eta \in \nice$, 
	\textit{any} admissible model
	$\PPi$ for $\CT$ extends \textit{uniquely} to an admissible model $\PPi^\eta$ for $\tilde \CT$ with the 
	property that $\PPi^\eta \tilde \Xi_\ft = \eta_\ft^n$ for every $\ft \in \Lab_-$. Furthermore
	the map $(\eta, \PPi) \mapsto \CY_\eta(\PPi) = \PPi^\eta$ is continuous in the respective topologies. We henceforth assume that 
	such a model is fixed.
	
	We then define the operators $\CJ^\ft(z) \colon \tilde \CT \to \Tpoly$ as in \cite[Eq.~5.11]{Hai14},
	as well as setting for $\tau \in \CT$, $\bar \CJ^{\ft}_{\tau}(z) \colon \tilde \CT \to \Tpoly$ to be the map obtained by setting
	$\bar \CJ^{\ft}_\tau(z) \eqdef \CQ_{<\deg \CI^\ft \tau}\CJ_\ft(z)$.
	
	Similarly to the way the modelled distributions $H_\tau^{x,\eta}$ are defined, we define a collection of $\tilde{\CT}$-valued modelled distributions indexed by $\tau \in \CT$ by setting
	$\tilde f^{X^k}_z = 0$ for all $k \in \bN^d$, 
	\begin{equ}
		\tilde f^{\Xi_\ft}_z = \tilde{\Xi}_\ft + \sum_{|k|_\fs < \reg_\theta \ft}{X^k \over k!} \big [D^k \eta_\ft^n - P_z^{|\ft|_\fs}[D^k \eta_\ft^n] \big ]\;,
	\end{equ}
	and then inductively 
	\begin{equ}[e:defProdf]
		\tilde f^{\tau \bar \tau}_z = \tilde f^{\tau }_z\,  f^{\bar \tau}_z + \tilde f^{\bar \tau}_z\, f^{\tau}_z\;,\qquad
		\tilde f_z^{\CI^\ft \tau}(\bar z) = \bigl(\CI^\ft  + \CJ^{\ft}(\bar z) - \Gamma_{\bar zz} \bar \CJ^{\ft}_ \tau(z)\Gamma_{z\bar z}^\eta\bigr)\tilde f_z^{\tau}(\bar z)\;.
	\end{equ}

	\begin{definition}
		A $\PPi^\eta$-polynomial is a map $f : \bR^{d} \to \tilde{\CT}$ such that for any $z, \bar{z}$ we have that $f(z) = \Gamma^\eta_{z \bar{z}} f(\bar{z})$
	\end{definition}
	\begin{proposition}\label{prop:idenftilde}
		For all $\tau \in \CT$ and $z \in \bR^d$, $\tilde{f}_z^\tau$ is a $\PPi^\eta$-polynomial. 
	\end{proposition}
	\begin{proof}
		We need only check that $\tilde{f}_z^\tau(\bar{z}) = \Gamma_{\bar{z}z}^\eta \tilde{f}_z^\tau(z)$.
		
		This is trivially true in the case where $\tau = X^k$ and in the case where $\tau = \Xi_\ft$ this follows by construction from the action of the structure group on $\tilde{\Xi}_\ft$ by a tedious calculation that is made somewhat simpler by noting that $f_z^{\Xi_\ft}(\bar z) - \Gamma_{\bar z z}f_z^{\Xi_\ft}(z)$ is valued in the polynomial part of the regularity structure so that one can leverage injectivity of $\Pi_{\bar z}$ on polynomials.
		
		Seeing that the property is stable under multiplication is straightforward since the maps $f_z^\tau$ have the same property and $\Gamma^\eta$ is multiplicative. Hence we will consider now the case of integration.
		
		Here we write
		\begin{equ}
			\Gamma_{\bar z, z}^\eta \tilde{f}_z^{\CI^\fl \tau}(z) = \Gamma_{\bar z, z}^\eta \left ( \CI^\fl + \CJ^\fl(z) - \bar{\CJ}^{\fl}_\tau(z) \right ) \tilde{f}_z^\tau(z).
		\end{equ} 
		
		Applying \cite[Lemma 5.16]{Hai14} to the first
		two terms on the right-hand side yields
		\begin{equs}
			\Gamma_{\bar{z}z}^\eta \tilde{f}_z^{\CI^\ft \tau}(z) = (\CI^\ft + \CJ^\ft(\bar{z})) \Gamma_{\bar{z}z}^\eta \tilde{f}_z^\tau(z) - \Gamma_{\bar{z}z}^\eta \bar{\CJ}^{\ft}_\tau(z) \tilde{f}_z^\tau(z).
		\end{equs}
		The first term on the right-hand side equals $(\CI^\ft + \CJ^\ft(\bar{z})) \tilde{f}_z^\tau(\bar{z})$ by the induction hypothesis. For the second term, we note that $\Gamma_{\bar{z}z}^\eta = \Gamma_{\bar{z}z}$ on polynomials and write $\operatorname{Id} = \Gamma_{z\bar{z}}^\eta \Gamma_{\bar{z}z}^\eta$ to see that this term is equal to $\Gamma_{\bar{z}z} \bar{\CJ}^\ft(z) \Gamma_{z \bar{z}}^h \Gamma_{\bar{z}z}^h \tilde{f}_z^\tau(z)$. We then recognise the right-hand side of \eqref{e:defProdf} and conclude that we have indeed proven the identity 
		$\Gamma_{\bar{z}z}^\eta \tilde{f}_z^{\CI^\ft \tau}(z) = \tilde{f}_z^{\CI^\ft \tau}(\bar z)$
		as claimed.
	\end{proof}
	
	We now write $\hat \CQ \colon \tilde \CT \to \CT$ for the map that sends any 
	symbol containing an instance of $\tilde{\Xi}_\ft$ (for some $\ft$) 
	to $0$ and otherwise acts as the identity. 
	We then have the following result
	relating these two definitions.
	
	\begin{proposition}\label{prop:MalDer}
		Let $\eta \in \nice$, let $\PPi$ be an admissible model for $\CT$, and endow
		$\tilde \CT$ with the admissible model $\PPi^\eta$.
		For any $\tau \in \CT$, one then has the identity
		\begin{equ}[e:idenHf]
			\CR H_\tau^{z,\eta} = \CR^\eta \tilde f_z^{\tau} = \Pi_z^\eta \tilde f_z^{\tau}(z)\;,
		\end{equ}
		where $\CR^\eta$ is the reconstruction operator for $\PPi^\eta$.
	\end{proposition}
	
	\begin{proof}
		The identity $\CR^\eta \tilde f_z^{\tau} = \Pi_z^\eta \tilde f_z^{\tau}(z)$ follows at once
		from the fact that $\tilde f_z^{\tau}(\bar z) = \Gamma^\eta_{\bar z z}\tilde f_z^{\tau}(z)$ by
		Proposition~\ref{prop:idenftilde}, so that $\Pi_{\bar z}^\eta \tilde f_z^{\tau}(\bar z)$ 
		is independent of $\bar z$.
		
		We then claim that for every homogeneous element $\tau$ one has the identity
		\begin{equ}[e:equalProj]
			\CQ_{< \bar \gamma_\tau} \tilde f_z^{\tau} = G_\tau^{z,\eta} 
		\end{equ}
		whence the first equality follows from the uniqueness of the reconstruction operator and Proposition~\ref{prop:regHtau}.
		We show \eqref{e:equalProj} again by induction. Clearly, it holds for the base cases
		thanks to the fact that $\deg \tilde \Xi_\ft = \bar \gamma_{\Xi_\ft}$. 
		
		Regarding integration, say $\tau = \CI^\ft \sigma$, we note that the definition 
		of $G_{\CI^\ft\sigma}^{z,\eta}$ guarantees
		that
		\begin{equ}
			\CR G_{\tau}^{z,\eta} = K^\ft \ast \CR G_{\sigma}^{z,\eta}
			- \sum_{|k|_\fs < \deg\tau} (D^k K^\ft \ast \CR G_{\sigma}^{z,\eta})(z) {(\cdot - z)^k \over k!}\;,
		\end{equ}
		and similarly for $\Pi_z \tilde f_z^{\tau}(z)$ (and therefore $\CR^\eta \tilde f_z^{\tau}$).
		
		It is also straightforward to see from the induction hypothesis that
		$\Delta \eqdef \CQ_{< \bar \gamma_\tau} (G_\tau^{z,\eta} - \tilde f_z^{\tau})$
		takes values in the Taylor polynomials. Since we know furthermore that when considered as modelled distributions valued in $\tilde{\CT}$
		both $G_{\tau}^{z,\eta}$ and $\tilde f_z^{\tau}$ belong to $\CD^{\bar \gamma_\tau}$ by
		Propositions~\ref{prop:regHtau} and~\ref{prop:idenftilde}, $\Delta$ coincides with the 
		Taylor lift of a function $\Delta_0 \in \CC^{\bar \gamma_\tau}$, so that we necessarily have
		$\Delta = 0$ since $\Delta_0 = \CR G_{\tau}^{z,\eta} - \CR^\eta \tilde f_z^{\tau} = 0$.
	\end{proof}
	It now remains to prove that $\CR^\eta \tilde f_z^\tau = D_\eta \Pi_z \tau$. To do this, we identify models constructed via the canonical lift and the action of the renormalisation group from the noise $\xi + \varepsilon \eta$. We are again inspired by the sketch argument given in \cite[Section 5]{Jonathan}.
	
	For $\varepsilon > 0$, to describe such a model we construct another family of $\PPi^\eta$-polynomials which we denote by $f_z^{\tau, \varepsilon}$. 
	
	We define $f_z^{X^k, \varepsilon}(\bar{z}) = \Gamma_{\bar{z}z}X^k$ and $f_z^{\Xi_\ft, \varepsilon}(x) = \Xi_\ft + \varepsilon \tilde{f}_z^{\Xi_\ft}$. We then recursively define
	\begin{equs}
		f_z^{\tau \bar{\tau}, \varepsilon}(\bar{z})= f_z^{\tau, \varepsilon}(\bar{z}) f_z^{\bar{\tau},\varepsilon}(\bar{z}), \qquad f_z^{\CI^\ft \tau, \varepsilon}(\bar{z}) = \left ( \CI^\ft + \CJ^\ft(\bar{z}) - \Gamma_{\bar{z} z} \bar{\CJ}^{\ft}_ \tau(z) \Gamma_{z \bar{z}}^\eta \right ) f_z^{\tau, \varepsilon}(\bar{z}).
	\end{equs}
	
	Additionally, we define maps $\Lambda_{\bar{z}z}^\varepsilon : \CT\to \CT$ recursively by setting $\Lambda_{\bar{z}z}^\varepsilon \Xi_\ft = \Xi_\ft$, $\Lambda_{\bar{z}z} X^k = \Gamma_{\bar{z}z}X^k$, declaring that $\Lambda_{\bar{z}z}^\varepsilon$ is multiplicative, and by setting
	\begin{equs}
		\Lambda_{\bar{z}z}^\varepsilon \CI^\ft \tau = \CI^\ft \Lambda_{\bar{z}z}^\varepsilon \tau + \left ( \bar{\CJ}^{\ft}_ \tau(\bar{z}) \Gamma_{\bar{z}z} - \Gamma_{\bar{z}z} \bar{\CJ}^{\ft}_\tau(z) \right ) f_{\bar{z}}^{\tau, \varepsilon}(\bar{z}).
	\end{equs}
	
	We then define $Z^\varepsilon: \CM_0(\tilde \CT) \to \CM_0(\CT)$ by setting $Z^\varepsilon (\Pi^\eta, \Gamma^\eta) = (\Pi, \Gamma)$ where $\Gamma_{\bar{z}z} = \Lambda_{\bar{z}z}^\varepsilon$ and $\Pi_z \tau = \Pi_z^\eta f_z^{\tau, \varepsilon}(z)$. 
	
	\begin{proposition}
		The map $Z^\varepsilon:\CM_0(\tilde \CT) \to \CM_0(\CT)$ is well-defined.
	\end{proposition}
	\begin{proof}
		We are required to show that for $\PPi^\eta = (\Pi^\eta, \Gamma^\eta) \in \CM_0(\tilde \CT)$, it is the case that $Z^\varepsilon \PPi^\eta = (\Pi, \Gamma)$ is a model on $\CT$.
		
		We observe that it follows by a simple induction that $f_z^{\Lambda_{z \bar{z}}^\varepsilon \tau, \varepsilon} = f_{\bar{z}}^{\tau, \varepsilon}$, where we extended the definition of $f_z^{\tau, \varepsilon}$ to all $\tau \in \CT$ by linearity. Since $\tau \mapsto f_z^{\tau, \varepsilon}$ is an injective map, this also yields that $\Lambda_{xy}^\varepsilon \Lambda_{yz}^\varepsilon = \Lambda_{xz}^\varepsilon$. Therefore the algebraic requirements for $Z^\varepsilon \PPi^h$ to be a model do hold
and it remains to verify the analytic requirements. 
		
		For the bounds on $\Pi_z$, it follows from a simple induction that $f_z^{\tau, \varepsilon}$ is a linear combination of terms of degree at least $|\tau|_\fs$ so that the bounds on $\Pi_z$ are immediate since no uniformity in $\varepsilon$ is required. 
		
		This leaves us to consider the bounds on $\Gamma$. These follow again by an inductive argument where the only non-trivial step is obtaining the appropriate bound for $\Gamma_{\bar{z}z} \CI^\ft \tau$. The bound in this case is immediate for the first term in the definition of $\Lambda_{\bar{z}z}^\varepsilon \CI^\ft \tau$ so that leaves us only to obtain the appropriate bound on $\left ( \bar{\CJ}^{\ft}_\tau \Gamma_{\bar{z}z}(\bar{z}) - \Gamma_{\bar{z}z} \bar{\CJ}^{\ft}_ \tau(z) \right ) f_{\bar{z}}^{\tau, \varepsilon}(\bar{z})$.
		This however follows from the observation made above on the degree of terms in $f_z^{\tau, \varepsilon}$ and a minor adaptation of the proof of \cite[Lemma 5.21]{Hai14}, so we omit the details.
	\end{proof}
	
	It is straightforward to check by induction that if $\CL$ denotes the canonical lift then $Z^\varepsilon \CY_\eta \CL(\xi) = \CL(\xi + \varepsilon \eta)$. (Recall that $\CY_\eta \colon \PPi \mapsto \PPi^\eta$.) Since we wish to work with renormalised models, it is natural to ask if this construction commutes with the renormalisation group in the sense that for $M_g \in \CG_-$ we have that
	\begin{equ}[e:commutRG]
		M_g Z^\varepsilon \CY_\eta \CL(\xi) = Z^\varepsilon \CY_\eta M_g \CL(\xi).
	\end{equ}
	We emphasise here that we restrict ourselves here to elements $g$ of the subgroup $\CG_-$ of $\fR$ described at the end of Section~\ref{section: Renormalisation Group}. This subgroup consists of elements that are of the form $\CQ M_g^\mathrm{ex} = \CQ (g \otimes 1) \Delta_{\mathrm{ex}}^-$ where $g$ is a reduced character (in the sense that it does not depend on the decoration $\fo$) on the space $\CT_-^\mathrm{ex}$. 
	
	This is necessary since we split the proof of the commutation result \eqref{e:commutRG} into two parts. In the first part, we establish a commutation result with the map $\CY_\eta$ and in the second we establish a commutation result with $Z^\varepsilon$.
	
	For the first of these steps we will need to restrict to elements of $\CG_-$ whilst in the second step we will need to introduce a particular element of $\fR \setminus \CG_-$ that will describe the action of $Z^\varepsilon$ so that we will require both descriptions in what follows. 
	
	As was mentioned in Section~\ref{section: Renormalisation Group}, we will identify elements of $\CG_-$ (which are properly speaking linear maps $M : \CT \to \CT$) with the corresponding reduced characters $g$ on $\CT_-^\mathrm{ex}$ and we will write $M_g$ when we wish to make clear that we mean the corresponding element of $\fR$.
	
	To prove the relation \eqref{e:commutRG}, we follow the sketch argument given at the end of \cite[Section 5.1]{Jonathan}, making the necessary adaptations to turn it
	into a rigorous and correct argument in the general setting considered here.
	
	As mentioned, for our first intermediate step, we would like to introduce a map $\Phi : \CG_- \to \tilde{\CG}_-$ on characters such that $\CY_\eta M_g \PPi = M_{\Phi(g)} \CY_\eta \PPi$. Here $\tilde{\CG}_-$ is the analogue of $\CG_-$ for the structure $\tilde{\CT}$.
	
	We define such a $\Phi$ by setting $\Phi(g) = g \circ \hat\CQ$ where $\hat \CQ$ is the linear map which sends each tree containing an edge with a decoration in $\tilde{\Lab}_- \setminus \Lab_-$ to $0$ and otherwise acts as the identity map.
	
	\begin{proposition}\label{prop:commute}
		For any $g \in \CG_-$, one has the identity $\CY_\eta M_g \PPi = M_{\Phi(g)} \CY_\eta \PPi$.
	\end{proposition}
	\begin{proof}
		Since it is clear that these models agree on instances of $\tilde{\Xi}_\ft$, by \cite[Theorem 5.14, Proposition 3.31]{Hai14}, it is only necessary to check that the two models agree on elements of $\CT$.
		
		This follows from the explicit description of $\Delta^{M_g}, \Delta^{M_{\Phi(g)}}$ given in \cite[Theorem 6.36]{BHZ} since one can easily check that the map $\tilde{L}$ such that $\Delta^{M_{\Phi(g)}} = \tilde{L} (g \otimes 1) \Delta_\mathrm{ex}^-$ extends the corresponding map $L$ for $M_g$.
	\end{proof}
	
	In order to show that \eqref{e:commutRG} holds, it now remains to see that $Z^\varepsilon M_{\Phi(g)} \PPi^\eta = M_g Z^\varepsilon \PPi^\eta$ for $\PPi^\eta \in \CM_0(\tilde{\CT})$. To do this, we describe the action of $Z^\varepsilon$ as being given (up to composing with the natural projection $\tilde{\CM} \to \CM$) by an element of the renormalisation group $\tilde{\fR}$ for $\tilde{\CT}$.
	
	We define a map $M^\varepsilon: \tilde{\CT} \to \tilde{\CT}$ by replacing all instances of a noise $\Xi_\ft$ in a symbol with $\Xi_\ft + \varepsilon \tilde{\Xi}_\ft$. At this point, since $M^\varepsilon$ is a multiplicative map, it follows by a simple induction using Lemmas~\ref{Lemma: Integration Coproduct} and~\ref{Lemma: Multiplication Coproduct} that $M^\varepsilon \in \tilde{\fR}$. Additionally, we have the following result relating $M^\varepsilon$ to the map $Z^\varepsilon$.
	
	\begin{proposition}\label{prop: Z map is renorm}
		Let $P: \tilde{\CM} \to \CM$ be the natural projection. Then $Z^\varepsilon = P \circ M^\varepsilon$.
	\end{proposition}
	\begin{proof}	
		It suffices to see that for $\tau \in \CT$, $(1 \otimes f_x) \Delta^{M^\varepsilon} \tau = f_x^{\tau, \varepsilon}(x)$ where $f_x$ is as defined in \cite[Eqs~8.20--8.21]{Hai14}.
		We check that this identity holds by induction. The base cases where $\tau = X^k$ or $\tau = \Xi_\ft$ hold by construction.  Since $M^\varepsilon$ is multiplicative, the inductive step involving multiplication is immediate from Lemma~\ref{Lemma: Multiplication Coproduct} since $f_x$ is multiplicative. 
		
		It remains to see that this relation is stable under integration. By Lemma~\ref{Lemma: Integration Coproduct} we can write 
		\begin{equs}
			(1 \otimes f_x) \Delta^{M^\varepsilon} \CI \tau = (\CI\otimes 1)(1 \otimes f_x) \Delta^{M^\varepsilon} \tau - \sum_{|k|> |\tau| + \beta} \frac{X^k}{k!} (f_x \CJ_k \otimes 1) (1 \otimes f_x) \Delta^{M^\varepsilon} \tau.
		\end{equs}
		By the induction hypothesis, this is nothing but
		\begin{equs}
			\CI f_x^{\tau, \varepsilon}(x) - \sum_{|k| > |\tau| + \beta} \frac{X^k}{k!} f_x( \CJ_k f_x^{\tau, \varepsilon}(x))
		\end{equs}
		which is easily checked to coincide with the definition of $f_x^{\CI \tau, \varepsilon}(x)$.
	\end{proof}
	
	We are now in position to prove the following result.
	\begin{proposition}
		For $\PPi^\eta \in \tilde{\CM}$ we have that $M_g Z^\varepsilon \PPi^\eta = Z^\varepsilon M_{\Phi(g)} \PPi^\eta$.
	\end{proposition}
	\begin{proof}
		By Proposition~\ref{prop: Z map is renorm}, it suffices to see that $M^\varepsilon M_{\Phi(g)} \PPi^\eta = M_{\Phi(g)} M^\varepsilon \PPi^\eta$. By Proposition~\ref{Lemma: Commutativity Coproduct} and the definition of the action of the renormalisation group, for this it suffices to see that $M^\varepsilon$ and $M_{\Phi(g)}$ commute on $\tilde{\CT}$. 
		
		This follows immediately from the fact that every tree containing a term $\tilde{\Xi}_\ft$ has positive degree so that by its definition $\tilde{\Delta}_{\mathrm{ex}}^-$ never extracts such a tree to be hit by the character $\tilde{g}$ in the definition of $M_{\Phi(g)}$. 
	\end{proof}
	
	The equality $Z^\varepsilon \CY_\eta \CL(\xi) = \CL(\xi + \varepsilon \eta)$ alongside the previous results on commutation with renormalisation then immediately yield the following Lemma.
	\begin{lemma}\label{Lemma: Identifying shift}
		For $g \in \CG_-$, we have that $M_g \CL(\xi + \varepsilon \eta) = Z^\varepsilon \CY_\eta(M_g \CL(\xi))$\qed
	\end{lemma}
	
	We are now in position to prove:
	\begin{theorem}\label{theo: Malliavin Derivative}
		For $\tau \in \CT$, $\eta \in \nice$ and $\PPi = (\Pi, \Gamma)$ in the orbit of $\CL(\xi_n)$ under the renormalisation group $\CG_-$, we have that 
		$D_\eta \Pi_z \tau = \CR H_{\tau;n}^{z, \eta}$.
	\end{theorem}
	\begin{proof}
		By Proposition~\ref{prop:MalDer}, it suffices to check that if $\CY_\eta(\PPi) = (\Pi^\eta, \Gamma^\eta)$ then $\Pi_z^\eta \tilde{f}_z^\tau(z) = D_\eta \Pi_z \tau$.
		
		It follows from a straightforward induction that $f_z^{\tau, \varepsilon}(z)$ is a polynomial in $\varepsilon$ with constant term reconstructing to $\Pi_z \tau$ and first order term given by $\varepsilon \tilde{f}_z^\tau$.
		
		By Lemma~\ref{Lemma: Identifying shift}, we can then write
		\begin{equs}
			D_\eta \Pi_z \tau = \lim_{\varepsilon \to 0} \varepsilon^{-1} \left [ \left ( \Pi_z \tau + \varepsilon \Pi_z^\eta \tilde{f}_z^\tau(z) + o(\varepsilon^2) \right ) - \Pi_z \tau \right ]
		\end{equs}
		which immediately yields the desired result.
	\end{proof}
	
\begin{remark}
As already hinted at earlier, Theorem~\ref{theo: Malliavin Derivative} does \emph{not} hold
in general if we replace $\CG_-$ by the larger renormalisation group $\fR$ defined
in \cite{Hai14}. This can be seen already on the base case $\tau = \Xi_\ft$ since, if there are two
elements $\ft_1, \ft_2 \in \Lab_-$ with $|\ft_1|_\fs = |\ft_2|_\fs$, then the map $M$ that swaps
the corresponding noises belongs to $\fR$. In that case however, 
setting $\PPi = M \CL(\xi_n)$, one has
\begin{equ}
D_\eta \Pi_z \Xi_{\ft_1} = \eta_{\ft_2}^n\;,\qquad
\CR H_{\Xi_{\ft_1};n}^{z, \eta}  = \eta_{\ft_1}^n\;.
\end{equ}
\end{remark}
	
	\section{Uniform Bounds for the $\bBPHZ$ Model}\label{sec: uniform bounds}
	
	In this section we will derive (uniform in $m$) bounds on $\bE[\|\bar{Z}^m\|_{\gamma; \ck}^p]$ where $\bar{Z}^m = (\bar{\Pi}^m, \bar{\Gamma}^m)$ is the $\bBPHZ$ model constructed from the driving noise $\xi^m$ as defined in Definition~\ref{def: bBPHZ} and $p = 2^k$ for an arbitrary $k \in \bZ_+$. In order to give a more precise statement of the main result of this section we first introduce some relevant notation.
	
	We let $\CF_-$ denote the set of trees $\tau \in \CT$ such that $|\tau|_\fs < 0$ and,
	given a tree $\tau \in \CF_-$, we write $\bar \CF_-(\tau)$ for the set of all subtrees of $\tau$ that have 
	at least one noise edge fewer than $\tau$. We then also set $\bar \CF_- = \bigcup_{\tau \in \CF_-}\bar \CF_-(\tau)$.
	
	Since $\bar{\CF}_-$ is a finite set, we can order its elements $\tau_1, \dots, \tau_{n_0}$ in such a way that for each $j \le n_0-1$ we have that $(n_{\tau_j}, |\tau_j|_\fs) \le (n_{\tau_{j+1}}, |\tau_{j+1}|_\fs)$ where $n_\tau$ is the number of noise edges of $\tau$ and the ordering on pairs is lexicographic. We remark here that when it is the case that $n_{\tau_j} = n_{\tau_k}$ and $|\tau_j|_\fs = |\tau_k|_\fs$, the order in which they appear will not matter to us since in this case we have that $\tau_k \not \in V^{\tau_j}$ and vice versa.
	
	The broad approach of our proof is to proceed to control the behaviour of the $\bBPHZ$ model on each of the symbols $\tau_j$ via induction in $j$. Since we will need to apply the Kolmogorov criterion given in Theorem~\ref{theo: Kolmogorov Criterion} at each stage of the induction, we will lose an arbitrarily small (but still positive) amount of regularity at each step. 
	
	In order to encode this in our proof in a way that separates the core analytic ideas from this technicality, we introduce the following sequence of regularity structures.
	We set $\CT^0 = \CT_{\mathrm{poly}}$. Then, given $\CT^j$ for $j < n_0$, we let $V^{(j+1)}$ be the smallest sector of $\CT$ containing $\CT^j$ and $\tau_{j+1}$. We then define $\CT^{(j+1)}$ to be the regularity structure with model space $V^{(j+1)}$, homogeneities given by $|\cdot|_\fs^{(N-j-1)}$ and structure group given by the restriction of the structure group of $\CT$ to $V^{(j+1)}$. Here $N$ is a fixed constant which is much larger than $n_0$.
	
	An important observation that will allow for the inductive nature of our proof is that the algebraic structure we just described is compatible with an induction in the number of noise edges appearing. 
	
	\begin{lemma}\label{lemma: induction algebraic}
		The sector $V^{(j)}$ is spanned by $\{\tau_1, \dots, \tau_j\}\cup \CT_{\mathrm{poly}}$.
	\end{lemma}
	\begin{proof}
		We proceed by induction in $j$ where the base case with $j = 0$ is trivial.
		
		Hence it remains to see that $\tilde{V}^{(j)} \eqdef V^{(j-1)} \oplus \R \tau_j$ is in fact a sector. By the induction hypothesis, if $\tau \in \tilde{V}^{(j)}$ then $\tau = a_0 \tau_{\mathrm{poly}} + \sum_{i=1}^j a_i \tau_i$ where $a_i \in \bR$ and $\tau_{\mathrm{poly}} \in \CT_{\mathrm{poly}}$. Therefore, we only need to show that $\Gamma \in \CG$, we have that $\Gamma \tau_j \in \tilde{V}^{(j)}$, since $V^{(j-1)}$ is a sector.
		
		We then note that if $\Delta$ is the operator introduced in \eqref{eq: Delta def base} and \eqref{eq: Delta def int} and $\Delta \tilde{\tau} = \sum \tilde{\tau}^{(1)} \otimes \tilde{\tau}^{(2)}$ (in Sweedler notation) then $\tilde{\tau}^{(1)}$ is a subtree of $\tilde{\tau}$ and for $\Gamma \in \CG$, $\Gamma \tilde{\tau}$ is a linear combination of the $\tilde{\tau}^{(1)}$. Hence it follows that $\Gamma \tau_j - \tau_j = b_0 \tau_j^{\mathrm{poly}} + \sum_{i=1}^{j-1} b_i \tau_i$ where $\tau_j^{\mathrm{poly}} \in \CT_{\mathrm{poly}}$ and $b_i \in \bR$. This completes the proof. 
	\end{proof}
	
	We also observe that if $k \ge j$ and $\pH{\tau_j}$ is the pointed modelled distribution corresponding to the Fr\'echet derivative of $\bar{\Pi}_0 \tau_j$ in direction $\eta$ on the regularity structure $\CT^k$, then as a $\CT$-valued function, $\pH{\tau}$ is independent of $k$. As a result, the only dependency on $k$ lies in the choice of space we consider this pointed modelled distribution to be a member of and hence the corresponding norm. Since this is always clear from context, we suppress the $k$ dependency in the notation.
	
	We are now ready to proceed with the main result of this section.
	
	\begin{proposition}\label{prop: uniform bounds}
		For all $k \in \bZ_+$, there exists a constant $C > 0$ such that uniformly in $m, n \in \bN$ and $1 \le j \le n_0$, we have the bound
		\begin{equ}\label{eq: induction hyp 1}
			\bE \left [\sum_{1 \le l \le j}  2^{np_k|\tau_l|_\fs^{(N-j)}} \left |\bar{\Pi}_0^m \tau_l (\phi_0^{(n)}) \right |^{p_k} + \|\bar{\Gamma}^m \|_{\CT^j}^{p_k} \right ] \le C.
		\end{equ}
		where $p_k = 2^k$.
		
		\sloppy In particular, it follows from Theorem~\ref{theo: Kolmogorov Criterion} that for each fixed $p \in[1, \infty)$,
		$
		\bE \left [ \| \bar{Z}^m \|_{\CT; \ck}^p \right ]
		$
		is bounded uniformly in $m$.
	\end{proposition}
	
	In order to separate the novel core of our proof from the technical details that use only ideas that are already more standard in the theory, we will break much of the induction step down into a variety of more simple preparatory results that will play a role in the inductive loop of the proof of Proposition~\ref{prop: uniform bounds}. In preparation for Section~\ref{sec: convergence}, each of these preparatory statements will also have a version for a difference of two models that will not be needed in the proof of Proposition~\ref{prop: uniform bounds}.
	
	A typical instance of our inductive step proceeds in 3 stages which we will outline below.
	
	In the first stage, we gain control on the norm of $\bar \Pi$ on $V^{(j-1)}$ when viewed as a sector of $\CT^j$. This is the point at which we will require a loss of regularity to apply the Kolmogorov criterion. 
	
	\begin{lemma}\label{lemma: Pi Norm Control}
		Suppose that for a sufficiently large $p \in [1, \infty)$ 
		\begin{equ}
			\bE \left [\sum_{1 \le l \le j - 1}  2^{np|\tau_l|_\fs^{(N-j + 1)}} \left |\bar{\Pi}_0^m \tau_l (\phi_0^{(n)}) \right |^p \right ] \lesssim 1.
		\end{equ}
		Then for each compact set $\ck \subseteq \bR^d$, we have that 
		\begin{equ}
			\bE \left [ \| \bar \Pi \|_{V^{(j-1)}; \ck}^p\right ] \lesssim 1
		\end{equ}
		where $V^{(j-1)}$ is viewed as a sector of $\CT^j$.
		
		Additionally, if there exists a $\theta > 0$ and $\varepsilon > 0$ sufficiently small such that for natural numbers $n_1 \ge n_2$
		\begin{equ}
			\bE \left [\sum_{1 \le l \le j - 1}  2^{np(|\tau_l|_\fs^{(N-j + 1)} - \varepsilon)} \left | \bar{\Pi}_0^{n_1} \tau_l (\phi_0^{(n)}) - \bar{\Pi}_0^{n_2} \tau_l (\phi_0^{(n)}) \right |^p \right ] \lesssim 2^{-pn_2 \theta}
		\end{equ}
		then for each compact set $\ck \subseteq \bR^d$, we have that 
		\begin{equ}
			\bE \left [ \| \bar \Pi^{n_1} - \bar \Pi^{n_2} \|_{V^{(j-1)}; \ck}^p\right ] \lesssim 2^{-p n_2 \theta}	
		\end{equ}
		where $V^{(j-1)}$ is again viewed as a sector of $\CT^j$.
	\end{lemma}
	\begin{proof}
		By Lemma~\ref{lemma: induction algebraic}, $V^{(j-1)}$ is spanned by $\{ \tau_1, \dots, \tau_{j-1}\} \cup \CT_{\mathrm{poly}}$. The result then follows immediately from the Kolmogorov criterion contained in Theorem~\ref{theo: Kolmogorov Criterion} since $|\tau|_\fs^{(k+1)} = |\tau|_\fs^{(k)} + \kappa n_\tau$ so that the relevant assumption amongst \eqref{eq: Kolmogorov single point ass} and \eqref{eq: Kolmogorov multi point ass} is satisfied in each case.
	\end{proof}
	
	The second step in our inductive loop will be to show that control on $\bar{Z}$ on $V^{(j-1)} \subseteq \CT^j$ already implies control on $\|\bar\Gamma\|_{\CT^j}$. The proof of this fact is essentially part of the proof of \cite[Theorem 10.7]{Hai14}, however since our setting is slightly different and we must take more care as to which moments we assume bounds for, we provide a proof here.
	
	\begin{lemma}\label{lemma: Gamma Norm Control}
		Let $p_k = 2^k$ for $k \in \bZ_+$ sufficiently large.
		
		Suppose that for each compact set $\ck \subseteq \bR^d$, we have the bound
		\begin{equ}
			\bE \left [ \| \bar Z^n \|_{\CT^{j-1}; \ck}^{p_{k+1}} \right ]^{1/p_{k+1}} \lesssim 1.
		\end{equ}
		Then for each compact set $\ck$,, $\bE \left [ \| \bar \Gamma \|_{\CT^j; \ck}^{p_k} \right ]^{1/p_k} \lesssim 1$.
		
		If additionally we have that for each compact set $\ck \subseteq \bR^d$ and $n_1 \ge n_2$, there exists $\theta > 0$ such that
		\begin{equ}
			\bE \left [ \| \bar{Z}^{n_1} ; \bar{Z}^{n_2} \|_{\CT^{(j-1)}; \ck}^{p_{k+1}} \right ]^{1/p_{k+1}} \lesssim 2^{- n_2 \theta}
		\end{equ}
		then for each compact set $\ck \subseteq \bR^d$ we have that 
		\begin{equ}
			\bE \left [ \|\bar{\Gamma}^{n_1} - \bar{\Gamma}^{n_2} \|_{\CT^j ; \ck}^{p_k} \right ]^{1/p_k} \lesssim 2^{- n_2 \theta}.
		\end{equ}
	\end{lemma}
	\begin{proof}
		We begin with the case of a single model. 
		
		By Lemma~\ref{lemma: induction algebraic}, it suffices to appropriately control $\|\bar{\Gamma}_{xy}^n \tau_l\|_\zeta^{(j)}$ for $\zeta < |\tau_l|_\fs^{(N-j)}$ and $l \le j$ where $\|\cdot\|_\zeta^{(j)}$ denotes the norm on $\CT_\zeta^j$. 
		
		For $l < j$, we fix a tree $\bar{\tau}$ such that $\zeta = |\bar{\tau}|_\fs^{(N-j)}$. We note that by Assumption~\ref{ass:alg}, $n_{\bar{\tau}}$ is independent of our choice of $\bar{\tau}$ satisfying this condition and so the quantity $n_\zeta \eqdef n_{\bar{\tau}}$ is well-defined. We can then write
		\begin{equs}
			\bE  \Bigg [ \Bigg( \sup_{\substack{x \neq y \\ x, y \in \ck}} & \frac{\| \bar{\Gamma}_{xy}^n \tau_l \|_\zeta^{(j)}}{\|x-y\|^{|\tau_l|_\fs^{(N-j)} - \zeta}} \Bigg)^{p_k} \Bigg ]^{1/p_k}
			\\ &= \bE \left [ \left( \sup_{\substack{x \neq y \\ x, y \in \ck}} \frac{\| \bar{\Gamma}_{xy}^n \tau_l \|_{\zeta + n_\zeta \kappa}^{(j-1)}}{\|x-y\|^{|\tau_l|_\fs^{(N-j + 1)} - \zeta - n_\zeta \kappa}} \|x - y\|^{(n_{\tau_l} - n_\zeta) \kappa} \right)^{p_k} \right ]^{1/p_k}
			\\
			&\lesssim \operatorname{diam}\ck^{(n_{\tau_l} - n_\zeta) \kappa} \bE \left [ \|\bar{\Gamma}^n\|_{\CT^{j-1}; \ck}^{p_k} \right ]^{1/p_k} \lesssim 1
		\end{equs}
		by assumption.
		
		It then remains to consider the case $l = j$. In order to do this, we note that by construction, we are in one of three cases. Either there exist $k,l < j$ such that $\tau_j = \tau_k \cdot \tau_l$ or there exists $k < j$ and $\ft \in \Lab_+$ such that $\tau_j = \CJ^\ft \tau_k$ or $\tau_j = \Xi_\fl$ for $\fl \in \Lab_-$ with $|\fl|_\fs > 0$.
		
		In the latter of these cases, we note $\CQ_{X^k} \bar{\Gamma}_{xy}^n \Xi_\fl = D^k \xi_\fl (x) - (T_y^{|\fl|_\fs} D^k \xi_\fl)(y)$ where $T_y^\gamma f$ denotes the Taylor jet of $f$ at $y$ up to order $\gamma$. In particular, the desired bound in this case follows from Taylor's theorem since by an application of Proposition~\ref{prop: Kolmogorov} the spectral gap inequality implies that $\bE [ \|\xi_\fl\|_{\CC^{|\fl|_\fs}}^{p_{k+1}}]^{1/p_{k+1}} \lesssim 1$.
		
		In the former case, we can write
		\begin{equs}
			\|\bar{\Gamma}_{xy}^n \tau_j \|_\zeta^{(j)} \leq \sum_{\beta \le \zeta} \| \bar{\Gamma}_{xy}^n \tau_k \|_{\beta}^{(j)} \| \bar{\Gamma}_{xy}^n \tau_l \|_{\zeta - \beta}^{(j)}.
		\end{equs} 
		
		As a result, writing $$\Gamma(n, k,j, \zeta) \defeq \sup_{\substack{x \neq y \\ x, y \in \ck}} \frac{ \|\bar{\Gamma}_{xy}^n \tau_k \|_\zeta^{(j)}}{\|x-y\|^{|\tau_j|_\fs^{(N-j)} - \zeta}},$$ we obtain the bound
		\begin{equs}
			\bE  \left [  \Gamma(n, j, j, \zeta)^{p_k} \right ] 
			&\lesssim \sum_{\beta \le \zeta} \bE \left [ \Gamma(n, k, j, \beta)^{p_k} \Gamma(n, l, j, \zeta - \beta)^{p_k} \right ]
			\\
			&\lesssim \sum_{\beta \le \zeta} \bE \left [ \Gamma(n, k, j, \beta)^{p_{k+1}} \right ]^\frac12 \bE \left [ \Gamma(n, l, j, \zeta - \beta)^{p_{k+1}} \right ]^\frac12
			\\
			&\lesssim  \bE \left [ \| \bar{\Gamma}^n \|_{\CT^{j-1}: \ck}^{p_{k+1}} \right] \lesssim 1
		\end{equs}
		where we used a similar argument as in the $l < j$ case to change between choices of degree which creates only a multiplicative factor of $\operatorname{diam}\ck$ to a positive power.
		
		In the latter case where $\tau_j = \CJ^\ft \tau_k$ we are in the setting of the extension theorem \cite[Theorem 5.14]{Hai14} which immediately yields our desired result, up to accommodation of the change between degree assignments, which is done in the same way as in the previous case. 
		
		We now turn to obtaining the desired bounds on $\bar{\Gamma}^{n_1} \tau_l - \bar{\Gamma}^{n_2} \tau_l$ for $l \leq j$. In fact, the bounds for $l < j$ just require us to handle the shift in degree in the same was as in the case of a single model hence we move straight to the case $l = j$. 
		
		First, we note that if $\tau_j = \Xi_\fl$ for $|\fl|_\fs > 0$ with $\fl \in \Lab_-$ then the desired bounds follow in a similar way to the case of a single model since $\bE [ \| \rho^{n_1} \ast \xi_\fl - \rho^{n_2} \ast \xi_\fl \|_{\CC^{|\fl|_\fs}}^p] \lesssim 2^{- p n_2 |\fl|_\fs} \bE [ \|\xi_\fl\|_{\CC^{|\fl|_\fs}}^p]$.
		
		Therefore we can again write without loss of generality $\tau_j = \tau_l \cdot \tau_k$ or $\tau_j = \CI^\ft \tau_k$  for $l, k < j$. In the latter case, the desired bound again follows from the Extension Theorem \cite[Theorem 5.14]{Hai14} so we consider only the former case. Here, we write
		\begin{equs}
			\bE & \left [ \left ( \sup_{\substack{x \neq y \\ x, y \in \ck}} \frac{ \|\bar{\Gamma}^{n_1}_{xy} \tau_j - \bar{\Gamma}^{n_2}_{xy} \tau_j \|_\eta^{(j)}}{\|x-y\|^{|\tau_j|_\fs^{(N-j)} - \eta}} \right)^{p_k} \right ] 
			\lesssim T_1 + T_2
		\end{equs}
		where 
		\begin{equs}
			T_1 &= \sum_{\zeta \le \eta} \bE \left [ \Gamma(n_2, k, j, \zeta)^{p_{k+1}} \right ]^\frac12 \bE \left [ \left ( \sup_{\substack{x \neq y \\ x, y \in \ck}} \frac{ \|\bar{\Gamma}^{n_1}_{xy} \tau_l - \bar{\Gamma}^{n_2}_{xy} \tau_l \|_{\eta - \zeta}^{(j)}}{\|x-y\|^{|\tau_l|_\fs^{(N-j)} - \eta + \zeta}} \right )^{p_{k+1}} \right ]^\frac12
			\\
			T_2 &= \sum_{\zeta \le \eta} \bE \left [ \Gamma(n_1, l, j, \zeta)^{p_{k+1}} \right ]^\frac12 \bE \left [ \left ( \sup_{\substack{x \neq y \\ x, y \in \ck}} \frac{ \|\bar{\Gamma}^{n_1}_{xy} \tau_k - \bar{\Gamma}^{n_2}_{xy} \tau_k \|_{\eta - \zeta}^{(j)}}{\|x-y\|^{|\tau_l|_\fs^{(N-j)} - \eta + \zeta}} \right )^{p_{k+1}} \right ]^\frac12.
		\end{equs}
		
		Making similar adjustments as in the case of a single model to accommodate the change in degree, we obtain the bound
		\begin{equ}
			T_1 + T_2 \lesssim  \bE\left[ \left (\| \bar{\Gamma}^{n_2}\|_{\CT^{j-1}; \ck}+ \| \bar{\Gamma}^{n_1}\|_{\CT^{j-1}; \ck} \right )^{p_{k+1}} \right ]^\frac12 \bE \left [ \| \bar{\Gamma}^{n_1}- \bar{\Gamma}^{n_2} \|_{\CT^{j-1}; \ck}^{p_{k+1}} \right ]^\frac12
		\end{equ}
		which is a bound of the desired order since the first term in the product on the right hand side is bounded a constant independent of $n_1, n_2$ by the result in the case of a single model.
	\end{proof}
	
	The final step in a typical instance of the inductive loop is to convert this norm control into control on 
	\begin{equ}
		\bE \left [  2^{np|\tau_j|_\fs^{(N-j)}} \left |\bar{\Pi}_0^m \tau_j (\phi_0^{(n)}) \right |^p \right].
	\end{equ}
	Here, our techniques will differ depending on the sign of the degree of $\tau_j$. In the case where $\tau_j$ has negative degree, we will apply the spectral gap inequality. This case is the meat of the proof of Proposition~\ref{prop: uniform bounds}.
	If this degree is positive then control automatically follows from the rigid structure that is imposed on the action of the model on symbols of positive degree as a result of their analytic properties, as shown in the following lemma. 
	
	\begin{lemma}\label{lemma: pos hom control}
		Suppose that $|\tau_j|_\fs^{(N-j)} > 0$ and that for some $k \in \bZ_+$ and for each compact set $\ck \subseteq \bR^d$ we have that
		\begin{equ}
			\bE \left [ \|\bar \Pi^m\|_{V^{(j-1)}; \ck}^{p_{k+1}} + \|\bar{\Gamma}^m \|_{\CT^{j}; \ck}^{p_{k+1}} \right ]^{1/{p_{k+1}}} \lesssim 1.
		\end{equ}
		Then $\bE \left [ 2^{np_k|\tau_j|_\fs^{(N-j)}} \left |\bar{\Pi}_0^m \tau_j (\phi_0^{n}) \right |^{p_k}  \right ]^{1/p_k} \lesssim 1$.
		
		If additionally for each compact set $\ck \subseteq \bR^d$ and $n_1 \ge n_2$ we have that for some $\theta > 0$,
		\begin{equ}
			\bE \left [ \|\bar{\Pi}^{n_1} - \bar{\Pi}^{n_2} \|_{V^{(j-1)}; \ck}^{p_{k+1}} + \|\bar{\Gamma}^{n_1} - \bar{\Gamma}^{n_2}\|_{\CT^j; \ck}^{p_{k+1}} \right ]^{{p_{k+1}}} \lesssim 2^{-n_2 \theta}
		\end{equ}
		then we have that 
		\begin{equ}
			\bE \left [ 2^{-np_k|\tau_j|_\fs^{(N-j)}} \left |\bar{\Pi}_0^{n_1} \tau_j (\phi_0^{n}) - \bar{\Pi}_0^{n_2} \tau_j (\phi_0^{n}) \right |^{p_k} \right ] ^{1/p_k} \lesssim 2^{n_2 \theta}.
		\end{equ}
	\end{lemma}
	\begin{proof}
		By \cite[Proposition 3.31]{Hai14} and its proof, we only need to see that if $f_j^x(y) = \bar \Gamma^m_{yx} \tau_j - \tau_j$ then $f_j^x$ is valued in $V^{(j-1)}$. By definition of $V^{(j)}$, $f_j^x$ is certainly valued in $V^{(j)}$. In addition, it follows from the construction of the structure group that $f_j^x(y)$ is a linear combination of trees with strictly fewer edges than are in $\tau_j$. This implies the desired result.
		
		The proof in the case of two models is similar, using our uniform bounds on the individual models when needed.
	\end{proof}
	
	With these preparatory results in place, we will turn to the core of the proof of Proposition~\ref{prop: uniform bounds}.
	\begin{proof}[Proof of Proposition~\ref{prop: uniform bounds}]
		We fix $p_k$ as in the statement. We let $N_1$ be an upper bound on the number of applications of Lemmas~\ref{lemma: Pi Norm Control}, \ref{lemma: Gamma Norm Control} and~\ref{lemma: pos hom control} and of Cauchy-Schwarz that are required in any of the finitely many induction steps required to exhaust the structure in our proof. We choose not to write this value explicitly since it depends on the power $k$ appearing in Proposition~\ref{prop: H control} and it is always clear that its exact value is unimportant to us. Let $N_2 = n_0 N_1$.
		
		We will prove by induction in $j$ that if $k(j) = k + N_2 - N_1 j$ then
		\begin{equ}\label{eq: induction hyp 1 strong}
			\bE \left [2^{np_{k(j)} |\tau_l|_\fs^{(N-j)}} \left |\bar{\Pi}_0^m \tau_l (\phi_0^{n}) \right |^{p_{k(j)}} + \|\bar{\Gamma}^m \|_{\CT^j}^{p_{k(j)}} \right ] \le C.
		\end{equ}
		which is stronger than \eqref{eq: induction hyp 1} by our choice of $k(j)$.
		
		Since $\CT^0 = \Tpoly$, the case $j = 0$ is trivial so that it remains to check the induction step.
		
		We therefore assume the bound \eqref{eq: induction hyp 1 strong} is known up to $j-1$ and aim to prove that it holds for $j$. Since all of our bounds will be uniform in choice of $m \in \bN$, we will suppress the $m$-dependency in the notation in what follows. 
		
		By Lemmas~\ref{lemma: Pi Norm Control}, \ref{lemma: Gamma Norm Control} and~\ref{lemma: pos hom control}, we can assume that $|\tau_j|_\fs^{(N-j)} < 0$ and that for each compact set $\ck \subseteq \bR^d$ we have that
		\begin{equ}
			\bE \left [ \| \bar \Pi \|_{V^{(j-1)}; \ck}^{p_{k(j) - 3}}+ \|\bar \Gamma \|_{\CT^j; \ck}^{p_{k(j) - 3}} \right ] \lesssim 1.
		\end{equ}
		
		We will now apply our spectral gap assumption. This yields
		\begin{equs} \label{eq: Induction SGap}
			\bE \left[ \left | \bar{\Pi}_0 \tau_j (\phi_0^{n}) \right |^p \right ]^{1/p} \lesssim \bE \left [ \bar{\Pi}_0 \tau_j (\phi_0^{n}) \right ] + \bE \left [ \left \| \frac{\partial \left [ \bar{\Pi}_0 \tau_j (\phi_0^{n}) \right ]}{\partial \xi} \right \|_{H^{- \reg}(\Lab_-)}^p \right ]^{1/p}
		\end{equs}
		where we will choose $p = p_{k + N - 4j}$.
		
		We consider the two terms on the right hand side separately and begin with the first of these two terms.
		
		Here the idea is to write $\phi_0^{n} = \phi_0^{0} + \sum_{k=0}^{n-1} (\phi_0^{k+1} - \phi_0^{k})$. Since $\bE \left [ \bar{\Pi}_0 \tau_j (\phi_0^{0}) \right ] = 0$ by the definition of the $\bBPHZ$ model, it remains to consider $$\bE \left [ \bar{\Pi}_0 \tau_j (\phi_0^{k+1} - \phi_0^{k}) \right ].$$
		
		Here, we note that it follows from \eqref{eq: translation invariance} that  for each $m$, $\bE[\bar{\Pi}_z \tau (\phi_z^m)]$ is independent of $z$. In particular, since $R^k$ integrates to $0$ we can write
		\begin{equs}
			\bE \left [ \bar{\Pi}_0 \tau_j (\phi_0^{k+1} - \phi_0^{k}) \right ] = \int \bE \left [ \bar{\Pi}_y \left ( \bar{\Gamma}_{y0} \tau_j - \tau_j \right )(\phi_y^{k+2}) \right ] R^k(y) \,dy\;.
		\end{equs}
		where $R$ is as in the proof of Lemma~\ref{lemma: pointed candidate exists}.
		
		Since $|y|_\fs \lesssim 2^{-k}$ on the domain of integration, it then immediately follows that
		\begin{equ}
			\left | \bE \left [ \bar{\Pi}_0 \tau_j (\phi_0^{k+1} - \phi_0^{k}) \right ] \right | \lesssim 2^{-k |\tau_j|_\fs^{(N-j)}} \bE \left [ \| \bar{\Pi}\|_{V^{(j-1)}; \ck}  \cdot \|\bar{\Gamma}\|_{\CT^j; \ck} \right ] 
		\end{equ}
		which by Cauchy--Schwarz is bounded by a constant multiple of $2^{-k|\tau_j|_\fs^{(N-j)}}$. Since $|\tau_j|_\fs^{(N-j)} < 0$ by assumption, summing this bound over the regime $k \le n$ yields the bound
		\begin{equ}
			\bE \left [ \bar{\Pi}_0 \tau_j (\phi_0^{n}) \right ] \lesssim 2^{-n |\tau_j|_\fs^{(N-j)}}
		\end{equ}
		as required.
		
		Hence, we now turn to the derivative term. Here, we begin by applying duality to write 
		\begin{equ}
			\left \| \frac{\partial \left [ \bar{\Pi}_0 \tau_j (\phi_0^{n}) \right ]}{\partial \xi} \right \|_{H^{- \reg}(\Lab_-)} = \sup_{\|\eta\|_{\CH} \le 1} \left | D_{\eta} \bar{\Pi}_0 \tau_j (\phi_0^{n}) \right |.
		\end{equ}
		
		From the results of Section~\ref{sec: Frechet} and a density argument, the right-hand side can be written as 
		\begin{equs}
			\sup_{\substack{\| \eta \|_{\CH} \le 1 \\ \eta \in \nice}} \left | \bar{\CR} \bar{H}_{\tau_j}^{0, \eta} (\phi_0^{n}) \right |
		\end{equs}
		where in the case where $\gamma_{\tau_j} \le 0$, the reconstruction should be interpreted as our provided candidate for the reconstruction given in Lemmas~\ref{lemma: eta candidate} and~\ref{lemma: pointed candidate exists}.
		
		In the case $\gamma_{\tau_j} > 0$, by Theorem~\ref{theo:pointedReconstruction}, since $\deg_2 \tau_j - |\fs|/2 = |\tau_j|_\fs^{(N-j)}$, we obtain a bound of order
		\begin{equ}
			 2^{-n |\tau_j|_\fs^{(N-j)}} \|\bar{\Pi}\|_{V^{(j-1)}; \bar{\ck}} \sup_{\|\eta \|_{\CH} \le 1} \$\bar{H}_{\tau_j}^{0,\eta}\$_{2, \gamma_{\tau_j}, \deg_2 \tau_j; x}.
		\end{equ}
		It follows by a simple induction that $\pH{\tau}$ is valued in $V^{(j-1)}$ so that the desired bound follows from the previous bounds, Proposition~\ref{prop: H control} and Cauchy--Schwarz.
		
		In the case $\gamma_\tau \le 0$, the desired bound follows from the definition of a candidate for the pointed reconstruction operator together with Lemma~\ref{lemma: eta candidate}, Lemma~\ref{lemma: pointed candidate exists} and Proposition~\ref{prop: H control} where the last of these is used to control the corresponding constant in the application of Lemma~\ref{lemma: pointed candidate exists}.
		
		This completes the inductive proof that $\eqref{eq: induction hyp 1}$ holds for all $j$. The remaining statement then follows immediately from Theorem~\ref{theo: Kolmogorov Criterion}.
	\end{proof}
	
	\section{Convergence of the $\bBPHZ$ Model}\label{sec: convergence}
	
	In this section, we turn to establishing the convergence of the $\bBPHZ$ models as mollification is removed. By the results of the previous section, we will be able to assume uniform control on the norms of the models appearing in this section. Additionally many of the steps in our inductive loop will have no real difference to those of the previous section, only requiring that we apply the versions of our statements that were obtained for the differences of models.
	
	Despite this, the techniques of the previous section are not enough to establish convergence since our approach to estimating the Fr\'echet derivative of the model will fail. To illustrate the reason for this, we point out that whilst it is the case that $\sup_{\|\eta\|_\CH \le 1} \| \varrho^n \ast \eta_\ft \|_{\CB_{2,2}^{\reg \ft}(\ck)}$ is bounded uniformly in $n$, it is not the case that $\sup_{\|\eta\|_\CH \le 1} \|\varrho^n \ast \eta_\ft - \eta_\ft\|_{\CB_{2,2}^{\reg \ft}(\ck)}\to 0$ as $n\to\infty$. Consequently, we do not expect $\pHa{\tau} - \pHb{\tau}$ to have vanishing norm as $n_1, n_2 \to \infty$ uniformly over the necessary set of choices of $\eta$ in our argument even in the case where $\tau = \Xi_\ft$.
	
	In order to account for this, we observe that it is the case that for $|\ft|_\fs > -|\fs|/2$, $\sup_{\|\eta\|_\CH\le 1} \|D^k \varrho^{n_1} \ast \eta_\ft - D^k \varrho^{n_2} \ast \eta_\ft \|_{H^{- \kappa}} \to 0$ as $n_1, n_2 \to \infty$ for any choice of $\kappa > 0$ and $|k|_\fs < \reg \ft$. 
	
	Since $\pH{\Xi_\ft} = 0$ if $|\ft|_\fs \le - |\fs|/2$, for $\tau = \Xi_\ft$ we do then know that the coefficients of $\pHa{\tau} - \pHb{\tau}$ have vanishing $H^{- \kappa}$ norm as $n_1, n_2 \to \infty$ with the desired uniformity in $\eta$.
	
	To formalise this statement, we recall some of the constructions from Section~\ref{section: Modelled Distributions with Sobolev Coefficients} and fix corresponding notation. For each $\zeta \in \CA$, we will assume that $B_\zeta$ is the set of trees of degree $\zeta$ so that $B \eqdef \cup_{\zeta \in \CA} B_\zeta$ forms a basis of $\CT$ which we will assume to be orthonormal without any loss of generality. For $f \in \CD_2^\gamma$, we remind the reader that $\|f\|_{\gamma, - \kappa, 2; \ck}$ then denotes a negative regularity Sobolev norm on the coefficients of $f$ in the directions in $B$; see Definition~\ref{def: sob mod norm}.
	
	\begin{lemma}\label{lemma: sob H base case}
		For any $\kappa > 0$, there exists a $\theta > 0$ such that for each compact set $\ck \subseteq \bR^d$ and $\ft \in \Lab_-$,
		\begin{equ}
			\sup_{\|\eta\|_\CH\le 1} \$ \pHa{\Xi_\ft} - \pHb{\Xi_\ft} \$_{\gamma_\ft, - \kappa, 2; \ck} \lesssim 2^{-n_2 \theta}.
		\end{equ}
	\end{lemma}
	\begin{proof}
		For $f \in L^p$, let $R_n f = \varrho^n \ast f$. Then for each $\kappa > 0$ and compact set $\ck \subseteq \bR^d$, a straightforward computation shows that
		\begin{equ}
			\sup_{\|f\|_{L^2} = 1} \|f - R_n f \|_{H^{- \kappa}} \lesssim 2^{-n \kappa /2}. 
		\end{equ}
		The result is then immediate from the preceding discussion.
	\end{proof}
	
	In practice, we will want to take $\kappa$ to be sufficiently small so that we can apply various duality results. Additionally, we will often have to replace $\theta$ with some smaller value when passing this bound onto bounds for the pointed modelled distributions corresponding to trees with more edges. As a result, we will specify a value for neither $\kappa$ nor $\theta$ here. Instead, we will assume that $\kappa > 0$ is fixed to be sufficiently small to satisfy the various implicit constraints throughout this section and will make no attempts to be optimal in our obtained rate of convergence. 
	
	The various reconstruction results of Section~\ref{section: Modelled Distributions with Sobolev Coefficients} then give us hope to obtain the desired estimates on $\bar{\CR}^{n_1} \pHa{\tau} - \bar{\CR}^{n_2} \pHb{\tau}$. The missing ingredient, and hence our next order of business in this section, is the propagation of our bounds on the negative Sobolev norms of coefficients to all levels of the regularity structure. 
	
	We begin our process of controlling these Sobolev type norms by considering their behaviour under multiplication. We have the following lemma.
	
	\begin{lemma}
		Fix $k \in \bZ_+$. Suppose that there exists a $\theta > 0$ such that for each $n_1 \ge n_2$ and compact set $\ck \subseteq \bR^d$ we have the bound
		\begin{equ}
			\bE \left [ \|\bar \Gamma^{n_1} - \bar \Gamma^{n_2} \|_{\CT; \ck}^{p_{k+1}} + \max_{\sigma \in \{\tau, \bar \tau\}} \sup_{\|\eta\|_\CH \le 1} \| \pHao{\sigma} - \pHbo{\sigma} \|_{\gamma_\sigma, - \kappa, 2; \ck}^{p_{k+1}}  \right ] \lesssim 2^{-{p_{k+1}} n_2 \theta}.
		\end{equ}
		Then it follows that for the same parameters we have that
		\begin{equ}
			\bE \left [ \sup_{\|\eta\|_\CH \le 1} \| \pHao{\tau \bar \tau} - \pHbo{\tau \bar \tau} \|_{ \gamma_{\tau \bar \tau}, - \kappa, 2; \ck}^{p_{k}} \right ] \lesssim 2^{-p_k n_2 \theta}.
		\end{equ}
	\end{lemma}
	\begin{proof}
		For $\zeta < \gamma_{\tau \bar{\tau}}$ and $a \in B_\zeta$, we write
		\begin{equs}\label{eq: Hsobprod equation}
			\CQ_a \Bigg ( \bar{H}_{\tau \bar \tau; n_1}^{0, \eta} & - \bar{H}_{\tau \bar \tau; n_2}^{0, \eta} \Bigg )  \\ & = \CQ_a \left [ \left ( \bar{H}_{\tau ; n_1}^{0, \eta} f_0^{\bar \tau; n_1} - \bar{H}_{\tau ; n_2}^{0, \eta} f_0^{\bar{\tau}; n_2} \right ) + \left ( \bar{H}_{\bar \tau ; n_1}^{0, \eta} f_0^{\tau; n_1} - \bar{H}_{\bar \tau ; n_2}^{0, \eta} f_0^{\tau; n_2}\right ) \right ],
		\end{equs}
		where $f_0^{\tau; n_i}(y) = \bar{\Gamma}_{y0}^{n_i} \tau$ and $\CQ_a$ is the projection onto the span of $a$.
		
		We then have that 
		\begin{equs}
			& \CQ_a \left ( \bar{H}_{\tau ; n_1}^{0, \eta} f_0^{\bar \tau; n_1} - \bar{H}_{\tau ; n_2}^{0, \eta} f_0^{\bar{\tau}; n_2} \right ) \\ & = \sum_{b, c} \CQ_b \left ( \bar{H}_{\tau; n_1}^{0, \eta} - \bar{H}_{\tau; n_2}^{0, \eta} \right ) \cdot \CQ_c f_0^{\bar \tau; n_2} + \CQ_b \bar{H}_{\tau; n_1}^{0, \eta} \cdot \CQ_c \left ( f_0^{\bar \tau; n_1} - f_0^{\bar \tau; n_2} \right )
		\end{equs}
		where the sum is over those trees $b,c \in \bigcup_{\eta \le \zeta} B_\eta$ such that $a = b \cdot c$.
		
		For the first term in this sum, we note that if $\kappa > 0$ is smaller than the smallest gap between consecutive homogeneities in our structure then $\CQ_c f_0^{\bar{\tau}; n_2}$ is in $C^{\kappa}(\ck)$ with a norm of order $\|\bar{\Gamma}\|_{\CT^j; \ck}^2$. Therefore, since multiplication is a bounded operator from $\CB_{2,2}^{-\kappa}(\ck) \times C^\kappa(\ck)$ to $\CB_{2,2}^{-\kappa}(\ck)$, we obtain a bound on the first term of order  $\|\bar{\Gamma}\|_{\CT^j; \ck}^2 \| \bar{H}_{\tau; n_1}^{0, \eta} - \bar{H}_{\tau; n_2}^{0, \eta}\|_{\gamma_\tau, - \kappa, 2; \ck}$. This is a bound of the correct order by an application of Cauchy--Schwarz and Proposition~\ref{prop: uniform bounds}.
		
		For the term $ \CQ_b \bar{H}_{\tau; n_2}^{0, \eta} \cdot \CQ_c \left ( f_0^{\bar \tau; n_1} - f_0^{\bar \tau; n_2} \right )$, we write
		\begin{equs}
			\| \CQ_c \left( f_0^{\bar \tau; n_1} - f_0^{\bar \tau; n_2} \right) \|_{L^{\infty}(\ck)} & = \sup_{z \in \ck}\left | \CQ_c \left ( \bar{\Gamma}_{z0}^{n_1} \bar \tau - \bar \Gamma_{z0}^{n_2} \bar \tau \right ) \right |  
			\\
			& \lesssim \|\bar{\Gamma}^{n_1} - \bar{\Gamma}^{n_2} \|_{\CT; \ck}
		\end{equs}
		and then note that our desired bound follows from the fact that the Sobolev norm is bounded by the $L^2$-norm. Indeed, we then have
		\begin{equs}
			\bE \Bigg [ \sup_{\|\eta\|_\CH \le 1} \big \|  \CQ_b \bar{H}_{\tau; n_2}^{0, \eta} & \cdot \CQ_c \left ( f_0^{\bar \tau; n_1} - f_0^{\bar \tau; n_2} \right ) \big \|_{\CB_{2,2}^{-\kappa}(\ck)}^{p_k} \Bigg ] \\ & \lesssim \bE \left [\sup_{\|\eta\|_\CH \le 1} \left \|  \CQ_b \bar{H}_{\tau; n_2}^{0, \eta} \cdot \CQ_c \left ( f_0^{\bar \tau; n_1} - f_0^{\bar \tau; n_2} \right ) \right \|_{L^2(\ck)}^{p_k} \right ]
			\\
			& \lesssim \bE \left [\sup_{\|\eta\|_\CH \le 1} \| \CQ_b H_{\tau; n_2}^{0, \eta} \|_{L^2(\ck)}^{p_{k+1}} \right ]^{1/2} \bE \left [ \|\bar{\Gamma}^{n_1} - \bar{\Gamma}^{n_2} \|_{\CT; \ck}^{p_{k+1}} \right ]^{1/2}
		\end{equs}
		which is a bound of the right order by Proposition~\ref{prop: uniform bounds}.
		
		The second term on the right hand side of \eqref{eq: Hsobprod equation} is bounded similarly by swapping the roles of $\tau$ and $\bar \tau$ in the above.
	\end{proof}
	
	It remains to consider the effect of the abstract integration operator. We again write $\bar{\CF}_- = \{\tau_1, \dots, \tau_n\}$ where the trees are as in the previous section. 
	
	In this setting, we break our consideration into three cases. We write $\tau_j = \CI^\fl \tau_l$ for some $l < j$. We then consider separately the case where $\gamma_{\tau_l} > 0$, where $\tau_l = \Xi_\ft$ for $\ft \in \Lab_-$ such that $|\ft|_\fs \le - |\fs|/2$ and finally the remaining symbols for which $\gamma_{\tau_l} \le 0$.
	
	In the first case, we will apply Theorem~\ref{theo: PointedSobSchauder} to obtain our desired result. In this result, we work on the regularity structure $\CT^j$. 
	
	\begin{lemma}\label{lemma: H sob integration 1}
		Fix $k \in \bZ_+$. Suppose that $\tau_j = \CI^\fl \tau_l$ for $l < j$ and that $\gamma_{\tau_l} > 0$. Suppose additionally that there exists a $\theta > 0$ such that for each compact $\ck \subseteq \bR^d$ we have the bound
		\begin{equs}
			\bE \Big [\sup_{\|\eta\|_\CH \le 1} \|\pHao{\tau_l} - \pHbo{\tau_l} \|_{\gamma_{\tau_l}, - \kappa, 2; \ck}^{p_{k+1}} + \|\bar \Pi^{n_1} - \bar{\Pi}^{n_2} \|_{V^{(j-1)}; \ck}^{p_k} & + \| \bar\Gamma^{n_1} - \bar \Gamma^{n_2} \|_{\CT^j ; \ck}^{p_{k+1}} \Big ] \\& \lesssim 2^{- p_{k+1} n_2 \theta}.
		\end{equs}
		Then there exists a $\bar \theta > 0$ such that for each compact $\ck \subseteq \bR^d$ we have that
		\begin{equ}
			\bE \left [ \sup_{\|\eta\|_\CH \le 1}\| \pHao{\tau_j} - \pHbo{\tau_j} \|_{\gamma_{\tau_j}, - \kappa, 2; \ck}^{p_k} \right ] \lesssim 2^{-p_k n_2 \bar \theta}.
		\end{equ}
	\end{lemma}
	\begin{proof}
		This result is an immediate corollary of Theorem~\ref{theo: PointedSobSchauder} and the uniform bounds of the previous section.
	\end{proof}
	
	We now turn to the remaining two cases. We begin with the case where $\tau_j = \CI^\fl \tau_l$ and $\gamma_{\tau_l} \le 0$. The key estimate that we will need is the contents of the following result.
	
	\begin{lemma}\label{lemma: base sob estimate}
		Suppose that $\reg \ft \le 0$ so that $|\ft|_\fs < - |\fs|/2$. Then for any $\varepsilon > 0$, $n_1 \ge n_2$ and $\ck \subseteq \bR^d$ we have that
		\begin{equ}	
			\sup_{\|\eta\|_\CH \le 1} \| \varrho^{n_1} \ast \eta_\ft - \varrho^{n_2} \ast \eta_\ft \|_{\CB_{2,\infty}^{\reg \ft - \varepsilon}(\ck)}  \lesssim 2^{- n_2 \varepsilon}.
		\end{equ} 
	\end{lemma}
	\begin{proof}
		It will suffice to show that $\sup_{\|\eta\|_\CH \le 1} \| \sup_{\psi \in \CB^r} \scal{\varrho^{n_1} \ast \eta_\ft - \eta_\ft, \psi_x^\lambda} \|_{L^2(\ck; dx)} \lesssim \lambda^{\reg \ft - \varepsilon} 2^{-n_1 \varepsilon}$.
		
		For $\lambda \ge 2^{-n_1}$, we write 
		\begin{equs}
			|\psi^\lambda(y) - \varrho^{n_1} \ast \psi^\lambda(y)| & \le \int |\psi^\lambda(y) - \psi^\lambda(y - z)| \varrho^{n_1}(z) dz
			\\
			& \lesssim \int \lambda^{- |\fs| - 1} |z| \rho^{n_1}(z) dz
			\lesssim \lambda^{- |\fs| - 1} 2^{-n_1}.
		\end{equs}
		As a result, for any $\varepsilon > 0$, $\lambda^{\varepsilon} 2^{n_1 \varepsilon} (\psi^\lambda - \varrho^{n_1} \ast \psi^\lambda)$ is a test function at scale $2\lambda$ up to a fixed multiplicative constant. Since we have that $\sup_{\|\eta\|_\CH \le 1} \|\eta\|_{\CB_{2, \infty}^{\reg \ft}(\ck)} \lesssim 1$, this then yields the bound
		\begin{equ}
			\sup_{\|\eta\|_\CH \le 1} \| \sup_{\psi \in \CB^r} \scal{\varrho^{n_1} \ast \eta_\ft - \eta_\ft, \psi_x^\lambda} \|_{L^2(\ck; dx)} \lesssim \lambda^{\reg \ft - \varepsilon} 2^{-n_1 \varepsilon}.
		\end{equ}
		
		It remains to consider the regime where $\lambda < 2^{-n_1}$. In this regime, we apply Jensen's inequality to obtain
		\begin{equ}
			\| \sup_{\psi \in \CB^r} \scal{ \eta_\ft, \varrho^{n_1} \ast \psi_x^\lambda} \|_{L^2(\ck; dx)}^2 \lesssim \iint \sup_{\psi \in \CB^r} |\scal{\eta_\ft, \psi_y^\lambda}|^2 \varrho_x^{n_1}(y) \, dx \, dy \lesssim \lambda^{ \reg \ft}
		\end{equ}
		which implies the desired bound. The term where no mollification is present is treated similarly. 
	\end{proof}
	
	With this estimate in hand, we are ready to obtain our integration result for the case where $\tau_l = \Xi_\ft$ and $\tau_j = \CI^\fl \tau_l$.
	
	\begin{lemma}\label{lemma: sob H base integral}
		Suppose that $\gamma_\ft \le 0$. Then there exists a $\theta > 0$ such that for each compact subset $\ck \subseteq \bR^d$ we have that
		\begin{equ}
			\sup_{\|\eta\|_\CH \le 1} \| \pHao{\CI^\fl \Xi_\ft} - \pHbo{\CI^\fl \Xi_\ft} \|_{\gamma_{\CI_\fl\tau}, - \kappa, 2; \ck} \lesssim 2^{-n_2 \theta}
		\end{equ}
	\end{lemma}
	\begin{proof}
		Since $\pHao{\Xi_\ft} = 0$ in this setting, we have that 
		\begin{equs}
			\pHao{\CI^\fl \Xi_\ft}(y) & - \pHbo{\CI^\fl \Xi_\ft}(y) \\ & = \sum_{|k|_\fs < \gamma_\ft + |\fl|_\fs} \frac{X^k}{k!} \scal{\varrho^{n_1} \ast \eta_\ft - \varrho^{n_2} \ast \eta_\ft, \partial^k K(y - \cdot)} \\ & \quad - \sum_{|k|_\fs < |\ft|_\fs + |l|_\fs} \frac{(X + y - x)^k}{k!} \scal{\varrho^{n_1} \ast \eta_\ft - \varrho^{n_2} \ast \eta_\ft, \partial^k K(x- \cdot)}.
		\end{equs}
		For terms appearing in the first sum, Lemma~\ref{lemma: base sob estimate} imply a bound on the $L^2(\ck; dy)$ norm as in the treatment of the corresponding terms in the proof of Theorem~\ref{theo: PointedSobSchauder}. For the terms in the second sum, we assume that $\varepsilon$ is chosen to be sufficiently small so that $\gamma_\ft + \varepsilon \le \reg \ft$. Then, applying the embedding $\CB_{2, \infty}^{\reg \ft - \varepsilon} \hookrightarrow \CB_{\infty, \infty}^{|\ft|_\fs}$ and a similar treatment yields the desired bound.
	\end{proof}
	
	This leaves us to consider the case where $\gamma_{\tau_l} \le 0$ and $\tau_l \neq \Xi_\ft$. Here, the key probabilistic estimate is contained in the following Lemma.
	
	\begin{lemma}\label{lemma: sob H bad case}
		Suppose that $k \in \bZ_+$, $\gamma_{\tau_j} \le 0$ and that we can write $\tau_j = \Xi_\ft \cdot \tau_l$ where $\tau_l = \prod_{i=1}^k \CI^{\fl_i} \sigma_i$. Suppose additionally that there exists a $\theta > 0$ such that for each compact $\ck \subseteq \bR^d$ we have the bound
		\begin{equ}
			\bE \left [\|\bar \Pi^{n_1} - \bar{\Pi}^{n_2} \|_{V^{(j-1)}; \ck}^{p_{k+1}} + \| \bar\Gamma^{n_1} - \bar \Gamma^{n_2} \|_{\CT^j ; \ck}^{p_{k+1}} \right ] \lesssim 2^{- {p_{k+1}} n_2 \theta}.
		\end{equ}
		Then there exists a $\bar \theta > 0$ such that for each compact $\ck \subseteq \bR^d$ we have that
		\begin{equ}
			\bE \left [ \sup_{\|\eta\|_\CH \le 1} \| (\varrho^{n_1} \ast \eta_\ft) \cdot \bar\Pi_x^{n_1} \tau_l - (\varrho^{n_2} \ast \eta_\ft) \cdot \bar \Pi_x^{n_2} \tau_l \|_{\CB_{2,\infty}^{\reg \ft}|; \ck)}^{p_k} \right ] \lesssim 2^{- p_k n_2 \bar \theta}.
		\end{equ}
	\end{lemma}
	\begin{proof}
		We write
		\begin{equs}
			(\varrho^{n_1} \ast \eta_\ft) \cdot \bar\Pi_x^{n_1} \tau_l & - (\varrho^{n_2} \ast \eta_\ft) \cdot \bar \Pi_x^{n_2} \tau_l \\ & = (\varrho^{n_1} \ast \eta_\ft - \varrho^{n_2} \ast \eta_\ft) \cdot \bar \Pi_x^{n_1} \tau_l + (\varrho^{n_2} \ast \eta_\ft) \cdot \left ( \bar \Pi_x^{n_1} \tau_l - \bar \Pi_x^{n_2} \tau_l \right )
		\end{equs}
		and estimate each term on the right hand side separately. 
		
		The desired estimate on the first term follows from Lemma~\ref{lemma: sob H base case} since by Lemma~\ref{lemma: bad tree identification} we can apply the multiplication result contained in \cite[Theorem 3.11]{BL21} and by Proposition~\ref{prop: uniform bounds} the part of the resulting bound that comes from the model is uniformly controlled.
		
		This leaves us to consider the second term. Here we observe that $\bar \Pi_x^{n_1} \tau_l - \bar \Pi_x^{n_2} \tau_l$ has a $\CC^{\alpha_{\tau_l}}(\ck)$ norm of order $\|\bar \Pi^{n_1} - \bar \Pi^{n_2} \|_{V^{(j-1)}; \ck}$. Therefore, the same approach as above yields a bound of order $\bE \left [ \| \bar \Pi^{n_1} - \bar \Pi^{n_2}\|_{V^{(j-1)}}^{p_k} \right ]$ which is of the correct order by hypothesis.
	\end{proof}
	
	At this point, we are ready to obtain our final integration result.
	
	\begin{lemma}\label{lemma: sob H bad integral}
		Suppose that $k \in \bZ_+$, $\gamma_{\tau_j} \le 0$ and that we can write $\tau_j = \Xi_\ft \cdot \tau_l$ where $\tau_l = \prod_{i=1}^k \CI^{\fl_i} \sigma_i$. Suppose additionally that there exists a $\theta > 0$ such that for each compact $\ck \subseteq \bR^d$ we have the bound
		\begin{equ}
			\bE \left [\|\bar \Pi^{n_1} - \bar{\Pi}^{n_2} \|_{V^{(j-1)}; \ck}^{p_{k+1}} + \| \bar\Gamma^{n_1} - \bar \Gamma^{n_2} \|_{\CT^j ; \ck}^{p_{k+1}} \right ] \lesssim 2^{- p_{k+1} n_2 \theta}.
		\end{equ}
		Then there exists a $\bar \theta > 0$ such that for each compact $\ck \subseteq \bR^d$ we have that
		\begin{equ}
			\bE \left [\sup_{\|\eta\|_\CH \le 1} \| \pHao{\CI^\fl \tau_j} - \pHbo{\CI^\fl  \tau_j} \|_{\gamma_{\tau_j}, - \kappa, 2; \ck}^{p_k} \right ] \lesssim 2^{-p_k n_2 \theta}.
		\end{equ}
	\end{lemma}
	\begin{proof}
		The proof is much the same as the proof of Theorem~\ref{theo: PointedSobSchauder}. The only real modification is that at points in the proof of that result where we applied Theorem~\ref{theo: PointedSobReconstr} we will instead use the bound provided in Lemma~\ref{lemma: sob H bad case}. This is done in much the same way as in the proof of Lemma~\ref{lemma: sob H base integral} and so we omit the details.
	\end{proof}
	Combining the results of the previous lemmas with Theorem~\ref{theo: PointedSobSchauder}, a simple induction then yields the following result.
	
	\begin{proposition}\label{prop: sob H control}
		Fix $k > j \in \bZ_+$. Suppose that there exists a $\theta > 0$ such that for each compact set $\ck \subseteq \bR^d$ and each $n_1 \ge n_2$ we have that
		\begin{equ}
			\bE \left [ \|\bar \Pi^{n_1} - \bar \Pi^{n_2} \|_{V^{(j-1)}; \ck}^{p_k} + \|\bar \Gamma^{n_1} - \bar \Gamma^{n_2} \|_{\CT^j; \ck}^{p_k} \right ] \lesssim 2^{-{p_k}n_2 \theta}.
		\end{equ}
		Then there exists a $\bar \theta$ such that for the same set of parameters and for $l \le j$
		\begin{equ}
			\bE \left [ \sup_{\|\eta\|_\CH \le 1} \| \pHao{\tau_l} - \pHbo{\tau_l} \|_{\gamma_{\tau_l}, - \kappa, 2; \ck}^{p_{k-j}} \right ] \lesssim 2^{-p_{k-j} n_2 \bar \theta}.
		\end{equ}
	\end{proposition}
	
	With this result in hand, we are finally in a position to prove the convergence result that was promised in this section.
	
	\begin{theorem}\label{theo: bar BPHZ}
		Fix $k \in \bZ_+$. There exists $\theta, \varepsilon > 0$ such that for each compact set $\ck \subseteq \bR^d$, $1 \le j \le n_0$ and $n_1 \ge n_2$, we have the bound 
		\begin{equs}
			\bE \Bigg [\sum_{1 \le l \le j}  2^{-np_k(|\tau_l|_\fs^{(N-j)} - \varepsilon)}  \left |\bar{\Pi}_0^{n_1} \tau_l (\phi_0^{n}) - \bar{\Pi}_0^{n_2} \tau_l (\phi_0^{n}) \right |^{p_k} & + \|\bar{\Gamma}^{n_1} -  \bar{\Gamma}^{n_2} \|_{\CT^j}^{p_k} \Bigg ] \\ & \lesssim 2^{- p_k n_2 \theta} \label{eq: induction hyp 2}.
		\end{equs}
	\end{theorem}
	\begin{proof}
		We choose $N, N_1, N_2, k(j)$ as in the proof of Proposition~\ref{prop: uniform bounds} and we aim to prove by induction in $j$ that 
		\begin{equs}
			\bE \Bigg [\sum_{1 \le l \le j}  2^{-np_{k(j)}(|\tau_l|_\fs^{(N-j)} - \varepsilon)}  \left |\bar{\Pi}_0^{n_1} \tau_l (\phi_0^{n}) - \bar{\Pi}_0^{n_2} \tau_l (\phi_0^{n}) \right |^{p_{k(j)}} & + \|\bar{\Gamma}^{n_1} -  \bar{\Gamma}^{n_2} \|_{\CT^j}^{p_{k(j)}} \Bigg ] \\ & \lesssim 2^{- p_{k(j)} n_2 \theta} \label{eq: induction hyp 2 strong}
		\end{equs}
		The base case $j = 0$ is again trivial. 
		
		So long as $\varepsilon > 0$ is sufficiently small, by Lemmas~\ref{lemma: Pi Norm Control}, \ref{lemma: Gamma Norm Control} and~\ref{lemma: pos hom control} we can assume at the j-th inductive step that there exists a $\theta > 0$ such that for each compact set $\ck$ and $n_1 \ge n_2$
		\begin{equ}
			\bE \left [ \| \bar \Pi^{n_1} - \bar \Pi^{n_2} \|_{V^{(j-1)}; \ck}^{p_{k(j) - 3}} + \| \bar \Gamma^{n_1} - \bar \Gamma^{n_2} \|_{\CT^j; \ck}^{p_{k(j) - 3}} \right ] \lesssim 2^{-{p_{k(j) - 3}}n_2\theta}
		\end{equ}
		and that $|\tau_j|_{\fs}^{(N-j)} < 0$ without any loss of generality.
		
		We again apply the spectral gap inequality to write
		\begin{equs}
			& \bE \left [ \left |\bar{\Pi}_0^{n_1} \tau_j(\phi_0^n) - \bar{\Pi}_0^{n_2}(\phi_0^n) \right |^p \right ]^{1/p} \\ & \lesssim \bE \left[ \bar{\Pi}_0^{n_1} \tau_j(\phi_0^n) - \bar{\Pi}_0^{n_2}(\phi_0^n)  \right ] + \bE \left [ \left \| \frac{\partial \left [ \bar{\Pi}_0^{n_1} \tau_j (\phi_0^n) - \bar{\Pi}_0^{n_2} \tau_j (\phi_0^n) \right ]}{\partial \xi} \right \|_{H^{- \reg}(\Lab_-)}^p \right ]^{1/p}
		\end{equs}
		for $p = {p_{k(j+1)}}$. 
		
		We consider first the first term on the right hand side. As in the single model case, we write $\phi_0^{n} = \phi_0^{0} + \sum_{k=0}^{n-1} (\phi_0^{k+1} - \phi_0^{k})$. Since $$\bE \left [ \bar{\Pi}_0^{n_i} \tau_j (\phi_0^{0}) \right ] = 0$$ by the definition of the $\bBPHZ$ model, it remains to consider $$\bE \left [ \left (\bar{\Pi}_0^{n_1} \tau_j - \bar{\Pi}_0^{n_2} \tau_j \right ) (\phi_0^{k+1} - \phi_0^{k}) \right ] = \int \bE \left [ \left (\bar{\Pi}_0^{n_1} \tau_j - \bar{\Pi}_0^{n_2} \tau_j \right ) (\phi_y^{k+2}) \right ]R^k(y) dy.$$
		
		Using translation invariance as in the proof of Proposition~\ref{prop: uniform bounds}, this is nothing but
		\begin{equs}
			 \int \bE & \left [  \bar \Pi_y^{n_1} [ \bar \Gamma_{y0}^{n_1} \tau_j -\bar \Gamma_{y0}^{n_2} \tau_j] (\phi_y^{k+2}) \right ] R^k(y)
			\\ & + \bE \left [ \left( \bar{\Pi}_y^{n_1} - \bar{\Pi}_y^{n_2} \right) (\Gamma_{y0}^{n_2} \tau_j - \tau_j)(\phi_y^{k+2}) \right ]R^k(y) dy.
		\end{equs}
		Therefore, we obtain
		\begin{equs}
			& \left | \bE \left [ \left ( \bar{\Pi}_0^{n_1} \tau_j - \bar{\Pi}_0^{n_2} \tau_j \right ) (\phi_0^{k+1} - \phi_0^{k}) \right ]\right | \\ &  \lesssim 2^{-k |\tau_j|_\fs^{(N-j)}} \bE \left [ \|\bar{\Pi}^{n_1} \|_{V^{(j-1)}; \ck} \| \bar{\Gamma}^{n_1} -  \bar{\Gamma}^{n_2} \|_{\CT^{j}; \ck} + \| \bar \Gamma^{n_2} \|_{\cT^j; \ck} \| \bar{\Pi}^{n_1} - \bar{\Pi}^{n_2} \|_{V^{(j-1)}; \ck} \right ].
		\end{equs}
		By our assumption that $|\tau_j|_\fs^{(N-j)} < 0$ and an application of Cauchy-Schwarz and the uniform boundedness given by Proposition~\ref{prop: uniform bounds}, the right hand side has a sum over $k \le n$ of order $2^{-n |\tau_j|_\fs^{(N-j)}} 2^{-n_2 \theta}$ as desired.
		
		It remains to consider the second term on the right hand side of the spectral gap inequality. In the same way as in the proof of Proposition~\ref{prop: uniform bounds}, it suffices to bound
		\begin{equs}
			\bE \left [\sup_{\substack{\| \eta \|_{\CH} \le 1 \\ \eta \in \nice}} \left | \bar{\CR}^{n_1} \pHao{\tau_j}(\phi_0^{n}) - \bar \CR^{n_2} \pHbo{\tau_j} (\phi_0^n) \right |^p \right ]
		\end{equs}
		where here $\bar\CR^n$ is the reconstruction operator corresponding to the model $\bar \Pi^n$ (or in the case where $\gamma_{\tau_j} \le 0$ is our candidate for that reconstruction).
		
		In the case where $\gamma_{\tau_j} > 0$, the desired bound is immediate from Theorem~\ref{theo: PointedSobReconstr} and Proposition~\ref{prop: sob H control}. When $\gamma_{\tau_j} \le 0$ however, we need to do slightly more.
		
		In the case where $\tau_j = \Xi_\ft$ for $|\ft|_\fs \le - |\fs|/2$, the desired bound follows from Lemma~\ref{lemma: base sob estimate} and the embedding $\CB_{2, \infty}^{\reg \ft - \varepsilon} \hookrightarrow \CB_{\infty, \infty}^{|\ft|_\fs}$ for sufficiently small $\varepsilon > 0$.
		
		In the remaining case where $\tau_j = \Xi_\ft \cdot \tau_l$ for $\tau_l = \prod_{i = 1}^k \CI^{\fl_i} \sigma_i$, it is sufficient to consider only the term $f_0^{\tau_l} \pH{\Xi_\ft}$ since the other term in the definition is covered by Theorem~\ref{theo: PointedSobReconstr} by the first part of Lemma~\ref{lemma: pointed candidate exists}. 
		
		On the one hand, we have the bound
		\begin{equ}
			\bE \left [ \sup_{\|\eta\|_\CH \le 1} | \scal{(\varrho^{n_1} \ast \eta_\ft) \cdot \bar\Pi_x^{n_1} \tau_l - (\varrho^{n_2} \ast \eta_\ft) \cdot \bar \Pi_x^{n_2} \tau_l , \phi_0^n}|^p \right ] \lesssim 2^{-pn_2 \bar \theta}2^{-p n |\ft|_\fs}
		\end{equ}
		provided by Lemma~\ref{lemma: sob H bad case} and the usual embedding between Besov spaces.
		
		On the other hand, we have the bound
		\begin{equ}
			\bE \left [ \sup_{\|\eta\|_\CH \le 1} | \scal{(\varrho^{n_1} \ast \eta_\ft) \cdot \bar\Pi_x^{n_1} \tau_l - (\varrho^{n_2} \ast \eta_\ft) \cdot \bar \Pi_x^{n_2} \tau_l , \phi_0^n}|^p \right ] \lesssim 2^{-p n |\tau_j|_\fs}
		\end{equ}
		from Lemma~\ref{lemma: pointed candidate exists}.
		
		Interpolation between these bounds yields the desired result and completes the proof.
	\end{proof}
	
	\section{From the $\bBPHZ$ model to the BPHZ Model}\label{sec: moving between models}
	
	In this section, we will deduce the converge of the BPHZ renormalised models as in Definition~\ref{def: BPHZ} from the convergence of the $\bBPHZ$ renormalised models, which in turn follows from Theorem~\ref{theo: bar BPHZ} in combination with Theorem~\ref{theo: Kolmogorov Criterion}. The key observation is that by \cite[Remark 6.20]{BHZ}, the BPHZ renormalised model constructed from the canonical lift of $\xi_n$ coincides with the BPHZ renormalisation of the $\bBPHZ$ model constructed from $\xi_n$. 
	
	We recall that from \cite{BHZ}, we can write $\hat{\Pi}_x^n \tau = (g^n \otimes \bar \Pi_x^n) \Delta_\mathrm{ex}^- \tau$ where $g^n(\tau) = \bE[ \bar{\PPi}^n \tilde{\CA}_\mathrm{ex}^- \tau (0)]$ and the critical property of the map $\Delta_\mathrm{ex}^-$ is that for each $\tau$, if we write in Sweedler notation $\Delta_\mathrm{ex}^- \tau = \tau^{(1)} \otimes \tau^{(2)}$ then each $\tau^{(1)}$ is of negative degree and we have for $\tau^{(2)}$ that $|\tau^{(2)}|_\fs \ge |\tau|_\fs$. Here $\tilde{\CA}_{\text{ex}}^-$ is the negative twisted antipode on the extended structure introduced in \cite[Proposition 6.6]{BHZ}. This leads to the following observation.
	
	\begin{lemma}
		Suppose that $g^n$ defined above forms a pointwise convergent sequence of characters. Then the sequence $\hat Z^n = (\hat \Pi^n, \hat \Gamma^n)$ of BPHZ renormalised models is convergent in $L^p$ in $\bar{\CM}_\text{rand}(\CT)$ for each $p \in [2, \infty)$. 
	\end{lemma}
	\begin{proof}
		We write
		\begin{equ}
			\bE \left [ \| \hat{\Pi}^{n_1} - \hat{\Pi}^{n_2} \|_{\CT; \ck}^p \right ] = \bE \left [ \sup_{(\psi, \lambda, x, \tau)} \lambda^{-p |\tau|_\fs} \left |\left (\hat{\Pi}_x^{n_1} \tau - \hat{\Pi}_x^{n_2} \tau \right )(\psi_x^\lambda) \right |^p \right ]
		\end{equ}
		where the supremum is over $\psi \in \CB^r$, $\lambda \in [0,1)$, $x \in \ck$ and homogeneous $\tau \in \CT$ such that $\|\tau\| = 1$. This can be written as
		\begin{equs}
			\bE \left [ \sup_{(\psi, \lambda, x, \tau)} \lambda^{- p |\tau|_\fs}\left |\left (g^{n_1}(\tau^{(1)}) \bar{\Pi}_x^{n_1} \tau^{(2)} - g^{n_2}(\tau^{(2)})\bar{\Pi}_x^{n_2} \tau^{(2)} \right )(\psi_x^\lambda) \right |^p \right ]
		\end{equs}
		where $\Delta^- \tau = \tau^{(1)} \otimes \tau^{(2)}$. Since $|\tau^{(2)}|_\fs \ge |\tau|_\fs$, this expression then admits a bound by a constant multiple of the following term:
		\begin{equs}
			\sup_\tau \left | g^{n_1}(\tau^{(2)}) - g^{n_2}(\tau^{(1)}) \right |^{p} & \bE \left [ \| \bar{\Pi}^{n_1}\|_{\CT^j; \ck}^p \right ] \\ & + \sup_\tau |g^{n_2}(\tau^{(2)})|^p \bE \left [ \| \bar{\Pi}^{n_1} - \bar{\Pi}^{n_2} \|_{\CT^j ; \ck}^p \right ].
		\end{equs}
		Since our regularity structure is finite dimensional, pointwise convergence of $g^n$ implies that the supremum over $\tau$ in the first of these terms vanishes as $n_1, n_2 \to \infty$. Therefore, by Proposition~\ref{prop: uniform bounds} the first of these terms vanishes as $n_1, n_2 \to \infty$. For the second term, we note that the pointwise convergence of $g^n$ implies that the supremum is uniformly bounded so that this term vanishes also as $n_1, n_2 \to \infty$ by Theorem~\ref{theo: bar BPHZ}.
	\end{proof}
	As a result the remainder of this section is dedicated to establishing that we do indeed have pointwise convergence of this sequence of characters.
	
	We fix an arbitrary smooth and compactly supported function $\psi$ such that $\int \psi(y) dy = 1$. We recall from \cite[Section 6.3]{BHZ} that we have two natural ways in which translation by an element of $\bR^d$ acts on our space of models. The first is given by setting $T_h(\PPi)\tau(z) \eqdef \PPi \tau(z-h)$ and the second is given by setting $\tilde{T}_h(\PPi) \tau(z) \eqdef (\PPi \otimes g_h) \Delta_{\text{ex}}^+ \tau (z)$ where $\Delta_\text{ex}^+$ is defined in \cite[Corollary 5.32]{BHZ} and $g_h$ is defined by setting $g_h(X_i) = - h_i$ and $g_h(\CJ_k^\fl \tau) = 0$. In particular, the exact form of the operator $\Delta_{\text{ex}}^+$ will not be important to us since $g_h$ will vanish on all non-polynomial trees it extracts to the right so that the tree on the left will agree with $\tau$ except in its decorations and will have degree no more than the degree of $\tau$.
	
	For $|\tau|_\fs< 0$, we can then exploit the stationarity assumption in the form of \cite[Definition 6.17]{BHZ} (which is equivalent to our stationarity assumption) to write
	\begin{equs}
		g^n(\tau) & = \bE \left [ \bar \PPi^n \tilde{\CA}_\mathrm{ex}^- \tau (0) \right ]
		\\
		& = \int \bE \left [ T_y \bar \PPi^n \tilde{\CA}_\mathrm{ex}^- \tau(y) \right ] \psi(y) dy 
		\\
		& = \bE \left [ \bar \PPi^n \tau^{(1)} ( g_\cdot(\tau^{(2)}) \psi) \right ]
	\end{equs}
	where $\tilde{\CA}_\mathrm{ex}^-$ is the negative twisted antipode introduced in \cite[Section 6.1]{BHZ} and $\Delta_{\mathrm{ex}}^+ \tilde{\CA}_{\mathrm{ex}}^- \tau = \tau^{(1)} \otimes \tau^{(2)}$ in Sweedler notation.
	
	The desired result then follows from the following lemma since in our setting if $\ft \in \Lab_-$ and $|\ft|_\fs - |k|_\fs > 0$ then $f_x^n(\CJ_k^\tau 1)$ is a linear combination of terms of the form $D^j \rho^n \ast \xi_\ft(x)$ which gives a Cauchy sequence in the relevant $L^p$ topology by the spectral gap assumption.
	
			\begin{lemma}\label{lemma: f_x continuity}
		Suppose that $Z^n = (\Pi^n, \Gamma^n)$ is a sequence of models such that for each $\ft \in \Lab_-$ with $|\ft|_\fs \ge 0$, each $k \in \bN^d$ and each $x \in \bR^d$, $f_x^n(\CJ_k^\ft 1)$ is a Cauchy sequence (where $f_x^n$ is the character on $\CT_+$ corresponding to the model $Z^n$).
		
		Then for each $\tau \in \CT_+$ and $x \in \bR^d$, $f_x^n(\tau)$ is a convergent sequence. In particular, it follows that for each $\tau \in \CT$ and $\psi \in \CB^r$, the map $\PPi^n \tau(\psi)$ is a convergent sequence.
	\end{lemma}
	\begin{proof}
		We fix models $(\Pi, \Gamma), (\bar\Pi, \bar\Gamma)$.
		
		The result for $f_x$ is trivially true when $\tau$ is a monomial and is true by assumption for $\tau = \CJ_k^\ft \sigma$ for $\ft \in \Lab_-$. By multiplicativity, it then suffices to consider the case where $\tau = \CI_k^\fl \sigma$ for $l \in \Lab_+$ and $k \in \bN^d$ such that $|\sigma|_\fs + |\fl|_\fs - |k|_\fs > 0$. We have that
		\begin{equs}
			|f_x(\tau) - \bar{f}_x(\tau)| & \lesssim \sum_{n \ge 0} \sum_{|l|_\fs < |\CJ_k^\fl \sigma|_\fs} \left | (\Pi_x \sigma - \bar{\Pi}_x \sigma) (D^{k+l} K^\fl_n (\cdot - x)) \right | \\&  \lesssim  \sum_{n \ge 0} \|\Pi - \bar{\Pi} \|_{\CT; B_x} \cdot 2^{-n( |\sigma|_\fs + |\fl|_\fs - |k + l|_\fs)}.
		\end{equs}
		where $\bar{f}_x$ is the character corresponding to the model $(\bar\Pi, \bar\Gamma)$.
		
		Therefore since $|\sigma|_\fs + |\fl|_\fs - |k +l|_\fs > 0$, we can perform the sum in $n$ to conclude.
	\end{proof}
	\begin{remark}
		In the case where $\Lab_-$ contains only labels of negative degree, Lemma~\ref{lemma: f_x continuity} should be interpreted as saying that the map $(\Pi, \Gamma) \mapsto f_x(\tau)$ is automatically continuous for each fixed $x \in \bR^d$ and $\tau \in \CT_+$ which is essentially the contents of \cite[Proposition 6.31]{BHZ}. However that result does not correctly handle the case of noises of positive degree and so we provide a separate statement here.
	\end{remark}

	\appendix
	
	\section{Local Besov Spaces}\label{appendix: Besov Spaces}
	
	In this appendix, we collect some results regarding the local Besov spaces defined in Section~\ref{s: notation}. The most important of these will be the characterisation of these spaces via a distinguished kernel with a semigroup type property.
	
	The other properties ought not surprise any reader familiar with the theory of Besov spaces defined by global behaviour. We include these results only because the Fourier analytic tools often used in the treatment of such spaces don't adapt well to our local setting and so it is necessary to affirm that this does not pose difficulties in obtaining the usual results.
	
	We will begin by introducing our alternative characterisation of the spaces $\CB_{p,q}^\alpha$. As an intermediary step, we note that the $L_\lambda^q$ norm in the definition can be replaced by an $\ell^q$ type norm. 
	
	To see this, given $\lambda \in (0,1]$, let $n(\lambda) \eqdef \max\{n: 2^{-n} \ge \lambda\}$. Then for $\lambda \in (0,1]$, $\eta \in \CB^r$ there exists a constant $C > 0$ (which is uniformly bounded in the choice of $\lambda$ and $\eta$) and a test function $\psi \in \CB^r$ such that $\eta^\lambda = C \psi^{2^{-n(\lambda)}}$. Consequently, in the case where $\alpha < 0$, an equivalent family of seminorms on $\CB_{p,q}^\alpha$ is given by 
	\begin{equs}
		\$\xi\$_{\CB_{p,q}^\alpha; \ck} \eqdef \left \| \left \| \sup_{\eta \in \CB^r} \frac{|\scal{\xi, \eta_x^{2^{-n}}}|}{2^{-n \alpha}}\right \|_{L^p(\ck; dx)} \right \|_{\ell^q(n)}.
	\end{equs}
	
	The analogous result is true in the case $\alpha \ge 0$ by replacing $\CB^r$ with $\CB_{\lfloor \alpha \rfloor}^r$ in the above discussion.
	
	With this step in hand, we introduce the convolution kernel we will use to characterise these spaces. Our definition is taken from \cite{Book}, though we refer the reader also to \cite[Section 2]{MW18a} for a similar construction.
	
	\begin{definition}\label{def: semigroup kernel}
		Fix an even, smooth function $\rho: \bR^d \to \bR$ that is supported in the $\fs$-scaled ball of radius $\frac{1}{2}$ such that $\int \rho(x) x^k dx = \delta_{k, 0}$ for $0 \le |k|_\fs \le r$. We let $\rho^{n}(x) \eqdef 2^{n |\fs|} \rho(2^{n \fs} x)$ and set for $m > n$
		\begin{equ}
			\rho^{(n,m)} = \rho^{n} \ast \rho^{n+1} \ast \dots \ast \rho^{m-1} \ast \rho^{m}.
		\end{equ}
		Define $\varphi^{n} = \lim_{m \to \infty} \rho^{(n,m)} \in \CC_c^\infty$ where the limit converges in $\CC^a$ for all $a > 0$.
	\end{definition}
	
	It is straightforward to verify that $\varphi^{n} (x) = 2^{n |\fs|} \varphi^{0}(2^{n \fs} x)$ by verifying the same relation for $\rho^{(n,m)}$ and passing to the limit. In particular, the superscript notation here coincides with the use of that notation for rescaling a test function. We also note that this choice of kernel has a convolution semigroup type property; namely that 
	\begin{equ}[e:convolProp]
	\phi^{n} = \rho^n \ast \phi^{n+1}\;.
	\end{equ}
	
	For $\zeta \in \CD'(\bR^d)$, we define $\zeta_n \eqdef \zeta \ast \varphi^{n}$. 
 	We then define the following collections of seminorms on $\CB_{p,q}^\alpha$.
	
	\begin{definition}
		 For $\alpha < 0$, compact $\ck \subseteq \bR^d$ and $p,q \in [1,\infty]$, we define 
		\begin{equ}
			|\zeta|_{\CB_{p,q}^\alpha ; \ck} \eqdef \left \| 2^{n \alpha} \left \| \zeta_n \right \|_{L^p(\ck)} \right \|_{\ell^q(n)} < \infty.
		\end{equ}
		
		For $\alpha \ge 0$, we let $|\zeta|_{\CB_{p,q}^\alpha; \ck}$ denote the quantity
		\begin{equ}
			 \|\zeta_0\|_{L^p(\ck)} + \left \| 2^{n \alpha} \left \| \sup_{\eta \in \CB_{\lfloor \alpha \rfloor}^r} |\scal{\zeta_n, \eta_x^n}| \right \|_{L^p(\ck)} + 2^{n \alpha} \| \zeta_{n+1} - \zeta_n \|_{L^p(\ck)}\right \|_{\ell^q(n)}.
		\end{equ} 
	\end{definition}

	We then have the following characterisation of $\CB_{p,q}^\alpha$.
	
	\begin{theorem}\label{theo: kernel swapping}
		Fix $p, q \in [1, \infty]$ and $\alpha \in \bR$. For $\zeta \in \CD'(\bR^d)$ and for compact $\ck \subseteq \bR^d$ we have that
		\begin{equ}
			\| \zeta\|_{\CB_{p,q}^{\alpha}(\ck)} \lesssim |\zeta|_{\CB_{p,q}^\alpha; \bar{\ck}}
		\end{equ}
		where $\bar{\ck}$ denotes the 1-fattening of $\ck$.
		
		In particular, $\CB_{p,q}^\alpha$ is the set of distributions $\zeta$ such that for each compact set $\ck$, $|\zeta|_{\CB_{p,q}^\alpha; \ck} < \infty$ and the family of seminorms $|\cdot|_{\CB_{p,q}^\alpha; \ck}$ induces the same topology.
	\end{theorem}
	\begin{proof}
		The case $\alpha < 0$ is a special case of \cite[Proposition A.5]{BL21}. The case $\alpha \ge 0$ follows by essentially the same ideas so we omit the details.
	\end{proof}
	
	\section{A Kolmogorov Criterion for Models}\label{appendix: Kolmogorov}
	
	In order to obtain boundedness or convergence of sequences of models on structures of decorated trees, the standard tool in the literature is the Kolmogorov criterion provided in \cite[Theorem 10.7]{Hai14}. In this section of the appendix, we will slightly modify the proof of that result. 
	
	We do this for two reasons. Firstly, the proof appearing in \cite{Hai14} uses in a crucial way techniques of wavelet analysis that we would like to remove in order to permit a more general choice of scaling. Secondly, the statement and proof in \cite{Hai14} do not allow the presence of noises of positive degree and so we must make the necessary adaptations to allow for those.
	
	The main result of this section is then the following.
	
	\begin{theorem}\label{theo: Kolmogorov Criterion}
		Let $V$ be a sector of $\CT$ and suppose there is a $\bar\kappa > 0$ such that for every tree $\tau$ which is either of negative degree or is a noise of positive degree, for every $n \geq 0$ and for every $p \in [1, \infty)$ that the bound 
		
		\begin{equ}\label{eq: Kolmogorov single point ass}
			\bE[|\Pi_0 \tau (\phi_0^{n})|^p] \lesssim 2^{-pn|\tau|_\fs - p \bar\kappa n}
		\end{equ}
		holds. Then for compact sets $\ck \subseteq \bR^d$, $\bE[\|Z\|_{V; \ck}^p] \lesssim 1$.
		
		If additionally the bound 
		\begin{equ}\label{eq: Kolmogorov multi point ass}
			\bE[|\Pi_0 \tau (\phi_0^{n}) - \bar{\Pi}_0 \tau (\phi_0^{n})|^p] \lesssim \varepsilon^{p \theta} 2^{-pn|\tau|_\fs - p \bar\kappa n}
		\end{equ}
		holds then $\bE[\|Z; \bar{Z}\|_{V; \ck}^p] \lesssim \varepsilon^{p \theta}$.
	\end{theorem}
	
	Our intention is to follow the structure of the inductive proof provided in \cite[Theorem 10.7]{Hai14} and make amendments only where necessary. In particular, in the interest of brevity, we freely borrow notation from that proof.
	
	The most significant difficulty not present in the setting of \cite{Hai14} is that since we do not use wavelet techniques, we no longer have a characterisation of the norm of a model in terms of its values at countably many base points. Instead, the best we can hope for is the following result which should be compared with \cite[Proposition 3.32]{Hai14} and whose proof is a corollary of the proof of the Reconstruction Theorem given in \cite{BL21} in the same way that \cite[Proposition 3.32]{Hai14} is a corollary of the proof of the Reconstruction Theorem given there.
	
	\begin{proposition}\label{prop: model norm characterisation}
		For every compact set $\ck \subseteq \bR^d$ and every sector $V$, the following bounds hold.
		\begin{equs}
			\|\Pi\|_{V; \ck} &\lesssim  (1 + \|\Gamma\|_{V; \ck}) \sup_{\alpha \in A_V} \sup_{a \in V_\alpha} \sup_{n\geq 0} \sup_{x \in \bar{\ck}} 2^{\alpha n} \frac{|\Pi_x a (\phi_x^{n})|}{\|a\|}
			\\ 
			\|\Pi - \bar{\Pi}\|_{V; \ck} &\lesssim  (1 + \|\Gamma\|_{V; \ck}) \sup_{\alpha \in A_V} \sup_{a \in V_\alpha} \sup_{n\geq 0} \sup_{x \in \bar{\ck}} 2^{\alpha n} \frac{|(\Pi_x a - \bar{\Pi}_x a) (\phi_x^{n}) |}{\|a\|}
		\end{equs}
	\end{proposition}

	This means that pulling the resulting suprema over base points $x$ outside of expectations will take more work. To achieve this, we will apply the classical Kolmogorov criterion. Whilst the result and its proof are essentially standard and found in many textbooks on stochastic analysis, often the explicit bounds we require are not part of the statement (though they are a corollary of a careful reading of the proofs given) and so we choose to provide an explicit statement here for the reader's convenience.

	\begin{proposition}\label{prop: Kolmogorov}
		Let $\ck \subset \bR^d$ be a compact, let $\varepsilon \in (0,1]$, $K \ge 1$, $p \ge 1$, and $\delta > 0$
		be such that $\delta p > |\fs|$, and let $\{F_x \}_{x \in \ck}$ be a collection of random
		variables such that
		\begin{equ}[e:assKolmogorov]
			\sup_{x \in \ck} \bE [|F_x|^p] \lesssim 1\;,\qquad
			\sup_{x,y \in \ck \atop |x-y|_\fs \le \varepsilon} \|x-y\|_\fs^{-\delta} \bE \left [ |F_x - F_y|^p \right]^{1/p} \lesssim K^\delta\;.
		\end{equ}
		Then, there exists a $\CC(\ck)$-valued random variable $F$ such that, for every $x \in \ck$, 
		$F(x) = F_x$ almost surely, and such that 
		\begin{equ}
			\bE \left [\sup_{x\in \ck}|F(x)|^p \right ]^{1/p}\le C \big(\varepsilon^{-|\fs|/p} + K^\delta \varepsilon^{\delta - |\fs|/p}\big)\;,
		\end{equ}
		for some fixed constant $C$ which depends only on the diameter of $\ck$ and the implicit constants in \eqref{e:assKolmogorov}.
	\end{proposition}
	
	In order to verify the hypotheses of this criterion, it will be useful to be able to apply bounds of the type given in \eqref{eq: Kolmogorov single point ass} and \eqref{eq: Kolmogorov multi point ass} with the test function $\phi$ replaced with an arbitrary $\psi \in \CB^r$. To this end, defining $V_\tau$ to be the smallest sector containing $\tau$ and $\mathring{V}_\tau$ to be the largest subsector of $V_\tau$ that does not contain $\tau$, we have the following result.
	
	\begin{lemma}\label{lemma: Kolmogorov Ass Upgrade}
		Let  $\tau \in \CT$ with $\|\tau\| = 1$ and suppose that for $p \in [1, \infty)$ and every $n \ge 0$ we have that
		\begin{equ}
		\bE[ |\Pi_0 \tau (\phi_0^n) |^p]^{1/p} \lesssim 2^{-n |\tau|_\fs}
		\end{equ}
		and that $$\bE[ \|\Gamma\|_{B_0; V_\tau}^{2p}]^{1/2p} + \bE [ \|\Pi\|_{B_0; \mathring{V}_\tau}^{2p}]^{1/2p} \lesssim 1.$$ Then $\sup_{\psi \in \CB^r} \bE[|\Pi_0 \tau (\psi_0^{n})|^p]^{1/p} \lesssim 2^{-n|\tau|_\fs}.$
		
		If additionally we suppose that for some $\theta > 0$
		\begin{equ}
			\bE[ |\Pi_0\tau(\phi_0^n) - \bar{\Pi}_0 \tau (\phi_0^n) |^p]^{1/p} \lesssim \varepsilon^{ \theta}2^{-n |\tau|_\fs}
		\end{equ} and that 
		\begin{equ}
			\bE [ \| \Gamma - \bar{\Gamma} \|_{B_0; V_\tau}^{2p}]^{1/2p} + \bE[ \|\Pi - \bar\Pi\|_{B_0; \mathring{V}_\tau}^{2p}]^{1/2p} \lesssim \varepsilon^{\theta}
		\end{equ} then we have that
		\begin{equ}
			\sup_{\psi \in \CB^r} \bE[|\Pi_0 \tau (\psi_0^{n}) - \bar{\Pi}_0 \tau (\psi_0^{n})|^p] \lesssim \varepsilon^{p \theta} 2^{-pn|\tau|_\fs}.
		\end{equ}
	\end{lemma}
	\begin{proof}
		As usual, we consider only the case of a single model here for brevity.
		
		We can write
		\begin{equs}
			\left \| \Pi_x \tau(\psi_x^{n}) \right \|_{L^p} &= \left \| \Pi_x \tau (\psi_x^{n} \ast \phi^{n}) + \sum_{k \geq 0} \Pi_x \tau \left ( \psi_x^{n} \ast (\phi^{n+k+1} - \phi^{n+k}) \right ) \right \|_{L^p}
			\\
			& \leq \left \| \Pi_x \tau (\psi_x^{n} \ast \phi^{n}) \right \|_{L^p} + \sum_{k \geq 0} \left \| \Pi_x \tau \left ( \psi_x^{n} \ast (\phi^{n+k+1} - \phi^{n+k}) \right ) \right \|_{L^p} \\ \label{eq: Kolmogorov 1}
		\end{equs}
		where the $L^p$-norms are with respect to the underlying probability measure.
		
		We consider first an arbitrary term in the sum on the right hand side. We can write the $k$-th term in this sum as
		\begin{equs}\label{eq: Kolmogorov 2}
			\left \| \int \bigl(\Pi_x \tau\bigr)(\phi_y^{n+k}) \left (\psi^{n}_x - \rho^{n+k} \ast \psi^{n}_x \right )(y) \,dy \right \|_{L^p}.
		\end{equs}
		Taylor expanding $\psi^{n}$ to order $N$ and arguing as in the paragraph following \cite[equation (13.20)]{Book}, we have that 
		\begin{equ}
			\left |\left (\psi^{n}_x - \rho^{n+k} \ast \psi^{n}_x \right) (y) \right | \lesssim 2^{n|\fs| - k N}.
		\end{equ}
		Additionally, we note that the integral in $y$ has support in a ball of radius of order $2^{-n\fs}$. 
		
		Therefore, inserting this bound and applying Jensen's inequality, we find that the term \eqref{eq: Kolmogorov 2} can be bounded by
		\begin{equs}
			2^{-kN} \left ( \int_{B_\fs(x, c2^{-n})} 2^{n|\fs|}\bE\left [ \left |\Pi_x \tau \left (\phi_y^{(n+k)}\right)\right|^p \right] dy \right)^{1/p}.
		\end{equs}
		
		\sloppy By writing $\Pi_x \tau \left ( \phi_y^{(n+k)} \right ) = \Pi_y \Gamma_{yx}\tau \left (\phi_y^{(n+k)} \right )$, we obtain a bound of order $2^{-Nk - (n+k) |\tau|_\fs}$ which has a sum in $k$ of the correct order so long as we select $N$ such that $N > - |\tau|_\fs$.
		
		Control on the first term of \eqref{eq: Kolmogorov 1} then follows similarly.
	\end{proof}

	With this result in hand, we obtain the following corollary of Proposition~\ref{prop: Kolmogorov}, which is a specialisation of that result to the case where $F_x = 2^{n |\tau|_\fs} \Pi_x \tau(\phi_x^n)$.
	
	\begin{corollary}\label{cor: Kolmogorov}
		Under the hypotheses of Lemma~\ref{lemma: Kolmogorov Ass Upgrade} we have that
		$$\bE \Big[ \sup_{x \in \ck} |\Pi_x \tau(\phi_x^n)|^p\Big]^{1/p} \lesssim 2^{-n |\tau|_\fs}$$ and that $$\bE\Big[\sup_{x \in \ck} |\Pi_x \tau(\phi_x^n) - \bar\Pi_x \tau (\phi_x^n)|^p\Big]^{1/p} \lesssim \varepsilon^\theta 2^{-n |\tau|_\fs}$$
	\end{corollary}
	\begin{proof}
		As usual, we illustrate only the case of a single model.
		
		The result will follow if we verify the hypotheses of Proposition~\ref{prop: Kolmogorov} for $F_x$ as above with $K = 2^n$ and $\varepsilon = 2^{-n}$. Since the first bound in \eqref{e:assKolmogorov} is automatic from our assumptions, we consider only the second required bound.
		
		For this, we write $$|\Pi_x \tau(\phi_x^n) - \Pi_y \tau(\phi_y^n)| \le |\Pi_x [\tau - \Gamma_{xy}\tau] (\phi_x^n)| + |\Pi_y \tau (\phi_x^n - \phi_y^n)|.$$ Control on the terms resulting from the first part of the right hand side are straightforward by assumption and for the second term on the right hand side control follows from Lemma~\ref{lemma: Kolmogorov Ass Upgrade} since for $|x-y| \le 2^{-n}$ we have that $|y-x|^{-1}(\phi_x^n - \phi_y^n)$ is a test function at scale $2^{-n}$ centred at $y$.
	\end{proof}
	With these preparatory results in place, the proof of Theorem~\ref{theo: Kolmogorov Criterion} is reduced to a fiddly set-up involving a sequence of shifted degree assignments. Since this holds no new conceptual ideas, we provide only a sketch proof.
	
	\begin{proof}[Sketch of Proof of Theorem~\ref{theo: Kolmogorov Criterion}]
		As usual, the proof in the case of two models is similar to the case of one model so we will deal only with the latter of these cases.
		
		We define $T_n, T_n^+$ and $T_n^-$ as in the proof of \cite[Theorem 10.7]{Hai14}. We let $N$ be such that $T_n = \CT$ and we equip $T_n$ (and hence $T_n^\pm$) with the degree assignment $|\tau|_\fs^{(N-n)}$ where $|\cdot|_\fs^{(j)}$ is as in \eqref{eq: shifted hom assignments} and $\kappa > 0$ appearing there is additionally chosen to be small enough so that for every $\tau \in \CT$, $|\tau|_\fs + \bar \kappa >  |\tau|_\fs^{(N)}$.
	
		We note that by Proposition~\ref{prop: model norm characterisation}
		\begin{equs}
			\bE & \Bigg [ \| \Pi \|_{T_{n+1}^-; \ck}^p \Bigg ]^{1/p} \\ & \lesssim \bE \Bigg [ ( 1+ \|\Gamma\|_{T_{n+1}^-; \ck})^{2p}\Bigg]^{1/2p} \sum_{n \ge 0} \sum_{i = 1}^k 2^{n |\tau_i|_\fs^{(N-n-1)}} \bE \Bigg [ \sup_{x \in \ck} |\Pi_x\tau_i(\phi_x^n)|^{2p} \Bigg ]^{1/2p}
		\end{equs}
		where $\{\tau_i : i = 1, \dots, k\}$ is a basis of $T_{n+1}^-$.
		
		Applying Corollary~\ref{cor: Kolmogorov} to control the right hand side yields the bound given in \cite[Equation (10.4)]{Hai14} without the use of wavelets. Since the rest of the proof given there is wavelet free, the only remaining detail is to show that the presence of noises of positive degree does not cause any issues.
		
		The reason that these require extra accommodation is that if $|\fl|_\fs > 0$ then $\Gamma_{xy} \Xi_\fl$ is not automatically controlled by the action of the model on symbols earlier in the induction via an application of \cite[Theorem 5.14]{Hai14}.
		
		However, since for $\tau = \Xi_\fl$ we have that $\mathring{V}_\tau \subset \Tpoly$, the bound \eqref{eq: Kolmogorov single point ass} can be written as
		\begin{equs}
			\bE \Bigg [ |\scal{ f - P_x^{|\fl|_\fs} f, \phi_x^n}|^p \Bigg]^{1/p} \lesssim 2^{-n |\tau|_\fs - n \bar \kappa}
		\end{equs}
		where $P_x^\gamma f$ is the Taylor jet of $f$ around $x$ to order $\gamma$ and here $f$ is any function such that $\PPi \Xi_\fl = f$ induces the model $(\Pi, \Gamma)$ on $V_\tau$. 
		
		We then note that if $F_x \eqdef 2^{n |\fl|_\fs} \scal{f - P_x^{|\fl|_\fs} f, \phi_x^n}$, we have the bound
		\begin{equs}
			& 2^{-n |\fl|_\fs}\| F_x - F_y \|_{L^p} \\ & \le \| \scal{f - P_x^{|\fl|_\fs}, \phi_x^n - \phi_y^n} \|_{L^p} + \| \scal{f - P_x^{|\fl|_\fs}, \phi_y^j} \|_{L^p} + \| \scal{f - P_x^{|\fl|_\fs}, \phi_y^j}\|_{L^p}
		\end{equs}
		where $|x-y| \sim 2^{-j}$ and we have used the fact that convolution with $\phi^k$ fixes polynomials of appropriate degree to change the scale of the test function in the latter two terms.
		
		It then follows from an application of Proposition~\ref{prop: Kolmogorov} that $$\bE \Bigg [ \sup_{n \ge 0} \sup_{x \in \ck}  2^{-n |\fl|_\fs}|\scal{ f - P_x^{|\fl|_\fs} f, \phi_x^n}|^p \Bigg ]^{1/p} \lesssim 1$$
		so long as $p$ is sufficiently large. 
		
		In turn, this implies that $\bE [ [D^k f]_{|\fl|_\fs - |k|_\fs; \ck}^p ]^{1/p} \lesssim 1$ where $k$ is such that for any non-zero multi-index $j$, $|j + k|_\fs \ge |\fl|_\fs$ and $[f]_{\alpha; \ck}$ denotes the usual $\alpha$-H\"older seminorm on $\ck$.
		
		Since the $X^j$ component of $\Gamma_{xy} \Xi_\fl$ is nothing but $D^j f(x) - (P_y^{|\fl|_\fs} D^j f)(x)$, the desired control on $\Gamma_{xy} \Xi_\fl$ then follows from an application of Taylor's theorem.
	\end{proof}

	\section{Properties of $\Delta^M$}
	
	We now gather some elementary properties of the group $\fR$ that will be useful in the sequel. For notational convenience, in these results and in their proofs we will suppress the dependence of integration maps on the decorations $\ft \in \Lab$.
		
		\begin{lemma}\label{Lemma: Integration Coproduct}
				Let $\Delta^M \tau = \tau^{(1)} \otimes \tau^{(2)}$ (adopting Sweedler's notation for coproducts). Then we have that
				$$\Delta^M \CI \tau = (\CI \otimes 1) \Delta^M \tau - \sum_{|k| > |\tau| + \beta} \frac{X^k}{k!} \otimes \CJ_{k}(\tau^{(1)}) \tau^{(2)}.$$
			\end{lemma}
		
		\begin{proof}
				Let $\CM$ denote the multiplication operator on $\CT_+$ and let $\Delta: \CT \to \CT \otimes \CT_+$ be as in \cite[Section 8.1]{Hai14}. Define $D = (1 \otimes \CM)(\Delta \otimes 1)$ and recall that by the proof of \cite[Proposition 8.36]{Hai14}, we have that 
				$$\Delta^M \tau = D^{-1} (M \otimes \hat{M}) \Delta \tau$$
				where we observed that $D$ is invertible since it can be written as $D = 1 - \bar{D}$ for a nilpotent map $\bar{D}$.
				
				By the definition of $\Delta$, we can write
				\begin{equs}
						(M \otimes \hat{M}) \Delta \CI\tau & = (\CI M \otimes \hat{M}) \Delta \tau + \sum_{|k+l|_\fs < |\tau|_\fs + \beta} \frac{X^k}{k!} \otimes \frac{X^l}{l!} \hat{M} \CJ_{k+l} \tau
						\\ & = (\CI \otimes 1) D \Delta^M \tau + \sum_{|k+l|_\fs < |\tau|_\fs + \beta} \frac{X^k}{k!} \otimes \frac{X^l}{l!} \CJ_{k+l}(\tau^{(1)}) \tau^{(2)}.
					\end{equs}
				
				On the other hand,
				\begin{equs}
						& D (\CI \otimes 1)\Delta^M \tau - \sum_{|k| > |\tau|+\beta} D\left (  \frac{X^k}{k!} \otimes \CJ_k(\tau^{(1)}) \tau^{(2)} \right ) \\ & = (\CI \otimes 1) D \Delta^M \tau + \sum_{|k+l|< |\tau^{(1)}| + \beta} \frac{X^k}{k!} \otimes \frac{X^l}{l!} \CJ_{k+l}(\tau^{(1)}) \tau^{(2)} \\
						& \quad - \sum_{|k| > |\tau|+\beta} D\left (  \frac{X^k}{k!} \otimes \CJ_k(\tau^{(1)}) \tau^{(2)} \right )	
						\\ & = 	(\CI \otimes 1) D \Delta^M \tau + \sum_{|k+l|< |\tau^{(1)}| + \beta} \frac{X^k}{k!} \otimes \frac{X^l}{l!} \CJ_{k+l}(\tau^{(1)}) \tau^{(2)} \\
						& - \sum_{|k+l|> |\tau| + \beta} \frac{X^k}{k!} \otimes \frac{X^l}{l!} \CJ_{k+l}(\tau^{(1)}) \tau^{(2)}
					\end{equs}
				where the last equality follows by definition of $\Delta$ on polynomial terms.
				
				Since $|\tau^{(1)}|_\fs > |\tau|_\fs$ by construction of $\Delta^M$, this shows that the expressions for $D$ applied to both the left and right hand side of our desired inequality coincide. Since $D$ is invertible this concludes the proof.
			\end{proof}
		
		\begin{lemma}\label{Lemma: Multiplication Coproduct}
				Suppose that whenever $\tau, \tau_1, \tau_2 \in \CT$ with $\tau = \tau_1 \tau_2$ we have that $M \tau = M \tau_1 M \tau_2$. Then $\Delta^M \tau = \Delta^M \tau_1 \Delta^M \tau_2$.
			\end{lemma}
		\begin{proof}
				This is immediate from the fact that $\Delta^M = D^{-1} (M \otimes \hat{M}) \Delta$ where $D, \hat{M}$ and $\Delta$ are all multiplicative.
			\end{proof}
		
		\begin{lemma}\label{Lemma: Commutativity Coproduct}
				Suppose that $M_1, M_2 \in \mathfrak{R}$ with $M_1 M_2 = M_2 M_1$. Then $(\Delta^{M_1} \otimes \hat{M}_1) \Delta^{M_2} = (\Delta^{M_2} \otimes \hat{M}_2) \Delta^{M_1}$.
			\end{lemma}
		\begin{proof}
				By the proof of \cite[Lemma 8.43]{Hai14} both sides coincide with $\Delta^M$ for $M = M_1 M_2 = M_2 M_1$.
			\end{proof}
\endappendix 

	\bibliographystyle{Martin}
	\bibliography{bib}

\end{document}